\documentclass[10pt]{article}
\usepackage[british]{babel}
\usepackage{amssymb,amsmath,amsthm}
\usepackage{cite}
\usepackage{graphicx}
\usepackage[bookmarks=true, bookmarksopen=true, bookmarksopenlevel=4]{hyperref}
\usepackage{color}
\usepackage{enumerate}%needed for roman enumerate with labels
\usepackage{pgf,tikz}\usetikzlibrary{matrix, calc, arrows}
\usepackage{soul}
\usepackage{subfig} % Euan added 05/07/19
\usetikzlibrary{decorations.pathreplacing,matrix,calc,arrows,shapes.geometric}

\hypersetup{pdftitle={Explicit bounds for the time-harmonic Maxwell equations in heterogeneous media},
  pdfauthor={T. Chaumont-Frelet, A. Moiola, E.A. Spence},
  colorlinks=true, linkcolor=myblue,  citecolor=myblue, filecolor=myblue,   urlcolor=myblue,}

\usepackage[a4paper]{geometry}
\allowdisplaybreaks[4]
\allowdisplaybreaks[4]
\definecolor{myblue}{rgb}{0,0,0.6}         % used for links in hyperref
\definecolor{gray}{rgb}{0.5,0.5,0.5}

\definecolor{amcol}{rgb}{0.8,0,0}

\definecolor{escol}{rgb}{0,0,0.8}
\definecolor{estcol}{rgb}{0,0.8,0}

\definecolor{tcfcol}{RGB}{251,94,5}

\newcommand{\di}{\,\mathrm{d}}                              % for integrals

\newcommand{\oon}{\;\text{on}\;}
\newcommand{\deO}{{\partial\Omega}}
\newcommand*{\conj}[1]{\overline{#1}}
\newcommand*{\N}[1]{\left\|#1\right\|}
\newcommand*{\abs}[1]{\left|#1\right|}
\newcommand{\Tnorm}[1]{|||#1|||}
\newcommand{\curl} {\mathop{\rm curl}\nolimits}
\newcommand{\dive} {\mathop{\rm div}\nolimits}
\def\div{\mathop{\rm div}\nolimits}

\DeclareMathOperator{\esssup}{ess\,sup}
\DeclareMathOperator{\essinf}{ess\,inf}

\DeclareMathOperator{\dist}{dist} %

\newcommand{\Uu}[1]{{\mathbf{#1}}}                   % bf which scales size for latin letters

\newcommand{\IC}{\mathbb{C}}

\newcommand{\IN}{\mathbb{N}}
\newcommand{\IR}{\mathbb{R}}

\newcommand{\ba}{{\Uu a}}\newcommand{\bb}{{\Uu b}}
\newcommand{\bd}{{\Uu d}}\newcommand{\be}{{\Uu e}}%\newcommand{\bf}{{\Uu f}}
\newcommand{\bg}{{\Uu g}}

\newcommand{\bn}{{\Uu n}}

\newcommand{\bv}{{\Uu v}}\newcommand{\bw}{{\Uu w}}\newcommand{\bx}{{\Uu x}}
\newcommand{\by}{{\Uu y}}\newcommand{\bz}{{\Uu z}}
\newcommand{\bA}{{\Uu A}}
\newcommand{\bE}{{\Uu E}}\newcommand{\bF}{{\Uu F}}
\newcommand{\bH}{{\Uu H}}
\newcommand{\bJ}{{\Uu J}}\newcommand{\bK}{{\Uu K}}

\newcommand{\bQ}{{\Uu Q}}

\newcommand{\bZ}{{\Uu Z}}

           \newcommand{\bzero}{\Uu{0}}

\newcommand{\calO}{{\mathcal O}}

\newcommand{\calT}{{\mathcal T}}

\newtheorem{theorem}{Theorem}[section]
\newtheorem{lemma}[theorem]{Lemma}
\newtheorem{defin}[theorem]{Definition}

\newtheorem{proposition}[theorem]{Proposition}
\newtheorem{cor}[theorem]{Corollary}

\newtheorem{remark}[theorem]{Remark}

 \newtheorem{example}[theorem]{Example}
\newcommand{\supp}{\operatorname{supp}}
\newcommand{\re}{{\rm e}}
\newcommand{\ri}{{\rm i}}
\newcommand{\rd}{{\rm d}}
\newcommand{\beq}{\begin{equation}}      \newcommand{\eeq}{\end{equation}}
\newcommand{\beqs}{\begin{equation*}}    \newcommand{\eeqs}{\end{equation*}}
\newcommand{\bit}{\begin{itemize}}       \newcommand{\eit}{\end{itemize}}
\newcommand{\ben}{\begin{enumerate}}     \newcommand{\een}{\end{enumerate}}
\newcommand{\bal}{\begin{align}}         \newcommand{\eal}{\end{align}}
\newcommand{\bals}{\begin{align*}}       \newcommand{\eals}{\end{align*}}
\newcommand{\bse}{\begin{subequations}}	 \newcommand{\ese}{\end{subequations}}
\newcommand{\bpr}{\begin{proposition}}   \newcommand{\epr}{\end{proposition}}
\newcommand{\bre}{\begin{remark}}        \newcommand{\ere}{\end{remark}}
\newcommand{\bpf}{\begin{proof}}         \newcommand{\epf}{\end{proof}}
\newcommand{\ble}{\begin{lemma}}         \newcommand{\ele}{\end{lemma}}
\newcommand{\bco}{\begin{corollary}}     \newcommand{\eco}{\end{corollary}}
\newcommand{\bex}{\begin{example}}       \newcommand{\eex}{\end{example}}
\newcommand{\bth}{\begin{theorem}}       \newcommand{\enth}{\end{theorem}}
\newcommand{\Rea}{\mathbb{R}}            \newcommand{\Com}{\mathbb{C}}

\newcommand{\eps}{\varepsilon}

\newcommand{\pdiff}[2]{\frac{\partial #1}{\partial #2}}

\newcommand{\gu}{\nabla u}

\newcommand{\gv}{\nabla v}

\newcommand{\half}{\frac{1}{2}}

\newcommand{\tendi}{\rightarrow \infty}
\newcommand{\tendo}{\rightarrow 0}
\newcommand{\DOmegabar}{C^\infty(\overline{D})}

\newcommand{\epz}{{\epsilon_0}}
\newcommand{\muz}{{\mu_0}}
\newcommand{\hn}{{\hat\bn}}
\newcommand{\hx}{{\widehat\bx}}

\newcommand{\deB}{{\partial B}}
\newcommand{\deOimp}{{\partial \domainimp}}

\newcommand{\domain}{{\Omega}}

\newcommand{\loc}{_{\mathrm{loc}}}
\newcommand{\Rimp}{R_{\domainimp}}%{R_{\chi}}}

\newcommand{\mythmname}[1]{\textbf{\emph{(#1.)}}}
\newcommand{\cO}{{\mathcal O}}
\newcommand{\cD}{{\mathcal D}}

\newcommand{\ton}{\text{ on }}
\newcommand{\tin}{\text{ in }}
\newcommand{\tfa}{\text{ for all }}
\newcommand{\tfor}{\text{ for }}
\newcommand{\tas}{\text{ as }}
\newcommand{\tand}{\text{ and }}

\newcommand{\mymatrix}[1]{\uul{\Uu{#1}}}
\newcommand{\MA}{{\mymatrix{A}}}

\newcommand{\MI}{{\mymatrix{I}}}

\newcommand{\MM}{{\mymatrix{M}}}

\newcommand{\sfA}{{\mathsf A}}

\newcommand{\SPD}{{\mathsf{SPD}}}
\newcommand{\Sym}{{\mathsf{Sym}}}

\newcommand{\domainimp}{\Omega}

\renewcommand{\mymatrix}[1]{\mathsf{#1}}

\newcommand{\imp}{_{\mathrm{imp}}}

\newcommand{\dudnA}{\pdiff{u}{n_\MA}}
\newcommand{\hb}{\widehat{\bb}}

\newcommand{\wn}{\omega}
\newcommand{\esssupp}{{\rm ess\,supp}}

\newcommand{\epsilonmin}{\epsilon_{\min}}
\newcommand{\epsilonmax}{\epsilon_{\max}}
\newcommand{\mumin}{\mu_{\min}}
\newcommand{\mumax}{\mu_{\max}}

\newcommand{\epsilonconstant}{\gamma_\epsilon}
\newcommand{\muconstant}{\gamma_\mu}

\newcommand{\weightedLtmu}{L^2(B_R;\mu)}
\newcommand{\weightedLtepsilon}{L^2(B_R;\epsilon)}

\numberwithin{equation}{section}

\usepackage{fancyhdr}
\newcommand{\BE}{\bE}%\boldsymbol E}

\newcommand{\BJ}{\bJ}%\boldsymbol J}
\newcommand{\BK}{\bK}%\boldsymbol K}
\newcommand{\BL}{L}%\boldsymbol L}

\newcommand{\gepsilon}{\gamma_{\epsilon}}
\newcommand{\gmu}{\gamma_{\mu}}

\newcommand{\wnorm}[3]{ \N{#3}_{L^2(#2; #1)}}

\newcommand{\epsmin}{\epsilon_{\min}}

\newcommand{\Rscat}{R_{\rm scat}}

\title{Explicit bounds for the high-frequency time-harmonic Maxwell equations in heterogeneous media}

\author{Th\'eophile Chaumont-Frelet\thanks{
University C\^{o}te d'Azur, Inria, CNRS, LJAD, 
2004 Route des Lucioles, 06902 Valbonne, France
(\texttt{theophile.chaumont@inria.fr})},\,
Andrea Moiola\thanks{Department of Mathematics, University of Pavia, 27100 Pavia, Italy (\texttt{andrea.moiola@unipv.it})},
\, Euan A.\ Spence\thanks{Department of Mathematical Sciences, University of Bath, Bath, BA2 7AY (\texttt{E.A.Spence@bath.ac.uk})}, 
}

\date{\today}

\begin{document}

\maketitle

\begin{abstract}
We consider the time-harmonic Maxwell equations posed in $\Rea^3$. 
We prove a priori bounds on the solution for $L^\infty$ coefficients $\epsilon$ and $\mu$ satisfying certain monotonicity properties, with these bounds 
valid for arbitrarily-large frequency, and explicit in the frequency and properties of $\epsilon$ and $\mu$. 
The class of coefficients covered includes (i) certain $\epsilon$ and $\mu$ for which well-posedness of the time-harmonic Maxwell equations had not previously been proved, and (ii)
scattering by a penetrable $C^0$ star-shaped obstacle
where $\epsilon$ and $\mu$ are smaller inside the obstacle than outside.
In this latter setting, the bounds are uniform across all such obstacles, and the first sharp frequency-explicit bounds for this problem at high-frequency.


\end{abstract}

\section{Introduction}\label{sec:intro}

\subsection{The Maxwell transmission problem}\label{sec:set_up}
We consider the time-harmonic Maxwell equations (in first-order form) posed in $\Rea^3$; i.e., find $\bE,\bH \in H_{\rm loc}(\curl;\Rea^3)$ satisfying
\begin{align}\label{eq:first_order}
\ri\wn\epsilon\bE+\nabla\times\bH=\bJ, \qquad
-\ri\wn\mu\bH+\nabla\times\bE=\bK
\end{align}
where the frequency $\omega>0$, the sources $\bJ, \bK\in L^2_{\rm comp}(\Rea^3)$, 
and the coefficients $\epsilon$ and $\mu$ are $3\times3$ real, symmetric, positive-definite matrices
 (with the set of these matrices denoted by $\SPD$)
such that $\epsilon = \epsilon_0\MI$ and $\mu = \mu_0 \MI$ outside a compact set, with
$\epsilon_0,\mu_0>0$.
The fields $\bE(\bx), \bH(\bx)$ additionally satisfy the Silver--M\"uller radiation condition 
\begin{equation}
\label{eq:intro_SM}
\abs{\sqrt{\epsilon_0}\bE(\bx)-\sqrt{\muz}\bH(\bx)\times\hx}=\calO(r^{-2})
\,\,\text{or}\,\,
\abs{\sqrt{\mu_0}\bH(\bx)+\sqrt{\epz}\bE(\bx)\times\hx}=\calO(r^{-2})
\,\,\tas r:=|\bx|\to \infty
\end{equation}
uniformly in $\hx:= \bx/r$.

The PDEs in \eqref{eq:first_order} are understood in a distributional sense, and thus are well defined for $\epsilon,\mu\in L^\infty$.
We are particularly interested in the case when $\epsilon$ and $\mu$ are discontinuous. 
 Recall that when $\epsilon$ and $\mu$ have a single jump on a common interface, \eqref{eq:first_order}  corresponds to transmission by a penetrable obstacle (see the examples in \S\ref{sec:main_results} below) with the conditions that the tangential jumps across the interface of both $\bE$ and $
 \bH$ are zero (coming from the condition that $\bE,\bH\in H_{\rm loc}(\curl;\Rea^3)$).

\subsection{Statement of the main results}\label{sec:main_results}

The main results of this paper give bounds on the solution of \eqref{eq:first_order}--\eqref{eq:intro_SM} that are explicit in both $\wn$ and properties of the coefficients $\epsilon$ and $\mu$, and valid for arbitrarily-large $\wn$.

\paragraph{Notation.} 
Given $\MM_1,\MM_2\in\SPD$ we write $\MM_1\preceq \MM_2$ to denote inequality in the sense of quadratic forms, namely $\MM_1\bv\cdot \conj\bv\le \MM_2\bv\cdot \conj\bv$ for all $\bv\in\IC^3$. 
For a positive scalar $m$ and $\MM\in\SPD$ we write $m\preceq \MM$ ($\MM\preceq m$) if $m\MI\preceq\MM$ ($\MM\preceq m\MI$, respectively), where $\MI$ is the identity matrix.
For a matrix field $\MM\in L^\infty(D,\SPD)$ defined over a given domain $D$, we define $\essinf_{\bx\in D} \MM$ as the largest non-negative number $m'$ such that $m'\preceq\MM(\bx)$ for almost every $\bx\in D$; $\esssup_{\bx\in D} \MM$ is defined similarly. 
For $\MM\in \SPD$, let 
\beq\label{eq:weighted_norms}
\N{\bv}^2_{L^2(D;\MM)}:= (\MM \bv, \bv)_{L^2(D)} \quad\tfa \bv \in \IC^3.
\eeq

Let $0<\epsilonmin\leq \epsilonmax<\infty$ and
$0<\mumin\leq \mumax<\infty$ be such that
\beq\label{eq:limits}
\epsilonmin\preceq \epsilon(\bx)\preceq \epsilonmax, \qquad
\mumin\preceq \mu(\bx)\preceq \mumax
\qquad \text{for almost every }\bx\in\IR^3,
\eeq
and
let $R>0$ be such that 
\beq\label{eq:support_scatterer}
 \supp(\epsilon- \epsilon_0\MI )\cup\supp(\mu-\mu_0 \MI )  \Subset B_R \quad\tand\quad \supp\, \bJ \cup \supp\, \bK \subset B_R,
\eeq
where $B_R$ denotes the ball of radius $R$ centred at the origin.

\begin{theorem}\mythmname{Bound on transmission problem with certain $W^{1,\infty}$ coefficients}\label{thm:BoundSmooth}
Suppose that, in addition to the set up in \S\ref{sec:set_up}, 
$\epsilon,\mu\in W^{1,\infty}(\IR^3,\SPD)$ 
with
\begin{align}\label{eq:GrowthCoeff}
\epsilonconstant:=\essinf_{\bx\in\Rea^3}
\big(\MI+\big((\bx\cdot\nabla)\epsilon \big)\epsilon^{-1}\big)
>0,
\qquad
\muconstant:=\essinf_{\bx\in\Rea^3}
\big(\MI+\big((\bx\cdot\nabla)\mu \big)\mu^{-1}\big)
>0.
\end{align}
Then the solution $\bE, \bH \in H_{\rm loc}(\curl;\Rea^3)$ of \eqref{eq:first_order}--\eqref{eq:intro_SM} exists, is unique, and satisfies
\begin{align}\nonumber
\epsilonconstant\N{\bE }^2_{L^2(B_R;\epsilon)} + \muconstant\N{\bH}^2_{L^2(B_R;\mu)} 
&\le
2\max\left\{ 8R^2 \left(\frac{\epsilonmax\mumax}{\muconstant} + \frac{\epsilon_0 \mu_0}{\epsilonconstant}\right), \frac{\epsilonconstant}{\omega^2}\right\}
\N{\bJ}_{L^2(B_R;\epsilon^{-1})}^2
\\
&+2\max\left\{ 8R^2 \left(\frac{\epsilonmax\mumax}{\epsilonconstant} + \frac{\epsilon_0 \mu_0}{\muconstant}\right), \frac{\muconstant}{\omega^2}\right\}
\N{\bK}_{L^2(B_R;\mu^{-1})}^2
\label{eq:thm:BoundSmooth}
\end{align}
where each instance of $8$ on the right-hand side reduces to $4$ if either $\bK=\bzero$ or $\bJ=\bzero$.
\end{theorem}

Note that $\epsilonconstant,\muconstant\le1$, and that $\epsilonconstant=1$ when $\epsilon$ is radially non-decreasing (and similarly for $\gamma_\mu$).
The conditions in \eqref{eq:GrowthCoeff} can be rewritten as
\beqs
\pdiff{}{r}\big( r\epsilon \big)\epsilon^{-1} \succeq \epsilonconstant
\quad\tand\quad
\pdiff{}{r}\big( r\mu \big) \mu^{-1}\succeq \muconstant, \quad \text{a.e.\ in }\IR^3,
\eeqs
suggesting that the most general $\epsilon$ and $\mu$ for which we can prove a bound are 
\beqs
\epsilon(\bx) = \epsilon^* \MI + \frac{\widetilde{\Pi}_{\epsilon}(\bx)}{r},
\quad\tand\quad 
\mu(\bx) = \mu^* \MI + \frac{\widetilde{\Pi}_{\mu}(\bx)}{r}
\eeqs
for $\epsilon^*, \mu^*>0$, where $\widetilde{\Pi}_\epsilon,\widetilde{\Pi}_{\mu} \in L^\infty(\Rea^3, \SPD)$,
are monotonically non-decreasing, in the sense of quadratic forms, in the radial direction.
To avoid technicalities arising from the singularity of $1/r$ at the origin,
we prove a bound under the following slightly-more-restrictive conditions.%

\begin{theorem}\mythmname{Bound on transmission problem with radially non-decreasing $L^\infty$ coefficients}\label{thm:BoundRough}
Suppose that, in addition to the set up in \S\ref{sec:set_up}, 
$\epsilon,\mu\in L^\infty(\Rea^3;\SPD)$ are such that 
\beq\label{eq:RoughCoefficients}
\epsilon(\bx) = \epsilonmin \MI +\Pi_{\epsilon}(\bx)  
\quad\tand\quad
\mu(\bx) = \mumin\MI + \Pi_{\mu}(\bx)
\quad\text{ for almost every }\, \bx\in \Rea^3,
\eeq
where $\Pi_{\epsilon},\Pi_{\mu} \in L^\infty(\Rea^3;\SPD)$
are monotonically non-decreasing, in the sense of quadratic forms, in the radial direction, i.e.,~for all $h\geq 0$,
\beq\label{eq:monotone}
\essinf_{\bx\in \Rea^3} \Big[ \Pi_{\epsilon}((1+h)\bx)- \Pi_{\epsilon}(\bx)\Big] \succeq 0.
\quad
\essinf_{\bx\in \Rea^3} \Big[ \Pi_{\mu}((1+h)\bx)- \Pi_{\mu}(\bx)\Big] \succeq 0.
\eeq
Then the solution of\eqref{eq:first_order}--\eqref{eq:intro_SM} exists, is unique, and satisfies the bound \eqref{eq:thm:BoundSmooth} 
with $\epsilonconstant=\muconstant=1$, $\epsilonmax=\epz$ and $\mumax=\mu$; i.e., 
\begin{equation*}
\N{\bE }^2_{L^2(B_R;\epsilon)} + \N{\bH}^2_{L^2(B_R;\mu)} 
\le
2
\max\big\{ 16R^2\epz\muz,\omega^{-2}\big\}\big(\N{\bJ}_{L^2(B_R;\epsilon^{-1})}^2+\N{\bK}_{L^2(B_R;\mu^{-1})}^2\big)
\end{equation*}
(with $16$ reducing to $8$ if either $\bK=\bzero$ or $\bJ=\bzero$).
\end{theorem}

The Helmholtz analogues of Theorems \ref{thm:BoundSmooth} and \ref{thm:BoundRough} are \cite[Theorem~2.5]{GrPeSp:18} and \cite[Theorem~2.7]{GrPeSp:18}, respectively, with the analogous Helmholtz result about transmission through a penetrable obstacle \cite[Theorem~3.1]{MS17}. The earlier papers \cite{Bl:73, BlKa:77, PeVe:99} prove similar Helmholtz bounds under assumptions on the coefficients similar to \eqref{eq:GrowthCoeff}; see \cite[Page 311]{Bl:73}, \cite[Equation 4.3]{BlKa:77}, \cite[Equation 1.8]{PeVe:99}.

\subsection{Examples of \texorpdfstring{$\epsilon$ and $\mu$}{epsilon and mu} satisfying \texorpdfstring{\eqref{eq:RoughCoefficients}}{the monotonicity condition}}

\begin{example}\mythmname{Transmission by a penetrable $C^0$ star-shaped obstacle}
\label{ex:1}
$\domain_-$ is a $C^0$ bounded open set that is star-shaped with respect to the origin (i.e., the segment $[\bzero,\bx]\subset \domain_-$ for all $\bx\in \domain_-$), 
\beqs
\epsilon(\bx) = 
\left(\epsilon_-\boldsymbol{1}_{\domain_-}(\bx)+ \epsilon_0\boldsymbol{1}_{\Rea^d\setminus\overline{\domain_-}}(\bx) \right)\MI
\quad\text{ and }
\quad
\mu(\bx) = 
\left(\mu_-\boldsymbol{1}_{\domain_-}(\bx)+ \mu_0\boldsymbol{1}_{\Rea^d\setminus\overline{\domain_-}}(\bx) \right)\MI 
\eeqs
where $\epsilon_-, \epsilon_0,\mu_-, \mu_0$ are positive real numbers satisfying
\beq\label{eq:jump}
\epsilon_-\leq \epsilon_0
\quad\tand\quad
\mu_-\leq \mu_0 
\eeq
\end{example}

Observe that the condition \eqref{eq:jump} implies that $\epsilon_-\mu_- \leq \epsilon_0 \mu_0$, which implies that $(\epsilon_0 \mu_0)^{-1/2}\leq (\epsilon_- \mu_-)^{-1/2}$; i.e., the wave speed outside $\domain_-$ is smaller than the wave speed inside.

Rotating the 2-d domain in Figure \ref{fig:1} around the vertical axis through the origin gives a 3-d domain $\Omega_-$ satisfying the conditions in Example \ref{ex:1}. This example therefore includes domains with inner and outer cusps.

A key feature of Theorem \ref{thm:BoundRough} for $\epsilon,\mu$ as in Example \ref{ex:1} is that the bound \eqref{eq:thm:BoundSmooth} is then uniform across all such penetrable obstacles; \S\ref{sec:UQ} below discusses one important application of this feature in the theory of uncertainty quantification for the time-harmonic Maxwell equations.

\begin{example}\mythmname{Transmission by a penetrable obstacle with self-intersecting surface}
\label{ex:2}
$\domain_1,\, \domain_2,$ and $\domain_3$ are the 2-d domains in Figure \ref{fig:2} rotated around the vertical axis through the origin, and 
$\epsilon|_{\Omega_j}=\epsilon_j\MI$ for $j\in\{1,2,3\}$,
where $0<\epsilon_1\leq \epsilon_2\leq \epsilon_3=\epsilon_0$, and similarly for $\mu$.
\end{example}

\begin{example}\mythmname{Transmission by a penetrable obstacle with infinitely many components accumulating towards a bounded limit surface}
\label{ex:3}
$\epsilon|_{B_{j/(j+1)}\setminus B_{(j-1)/j}}=\epsilon_j$ for $j\in\IN$ and $\epsilon|_{\IR^3\setminus B_1}=\epz$ 
where
$\epsilon_1\leq \ldots \leq \epsilon_j \leq \epsilon_{j+1}\leq \ldots \leq \epsilon_0$, and similarly for $\mu$.
\end{example}

Figure \ref{fig:3} plots a 2-d cross section of the domains on which $\epsilon$ and $\mu$ are constant in Example \ref{ex:3}.%

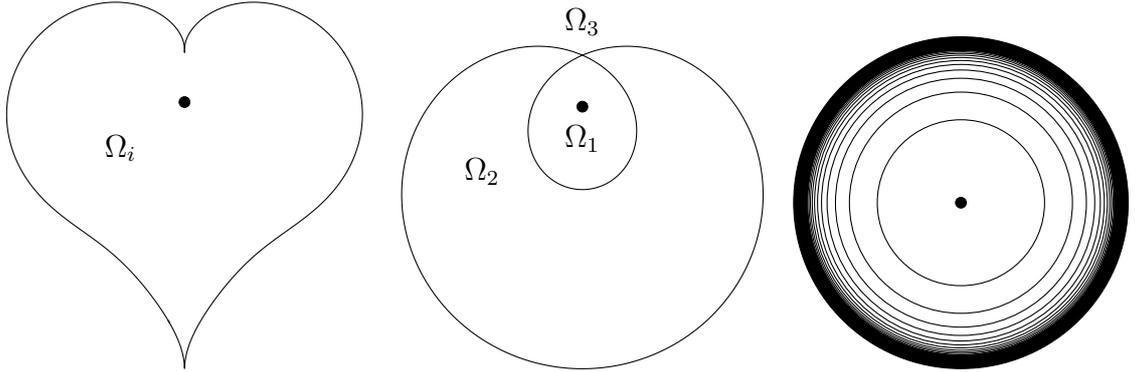
\begin{figure}
  \centering
  \subfloat[A 2-d cross section of an admissible $\domain_-$ in Example \ref{ex:1}]
  {
\begin{tikzpicture}[scale=2.1,inner sep=0.5mm]
\draw[domain = 1.56:1.58,smooth,variable=\t,samples=100] % One cusp
plot ({\t r}:{(2 - 2 * sin(\t r) + sin(\t r)*sqrt(abs(cos(\t r)))/(sin(\t r)+1.4))/2}); % From http://mathworld.wolfram.com/HeartCurve.html 
\draw[domain = 4.65:4.75,smooth,variable=\t,samples=100] % Another cusp
plot ({\t r}:{(2 - 2 * sin(\t r) + sin(\t r)*sqrt(abs(cos(\t r)))/(sin(\t r)+1.4))/2}); % From http://mathworld.wolfram.com/HeartCurve.html 
\draw[domain = 0:1.56,smooth,variable=\t,samples=200] % A nice bit
plot ({\t r}:{(2 - 2 * sin(\t r) + sin(\t r)*sqrt(abs(cos(\t r)))/(sin(\t r)+1.4))/2}); % From http://mathworld.wolfram.com/HeartCurve.html 
\draw[domain = 1.58:4.66,smooth,variable=\t,samples=200] % Another nice bit
plot ({\t r}:{(2 - 2 * sin(\t r) + sin(\t r)*sqrt(abs(cos(\t r)))/(sin(\t r)+1.4))/2}); % From http://mathworld.wolfram.com/HeartCurve.html 
\draw[domain = 4.749:6.29,smooth,variable=\t,samples=200] % Another nice bit
plot ({\t r}:{(2 - 2 * sin(\t r) + sin(\t r)*sqrt(abs(cos(\t r)))/(sin(\t r)+1.4))/2}); % From http://mathworld.wolfram.com/HeartCurve.html 
\node at (0,-0.31) [circle,draw,fill=black] {};
\node at (-0.4,-0.6) {{\large $\domain_i$}};
\end{tikzpicture}     \label{fig:1}  }  \hfill
  \subfloat[A 2-d cross section of the subdomains on which $\epsilon$ and $\mu$ are constant in Example \ref{ex:2}.]
  {
\begin{tikzpicture}[scale=2.2,inner sep=0.5mm]
\draw[domain = 0:6.2832,smooth,variable=\t,samples=400]plot ({\t r}:{(0.4-sin(\t r))*1.35});
\node at (0,-0.31) [circle,draw,fill=black] {};
\node at (0,-0.5) {{\large $\domain_1$}}; %was (0,-0.2) in Owen's original diagram
\node at (-0.6,-0.7) {{\large $\domain_2$}};
\node at (0,0.2) {{\large $\domain_3$}};
\end{tikzpicture}    \label{fig:2}  }  \hfill
  \subfloat[A 2-d cross section of the subdomains on which $\epsilon$ and $\mu$ are constant in Example \ref{ex:3}.]
  {
\begin{tikzpicture}[scale=2.2,inner sep=0.5mm]
\fill[fill=black] (0,0) circle [radius =1];
\fill[fill=white] (0,0) circle [radius =11/12];
\foreach \x in {1,2,...,11}
\draw (0,0) circle [radius = (\x/(\x+1))];
\draw (0,0) circle [radius = 1];
% Add the origin
\node at (0,0) [circle,draw,fill=black] {};
\end{tikzpicture}   \label{fig:3}  }
\caption{2-d cross sections of domains on which $\epsilon$ and $\mu$ are constant in Examples \ref{ex:1}--\ref{ex:3}; in all three, the black dot marks the origin.}
\end{figure}

\subsection{Discussion of Theorems \ref{thm:BoundSmooth} and \ref{thm:BoundRough} and their novelty}

\subsubsection{The \texorpdfstring{$\wn$}{omega}-dependence of the bounds}

The $\wn$-dependence of the bound \eqref{eq:thm:BoundSmooth} is the same as in the sharp bound on the solution to \eqref{eq:first_order}--\eqref{eq:intro_SM} when $\mu=\mu_0\MI$ and $\epsilon=\epsilon_0\MI$. A
simple way to see this sharpness is to 
let $\bE = \chi \bE^I$ and $\bH= \chi \bH^I$ where $\chi \in C^\infty_{\rm comp}(B_R)$ and $\bE^I$ and $\bH^I$ are the 
plane-wave solutions
\beq\label{eq:plane_wave}
\bE^I(\bx)=\sqrt{\muz}\bA\re^{\ri\wn\sqrt{\epz\muz}\bd\cdot\bx} \quad\tand\quad
\bH^I=\sqrt{\epz}\bd\times\bA\re^{\ri\wn\sqrt{\epz\muz}\bd\cdot\bx}
\eeq
with $|\bA|=|\bd|=1$ and $\bA\cdot\bd=0$.
Then $(\bE,\bH)$ is solution of the Maxwell problem \eqref{eq:first_order}--\eqref{eq:intro_SM} with 
$\bJ=\nabla\chi\times\bH^I$ and $\bK=\nabla\chi\times\bE^I$,
and the $L^2$ norms of the solutions and sources are all independent of $\wn$:
\begin{align*}
\muz^{-1/2}\N{\bE}_{L^2(\IR^3)}
&=\epz^{-1/2}\N{\bH}_{L^2(\IR^3)}
=\N{\chi}_{L^2(\IR^3)},\\
\N{\bJ}_{L^2(\IR^3)}&=\epz^{1/2}\N{\nabla\chi(\bx)\times(\bd\times\bA)}_{L^2(\IR^3)},
\qquad
\N{\bK}_{L^2(\IR^3)}=\muz^{1/2}\N{\nabla\chi(\bx)\times\bA}_{L^2(\IR^3)}.
\end{align*}

These plane-wave solutions also show that the dependence on $R$ in the bound \eqref{eq:thm:BoundSmooth} is sharp and that this bound cannot be improved, in general, by a factor larger than  $32\pi^2\approx316$; indeed, 
fix $R>0$ and define $\chi\in C^{0,1}(\IR^3)$ by 
$\chi(\bx):=R\sin(\pi|\bx|/R) (\pi|\bx|)^{-1}$ if $|\bx|<R$ and $\chi(\bx):=0$ otherwise.
$\chi$ is the first Laplace--Dirichlet eigenfunction of the ball (the spherical Bessel function $j_0(\frac\pi R|\bx|)$): $\Delta \chi+\frac{\pi^2}{R^2}\chi=0$ in $B_R$ and $\chi=0$ on $\deB_R$, so $\N{\nabla\chi}_{B_R}^2=(\pi/R)^2\N{\chi}_{B_R}^2$.
Choosing $\bE,\bH$ as the cut-off plane waves above, we obtain a solution of the constant-coefficient problem \eqref{eq:first_order}--\eqref{eq:intro_SM} with 
\begin{align*}
32\epz\muz R^2
&\overset{\eqref{eq:thm:BoundSmooth}}\ge\frac{\N{\bE}_{L^2(B_R;\epsilon)}^2+\|\bH\|_{L^2(B_R;\mu)}^2}{\|\bK\|_{L^2(B_R;\mu^{-1})}^2+\|\bJ\|_{L^2(B_R;\epsilon^{-1})}^2}
\\&
=\frac{2\epz\muz\N{\chi}_{L^2(B_R)}^2}{\N{\nabla\chi\times(\bd\times\bA)}_{L^2(B_R)}^2+\N{\nabla\chi\times\bA}_{L^2(B_R)}^2}
\ge\frac{\epz\muz\N{\chi}_{B_R}^2}{\N{\nabla\chi}_{B_R}^2}
=\frac{\epz\muz R^2}{\pi^2}.
\end{align*}

It is well-known that the behaviour of solutions of the time-harmonic Maxwell equations in the limit $\wn\to\infty$ is dictated by the behaviour of the geometric optic rays, and the conditions on $\epsilon$ and $\mu$ in \eqref{eq:GrowthCoeff} ensure that all the rays starting in a neighbourhood of $\supp(\epsilon-\epsilon_0\MI)\cup\supp(\mu-\mu_0\MI)$ escape from that neighbourhood in a uniform time -- see, e.g, \cite[\S7]{GrPeSp:18} -- i.e., the problem is \emph{nontrapping}. When $\epsilon$ and $\mu$ correspond to transmission through a penetrable obstacle (as in Example \ref{ex:1}), the conditions \eqref{eq:jump} imply that the wave speed inside the obstacle is larger than the wave speed outside the obstacle, ruling out total internal reflection, and thus ruling out trapped rays.

Obtaining $\omega$-explicit bounds on solutions of the 
Helmholtz equation 
\beq\label{eq:Helmholtz}
\nabla\cdot(A\nabla u)+\omega^{2} n u =-f
\eeq
under the nontrapping hypothesis is a classic topic. These bounds then imply results about the location of resonances and the local-energy decay of solutions of the corresponding wave equation; see the overview in \cite[\S4.6]{DyZw:19}.
For certain nontrapping geometries and coefficients, such $\wn$-explicit bounds can be obtained by multiplying the PDE by a carefully-chosen test function and integrating by parts \cite{MoLu:68, Mo:75, Bl:73, BlKa:77, PeVe:99}; this is the method we use in the Maxwell case -- see the discussion in \S\ref{sec:overview} below. 
For smooth geometries and coefficients, the propagation-of-singularities results of \cite{MeSj:78, MeSj:82} and the parametrix argument of \cite{Va:75} prove the sharp bound on the solution under the general nontrapping hypothesis (see the recent presentation in \cite[Theorem 4.43]{DyZw:19}).
For Helmholtz transmission problems, i.e., \eqref{eq:Helmholtz} with discontinuous $A$ and $n$,
the propagation-of-singularities results are much more complicated, and $\wn$-explicit bounds on the Helmholtz transmission problem proved using propagation of singularities only exist for smooth obstacles with strictly positive curvature; see \cite{CaPoVo:99}.

For the Maxwell equations, the propagation-of-singularities results analogous to \cite{MeSj:78,MeSj:82} were proved in \cite{Ya:88}
for the constant-coefficient Maxwell equations in the exterior of a perfectly-conducting obstacle. To the best of our knowledge, there do not yet exist corresponding results for the Maxwell transmission problem.

As stated above, the bounds in the present paper are proved by multiplying the PDE by a carefully-chosen test function and integrating by parts
(see the discussion in \S\ref{sec:overview} below). Perhaps surprisingly, the present paper appears to be the first time this technique has been applied to the time-harmonic Maxwell transmission problem, and Theorem \ref{thm:BoundRough} therefore contains the first $\wn$-explicit bounds on the solution of this problem.

We highlight that if the monotonicity conditions \eqref{eq:GrowthCoeff} or \eqref{eq:jump} on $\epsilon$ and $\mu$ are violated then the Maxwell solution operator can grow exponentially through a sequence of $\wn$s; this is proved for the Helmholtz solution operator in \cite{Ra:71} (for smooth coefficients) and \cite{PoVo:99, Cap12}, \cite[Chapter 5]{AC16} 
(for discontinuous coefficients such that the wave speed outside is higher than the wave speed inside).

\subsubsection{The novelty of the well-posedness result in Theorem \ref{thm:BoundRough}}

The class of coefficients for which existence and uniqueness of the Maxwell solution is proved in Theorem \ref{thm:BoundRough} contains configurations for which existence and uniqueness of the Maxwell solution had not yet been established. Indeed, 
the general arguments of \cite{PiWeWi01} prove existence of a solution to \eqref{eq:first_order}--\eqref{eq:intro_SM} for $\bJ,\bK\in  L^2_{\rm comp}(\IR^3)$
once uniqueness is established; see \cite[Theorem 2.10]{PiWeWi01}. 
The Baire-category argument of \cite{BCT12} uses the fact that a UCP is known for the time-harmonic Maxwell system with Lipschitz coefficients \cite{NgWa:12} to prove uniqueness of the solution of the time-harmonic Maxwell problem posed in $\Rea^3$ with \emph{piecewise}-Lipschitz coefficients, provided that the subdomains on which the coefficients are defined 
satisfy \cite[Assumption 1]{BCT12}/\!\!\cite[(i)--(iii) in statement of Proposition 2.11]{LRX16}.
This assumption allows a large class of subdomains (including any bounded finite collection), but does not allow the subdomain boundaries to concentrate ``from below" on a bounded surface in $\Rea^3$, and thus the subdomains corresponding to Example \ref{ex:3} are ruled out;
 see \cite[Figure 1]{BCT12} (in the notation of \cite[Proposition 2.13]{LRX16}, $\partial B_1(\bzero)\subset C$ and thus $\Rea^3\setminus C$ is not connected).
Theorem \ref{thm:BoundRough} is therefore the first time existence and uniqueness of the Maxwell solution with $\epsilon,\mu$ as in Example \ref{ex:3} has been proved.

\subsubsection{Uniformity of the bounds with respect to the penetrable obstacle and implications for uncertainty quantification}\label{sec:UQ}

In a PDE context, uncertainty quantification consists of theory and algorithms for computing statistics of quantities of interest involving PDEs \emph{either} posed on a random domain \emph{or} having random coefficients. The bounds in the present paper are relevant for UQ of the time-harmonic Maxwell equations in two ways.

First, when $\epsilon$ and $\mu$ are as in Example \ref{ex:1}, i.e., correspond to the transmission problem by a penetrable $C^0$ obstacle, the coefficients in the bound \eqref{eq:thm:BoundSmooth} are uniform across all such obstacles.
The result analogous to this for the Helmholtz equation was proved in \cite[Theorem 3.1]{MS17}/\cite[Theorem 2.7]{GrPeSp:18}, and is used in  \cite{HiScSp:22} as the basis of a frequency-explicit analysis of UQ algorithms for the high-frequency Helmholtz transmission problem (following the analysis at fixed frequency in \cite{HiScScSc:18}). 
Theorem \ref{thm:BoundRough} can therefore form the basis of the Maxwell analogue of \cite{HiScSp:22}.

Second, the bounds on the Helmholtz equation in \cite{GrPeSp:18} that are explicit in properties of the coefficients were used in \cite{PeSp:18} to prove the first well-posedness results about the Helmholtz equation with random coefficients (see \cite[\S1.1]{PeSp:18}). Inputting the bounds from Theorem \ref{thm:BoundRough} into the general framework of \cite[\S2]{PeSp:18}, one can obtain the analogous results for the time-harmonic Maxwell equations.

\subsubsection{Application to the time-harmonic Maxwell scattering problem}

Given $\epsilon,\mu$ as in \S\ref{sec:set_up} and $\bE^I, \bH^I$ satisfying
\beqs
\ri\wn\epsilon_0\bE^I+\nabla\times\bH^I=\bzero \quad\tand\quad
-\ri\wn\mu_0\bH^I+\nabla\times\bE^I=\bzero
\quad\tin \Rea^3,
\eeqs
such as the plane-wave solutions \eqref{eq:plane_wave}, the scattering problem is:~find $\bE^T$ (the total field) satisfying 
\beqs
\ri\wn\epsilon\bE^T+\nabla\times\bH^T=\bzero \quad\tand\quad
-\ri\wn\mu\bH^T+\nabla\times\bE^T=\bzero
\eeqs
and such that the scattered fields $\bE^S:= \bE^T - \bE^I$, $\bH^S:=\bH^T- \bH^I$ satisfy the Silver--M\"uller radiation condition \eqref{eq:intro_SM}.
Theorems \ref{thm:BoundSmooth} and \ref{thm:BoundRough} imply the follow result about the scattering problem.

\begin{cor}\mythmname{Bound on the total field in the scattering problem} \label{cor:scat}
Suppose that $\epsilon$ and $\mu$ satisfy the assumptions of \emph{either} Theorem \ref{thm:BoundSmooth} \emph{or} Theorem \ref{thm:BoundRough}.
Then the solution $\bE^T, \bH^T$ of the scattering problem exists and is unique. 
Let $\Rscat:= \min\{ R : \supp(\epsilon-\epsilon_0\MI)\cup \supp(\mu-\mu_0\MI)\subset B_R\}$. Given $R>\Rscat$, 
\begin{align}\nonumber
&\epsilonconstant\big\|\bE^T\big\|^2_{L^2(B_R;\epsilon)} + \muconstant\big\|\bH^T\big\|^2_{L^2(B_R;\mu)} \\ \nonumber
&\le 
2 \epsilonconstant \big\|\bE^I\big\|^2_{L^2(B_R;\epsilon)} + 2 \muconstant\big\|\bH^I\big\|^2_{L^2(B_R;\mu)} 
\\ \nonumber
&\qquad+\frac{4}{(R-\Rscat)^2}
\bigg(
\max\left\{ 8R^2 \left(\frac{\epsilonmax\mumax}{\muconstant} + \frac{\epsilon_0 \mu_0}{\epsilonconstant}\right), \frac{\epsilonconstant}{\omega^2}\right\}
\N{ \bH^I }_{L^2(B_R\setminus B_{\Rscat};\epsilon^{-1})}^2\\
&\hspace{3.5cm}+\max\left\{ 8R^2 \left(\frac{\epsilonmax\mumax}{\epsilonconstant} + \frac{\epsilon_0 \mu_0}{\muconstant}\right), \frac{\muconstant}{\omega^2}\right\}
\N{ \bE^I}_{L^2(B_R\setminus B_{\Rscat};\mu^{-1})}^2\bigg);
\label{eq:BoundScat}
\end{align}
i.e., the total fields are bounded (uniformly in $\wn$) in terms of the incident fields.
\end{cor}

\subsubsection{Overview of how Theorems \ref{thm:BoundSmooth} and \ref{thm:BoundRough} are proved}\label{sec:overview}

The basic ingredient of the proofs of Theorems \ref{thm:BoundSmooth} and \ref{thm:BoundRough} is the identity 
\begin{align}\nonumber
&2\Re\Big\{\big(\nabla\times\bE-\ri\wn\mu\bH\big)\cdot(\epsilon\conj\bE\times\bx+\beta\conj \bH)
+\big(\nabla\times\bH+\ri\wn\epsilon\bE\big)\cdot(\mu\conj\bH\times\bx-\beta\conj\bE)\Big\}
\\ \nonumber
&\qquad=\nabla\cdot\Big[
2\Re\Big\{(\bE\cdot\bx)\epsilon\conj\bE+(\bH\cdot\bx)\mu\conj\bH
+\beta\bE\times\conj\bH
\Big\}
-(\epsilon \bE\cdot\conj\bE)\bx -(\mu \bH\cdot\conj\bH)\bx 
\Big]\\ 
&\qquad\qquad -2\Re\Big\{
(\bE\cdot\bx)\nabla\cdot[\epsilon\conj\bE]+(\bH\cdot\bx)\nabla\cdot[\mu\conj\bH]
+\nabla\beta\cdot\bE\times\conj\bH
\Big\}
\nonumber\\
&\qquad\qquad
+\big(\epsilon+(\bx\cdot\nabla)\epsilon\big)\bE\cdot\conj\bE
+\big(\mu+(\bx\cdot\nabla)\mu\big)\bH\cdot\conj\bH,
\label{eq:morid1}
\end{align}
for a suitable scalar function $\beta$.
The bound in Theorem \ref{thm:BoundSmooth} arises from integrating this identity over $B_R$, ensuring that the non-divergence terms on the right-hand side control the appropriate weighted $H(\curl;B_R)$ norm of $\bE$ (observe
from the right-hand side of  \eqref{eq:identity2ndorder} how the conditions $ \epsilon + (\bx\cdot \nabla)\epsilon>0$ and 
$(\mu + (\bx\cdot\nabla)\mu)>0$ then arise), and show that the term on $\partial B_R$ has the appropriate sign using the fact that $\bE$ satisfies the Silver-M\"uller radiation condition \eqref{eq:intro_SM}.
The bound in Theorem \ref{thm:BoundRough} is then obtained from Theorem \ref{thm:BoundSmooth} using approximation arguments similar to those in \cite{GrPeSp:18}, which in turn were inspired by analogous arguments in the setting of rough-surface scattering in \cite{Th:06} (with this thesis recently made available as \cite{Ba:19}).

To connect \eqref{eq:morid1} with other identities in the literature, it is convenient to consider the case when $\ri \wn \mu \bH =\nabla\times \bE$
(i.e., $\bK=\bzero$ in \eqref{eq:first_order}) and 
\eqref{eq:morid1} then becomes 
\begin{align}\nonumber
&2 \Re \left\{ \big(\wn^{-2} \nabla\times \big(\mu^{-1}\nabla \times \bE\big) -  \epsilon \bE\big) \cdot \overline{ \big( \big(\mu^{-1} \nabla\times \bE\big)\times \bx - \ri \wn \beta \bE\big)}\right\}\\ \nonumber
&= 
\nabla\cdot \Big[ 2 \Re\left\{  (\bE\cdot\bx)\epsilon \overline{\bE} + \wn^{-2}\big(\big(\mu^{-1}\nabla\times \bE\big)\cdot \bx\big) \nabla\times \overline{\bE} + \ri \wn^{-1} \beta \bE \times \big(\mu^{-1}\nabla\times \overline{\bE}\big)\right\} \\ \nonumber
& \hspace{6cm}-  \big(\epsilon \bE\cdot \overline{\bE} \big)\bx -\wn^{-2}\big(\nabla \times \bE\big) \cdot \big(\mu^{-1} \nabla\times \overline{\bE}\big) \bx\Big]\\ \nonumber
&\hspace{1cm}- 2 \Re \Big\{ (\bE\cdot\bx) \nabla \cdot \big[\epsilon \overline{\bE}\big] + \ri \wn^{-1} \nabla \beta \cdot \bE \times \big(\mu^{-1} \nabla \times \overline{\bE}\big) \Big\}\\ \label{eq:identity2ndorder}
&\hspace{1cm}
+  \big( \epsilon + (\bx\cdot \nabla)\epsilon \big) \bE\cdot \overline{\bE} +\wn^{-2} \big( (\mu + (\bx\cdot\nabla)\mu)\mu^{-1}\nabla\times \bE \big) \cdot \mu^{-1}\nabla\times \overline{\bE}.
\end{align}
Observe that the left-hand side of \eqref{eq:identity2ndorder} involves the second-order form of the time-harmonic Maxwell equations multiplied by a linear combination of $(\mu^{-1}\nabla \times \bE)\times \bx$ and $\bE$.
For the Helmholtz equation, Morawetz pioneered the use of multipliers that are a linear combination of a derivative of $u$ and $u$ itself \cite{MoLu:68, Mo:75}, with the key insight being that this linear combination could deal with the contribution ``at infinity'' -- in our case on $\partial B_R$ -- using the radiation condition (for more on this, see the more-recent presentation and discussion in \cite[Remark 2.3 and Lemma 2.4]{SpChGrSm:11}).

For the Maxwell equations in the time domain,
Morawetz herself used a linear combination of multipliers \cite{Mo:74}, 
with similar multipliers used in control theory by \cite{Ka:90, Ka:94, NiPi:05}, and in general relativity by \cite{Bl:08, AnBl:15, Ma:20}.
For the second-order form of the time-harmonic Maxwell equations, multipliers involving $(\nabla \times \bE)\times \bZ$ for a vector field $\bZ$ have been well-used; see, e.g.,  \cite{Mi:95, HaLe:11, LeNg:13, NgTr:19,AndreaPhD} 

However, the time-harmonic Maxwell equations posed in $\Rea^3$ with the Silver--M\"uller radiation condition seem not to have been studied using the multiplier technique before, and, correspondingly, we have not been able to find in the literature the identity \eqref{eq:identity2ndorder}/\eqref{eq:morid1}, involving the linear combination of multipliers needed to deal with the radiation condition.

\subsubsection{Bounds in unweighted norms}

The $L^2$ norms on the left- and right-hand sides of the bound \eqref{eq:thm:BoundSmooth} are weighted with the coefficients $\epsilon,\mu$.
Alternatively, one can repeat the arguments leading to 
\eqref{eq:thm:BoundSmooth} and work in unweighted norms.
The analogue of \eqref{eq:thm:BoundSmooth} is then
\begin{align}\label{eq:unweighted}
\begin{aligned}
\epsilon_*\N{\bE}^2_{L^2(B_R)} 
+\mu_*\N{\bH}^2_{L^2(B_R)} 
&\le 
 2\max \left\{ 
8 R^2 \bigg(\frac{\mumax^2
}{\mu_*}
\frac{\epsilonmax}{\epsilonmin}
+\frac{\epz\muz}{\epsilon_*}\frac{\epsilonmax}{\epsilonmin}
 \bigg), \frac{\epsilon_*}{\epsilonmin^2 \omega^2}\right\}
\N{\BJ}^2_{\BL^2(B_R)}\\
&
+
2\max \left\{
8 R^2
\bigg(
\frac{
\epsilonmax^2
}{\epsilon_*}\frac{\mumax}{\mumin}
+\frac{\epz\muz}{\mu_*}\frac{\mumax}{\mumin}
\bigg)
, \frac{\mu_*}{\mumin^2 \omega^2}\right\}
\N{\BK}^2_{\BL^2(B_R)},
\end{aligned}
\end{align}
where now
\begin{align}\label{eq:GrowthCoeffunweighted}
\epsilon_*:=\essinf_{\bx\in\Rea^3}\big(\epsilon+(\bx\cdot\nabla)\epsilon \big),
\qquad
\mu_*:=\essinf_{\bx\in\Rea^3}\big(\mu+(\bx\cdot\nabla)\mu \big)
\end{align}
with both assumed $>0$. Remark \ref{rem:unweighted} below discusses in more detail how to obtain \eqref{eq:unweighted}.

\subsubsection{The analogous results for the interior impedance problem}

In \S\ref{sec:IIP}, we prove results analogous to Theorems  \ref{thm:BoundSmooth} and \ref{thm:BoundRough} for the Maxwell interior impedance problem in Lipschitz domains that are star-shaped with respect to a ball.
This problem is:~given a bounded Lipschitz open set $\domainimp$
that is star-shaped with respect to a ball 
 (i.e., star-shaped with respect to each point in a ball of non-zero radius) and has outward-pointing unit vector $\hn$, $\epsilon,\mu \in W^{1,\infty}(\domainimp;\SPD)$, $\vartheta\in L^\infty(\deOimp)$ uniformly positive, $\bJ,\bK\in L^2(\domainimp)$, 
 and $\bg \in L^2_T(\partial \Omega)$, 
 find 
 \beqs
 \bE, \bH\in H\imp(\curl;\domain):=\big\{\bv\in H(\curl;\domain):\bv_T\in L^2_T(\partial \domain)\big\}
\eeqs 
 such that 
\beq\label{eq:ImpedanceFirst}
\ri\wn\epsilon\bE+\nabla\times\bH =\bJ \quad\tand\quad
-\ri\wn\mu\bH+\nabla\times\bE=\bK \quad\tin \domainimp,
\eeq
and
\beq\label{eq:ImpedanceBC}
\bH\times\hn-\vartheta\bE_T=\bg\quad\oon\deOimp,
\eeq
where $\bE_T$ denotes the tangential trace of $\bE$, defined for smooth vector fields by $\bv_T:=(\hn\times\bv)\times\hn$.
These bounds, in Theorems \ref{thm:ImpedanceSmooth} and \ref{thm:ImpedanceRough} below, generalise the bounds for $\epsilon=\mu=\MI$ in \cite[\S3]{MaxwellPDE} and are the Maxwell analogues of the Helmholtz results in \cite[Theorem 1]{BaChGo:17}, \cite[Theorem A.6]{GrPeSp:18}.

There are two reasons we prove these results. First, the interior impedance problem is a ubiquitous model problem in the numerical analysis of the time-harmonic Maxwell equations; see, e.g., \cite[Chapter 7]{MON03}, \cite{GaMe:12, HiMoPe:13, FeWu:14, LuChQi:17, Ve:19, NiTo:20, ChGrLaTa:21, MeSa:22}. 
Second, domain-decomposition methods for time-harmonic wave problems (including the Helmholtz and Maxwell equations) often use impedance boundary conditions on the subdomains, following the work of \cite{De:91, BeDe:97} in the Helmholtz context; see, e.g., \cite{BoDoGrSpTo:19, BoDoKyPe:20}.
Such impedance boundary conditions are then the starting point for so-called \emph{optimised Schwarz methods}; see, e.g.,
\cite{RoGe:06, DoLaPe:08, DoGaGe:09}. 
Just as results about the Helmholtz interior impedance problem can be used to analyse these methods in the Helmholtz context (see, e.g., the heterogeneous analysis in \cite{GoGrSp:21}), we expect the bounds in  \S\ref{sec:IIP} to play the analogous role in the analysis of Maxwell domain-decomposition methods.

The proofs of Theorems \ref{thm:ImpedanceSmooth} and \ref{thm:ImpedanceRough} below are very similar to 
those of Theorems \ref{thm:BoundSmooth} and \ref{thm:BoundRough}. In particular, recall from \S\ref{sec:overview} that, for the problem in $\Rea^3$, multiplying the PDE by a linear combination of $(\mu^{-1}\nabla\times\bE)\times \bx$ and $\bE$ allows one to deal with the contribution on $\partial B_R$ using the Silver--M\"uller radiation condition; for the interior impedance problem, this linear combination allows one to deal with the contribution on $\partial \Omega$ using the impedance boundary condition.

In fact, the results about the interior impedance problem are actually harder to prove that the results about the problem in $\Rea^3$.
The difficulty comes in integrating the Morawetz-type identity over the domain. For the problem in $\Rea^3$, we integrate the identity over $B_R$, and
 $\BE$ has sufficient regularity for this integration when $\epsilon,\mu \in W^{1,\infty}$ because of interior regularity of solutions of the Maxwell equations. In contrast, for the interior impedance problem, we integrate the identity over $\Omega$, and the regularity of the solution of the interior impedance problem when $\Omega$ is only Lipschitz is more delicate. This exactly parallels the Helmholtz case, where, at least for constant-coefficient problems, 
justifying integrating Morawetz identities over Lipschitz domains follows from the density result \cite[Theorem~1]{CoD98}, which uses the harmonic-analysis results of \cite{JeKe:81, JeKe:95}. 
In Theorem \ref{thm:Density} we generalise \cite[Theorem~1]{CoD98} to $\epsilon,\mu\neq \MI$ using 
the harmonic-analysis results of \cite{MiTa:99, MiTa:01}. This result is used to justify integrating the Morawetz-type identity over Lipschitz domains in Part (i) of Lemma~\ref{lem:IntegratedM}.

%%%%%%%%%%%%%%%%%%%%%%%%%%%%%%%%%%%%%%%%%%%%%%%%%%%%%%%%%%%%%%%%%%%%%%%%%%%%%%%%%%%%%%%%%%%%
\section{Morawetz-type identities}

\subsection{The identities in pointwise form}

In this section, $\Sym$ denotes the set of $3\times3$ real, symmetric matrices. 

\ble\mythmname{First Morawetz-type identity}\label{lem:Morawetz1}
Let $D\subset \Rea^3$ be open and let $\bE$, $\bH \in C^2(D)^3$.
Let $\epsilon,\mu\in C^1(D,\Sym)$ and $\beta\in C^1(D,\IR)$.
Then the identity \eqref{eq:morid1} holds in $D$.
%\begin{align}\nonumber
%&2\Re\Big\{\big(\nabla\times\bE-\ri\wn\mu\bH\big)\cdot(\epsilon\conj\bE\times\bx+\beta\conj \bH)
%+\big(\nabla\times\bH+\ri\wn\epsilon\bE\big)\cdot(\mu\conj\bH\times\bx-\beta\conj\bE)\Big\}
%\\ \nonumber
%&=\nabla\cdot\Big[
%2\Re\Big\{(\bE\cdot\bx)\epsilon\conj\bE+(\bH\cdot\bx)\mu\conj\bH
%+\beta\bE\times\conj\bH
%\Big\}
%-(\epsilon \bE\cdot\conj\bE)\bx -(\mu \bH\cdot\conj\bH)\bx 
%\Big]\\ 
%&\quad -2\Re\Big\{
%(\bE\cdot\bx)\nabla\cdot[\epsilon\conj\bE]+(\bH\cdot\bx)\nabla\cdot[\mu\conj\bH]
%+\nabla\beta\cdot\bE\times\conj\bH
%\Big\}
%\nonumber\\&\quad
%+\big(\epsilon+(\bx\cdot\nabla)\epsilon\big)\bE\cdot\conj\bE
%+\big(\mu+(\bx\cdot\nabla)\mu\big)\bH\cdot\conj\bH
%\qquad \iin D.\label{eq:morid1}
%\end{align}
\ele
We write \eqref{eq:morid1} as 
\beqs
R_\beta = \nabla\cdot \bQ_\beta + P_\beta;
\eeqs
i.e., $\bQ_\beta$ is defined by
\beq\label{eq:Qbeta}
\bQ_\beta:= 2\Re\Big\{(\bE\cdot\bx)\epsilon\conj\bE+(\bH\cdot\bx)\mu\conj\bH
+\beta\bE\times\conj\bH
\Big\}
-(\epsilon \bE\cdot\conj\bE)\bx -(\mu\bH\cdot\conj\bH)\bx,
\eeq
and $R_\beta$ and $P_\beta$ are defined analogously. Observe that, when $\bE$ and $\bH$ satisfy \eqref{eq:first_order} the term $R_\beta$ depends linearly on the data $\bJ$. 

The following lemma describes a special case of \eqref{eq:morid1} with $\epsilon$ and $\mu$ constant and scalar, and $\beta = r\sqrt{\epsilon\mu}$; in this case $P_\beta$ is the sum of (i) terms that vanish when $\nabla\cdot[\epsilon \bE]=0$ and $\nabla \cdot [\mu \bH]=0$, and (ii) terms that are non-negative.

\ble\mythmname{Second Morawetz-type identity}\label{lem:Morawetz2}
Let $D\subset \Rea^3$ and let $\bE$, $\bH \in C^2(D)^3$.
Let $\epsilon_0,\mu_0$ be {real-valued} constants and recall that $r:=|\hx|$. Then
\begin{align}\nonumber
&2\Re\Big\{\big
(\nabla\!\times\!\bE-\ri\wn\muz\bH\big)\cdot(\epz\conj\bE\times\bx+r\sqrt{\epz\muz}\conj \bH)
+\big(\nabla\!\times\!\bH+\ri\wn\epz\bE\big)\cdot(\muz\conj\bH\times\bx-r\sqrt{\epsilon_0 \mu_0}\conj\bE)\Big\}
\\ \nonumber
&=\nabla\cdot\Big[
2\Re\Big\{\epsilon_0(\bE\cdot\bx)\conj\bE+\mu_0(\bH\cdot\bx)\conj\bH
+r\sqrt{\epsilon_0 \mu_0}\bE\times\conj\bH
\Big\}
-\epsilon_0|\bE|^2\bx-\mu_0|\bH|^2\bx
\Big]\\  \nonumber
&\;\; -2\Re\Big\{
(\bE\cdot\bx)\nabla\cdot[\epsilon_0\conj\bE]+(\bH\cdot\bx)\nabla\cdot[\mu_0\conj\bH]
\Big\}
+\frac{\epsilon_0}{2}\big( |\bE|^2 - |\bE\times \hx|^2\big)
+\frac{\mu_0}{2}\big( |\bH|^2 - |\bH\times \hx|^2\big)\\
&\;\;+ 
\frac{1}{2\mu_0} \big| \mu_0 \bH \times \hx - \sqrt{\epsilon_0 \mu_0} \bE\big|^2 + \frac{1}{2\epsilon_0}\big| \epz \hx \times \bE - \sqrt{\epsilon_0 \mu_0} \bH\big|^2.
\label{eq:morid1a}
\end{align}
\ele

To prove the bound \eqref{eq:thm:BoundSmooth} on the solution of the transmission problem, the plan is to use the identity \eqref{eq:morid1} in $B_R$ and then the identity \eqref{eq:morid1a} in $\Rea^3\setminus B_R$ to deal with the contribution from infinity.

\bpf[Proof of Lemma \ref{lem:Morawetz1}]
The identity \eqref{eq:morid1} is the sum of the identity for the ``Rellich multipliers'', 
i.e.\ the test fields $\epsilon\conj\bE\times\bx$ and $\mu\conj\bH\times\bx$,
\begin{align}\label{eq:morid2}
&2\Re\Big\{\big(\nabla\times\bE-\ri\wn\mu\bH\big)\cdot(\epsilon\conj\bE\times\bx)
+\big(\nabla\times\bH+\ri\wn\epsilon\bE\big)\cdot(\mu\conj\bH\times\bx)\Big\}
\\ \nonumber
&=\nabla\cdot\Big[
2\Re\Big\{(\bE\cdot\bx)\epsilon\conj\bE+(\bH\cdot\bx)\mu\conj\bH\Big\}
-(\epsilon \bE\cdot\conj\bE)\bx 
-(\mu \bH\cdot\conj\bH)\bx 
\Big]\\ \nonumber
&\quad -2\Re\Big\{(\bE\cdot\bx)\nabla\cdot[\epsilon\conj\bE]+(\bH\cdot\bx)\nabla\cdot[\mu\conj\bH]\Big\}
+\big(\epsilon+(\bx\cdot\nabla)\epsilon\big) \bE\cdot\conj\bE+\big(\mu+(\bx\cdot\nabla)\mu\big)\bH\cdot\conj\bH
\end{align}
and the identity for the parts of the multipliers containing $\beta$
\beq\label{eq:morid3}
2\Re\Big\{\big(\nabla\times\bE-\ri\wn\mu\bH\big)\cdot(\beta\conj \bH)
-\big(\nabla\times\bH+\ri\wn\epsilon\bE\big)\cdot(\beta\conj\bE)\Big\}
=2\Re\Big\{\nabla\cdot[\beta\bE\times\conj\bH]-\nabla\beta\cdot\bE\times\conj\bH\Big\}.
\eeq
Since $\mu \bH\cdot \conj\bH$ and $\epsilon\bE\cdot\conj\bE$ are real, the left-hand side of \eqref{eq:morid3} equals
\begin{align*}
2\Re\Big\{\beta(\nabla\times\bE\cdot\conj\bH-\nabla\times\bH\cdot\conj\bE)\Big\}
&=2\Re\Big\{\beta(\nabla\times\bE\cdot\conj\bH-\nabla\times\conj\bH\cdot\bE)\Big\}\\
&=2\Re\Big\{\beta\nabla\cdot[\bE\times\conj\bH]\Big\}
\end{align*}
which equals the right-hand side of \eqref{eq:morid3}, and thus we only need to prove \eqref{eq:morid2}.

The left-hand side of \eqref{eq:morid2} equals
\begin{align*}
&2\Re\Big\{\big(\nabla\times\bE-\ri\wn\mu\bH\big)\cdot(\epsilon\conj\bE\times\bx)
+\big(\nabla\times\bH+\ri\wn\epsilon\bE\big)\cdot(\mu\conj\bH\times\bx)\Big\}\\
&=2\Re\Big\{\big(\nabla\times\bE\big)\cdot(\epsilon\conj\bE\times\bx)
+\big(\nabla\times\bH\big)\cdot(\mu\conj\bH\times\bx) +\ri\wn\big( - \mu\bH\cdot \epsilon\conj \bE \times \bx + \epsilon\bE \cdot \mu\conj \bH \times \bx\big)\Big\}\\
&=2\Re\Big\{\big(\nabla\times\bE\big)\cdot(\epsilon\conj\bE\times\bx)
+\big(\nabla\times\bH\big)\cdot(\mu\conj\bH\times\bx)\Big\}. 
\end{align*}
We now claim that for all $\bv\in C^2(D)^3$ and all $\alpha\in C^1(D,\SPD)$,
\begin{align}\nonumber
&2\Re\Big\{(\nabla\times\bv)\cdot(\alpha\conj\bv\times\bx)\Big\}\\
&= \nabla\cdot \Big[ 2 \Re \big\{ (\bv \cdot \bx) \alpha \conj \bv\big\} - (\alpha \bv\cdot\conj\bv) \bx\Big]
- 2 \Re\Big\{ (\bv\cdot\bx) \nabla\cdot[\alpha\conj\bv]\Big\}+  \big( \alpha + (\bx\cdot\nabla)\alpha\big)\bv\cdot\conj\bv.\label{eq:Euan1}
\end{align}
Summing \eqref{eq:Euan1} with $\bv=\bE$ and $\alpha=\epsilon$ to the same expression with $\bv=\bH$ and $\alpha=\mu$, we arrive at \eqref{eq:morid2}.
Therefore, we only need to show that \eqref{eq:Euan1} holds.
Proceeding in a similar way to that in the proof of \cite[Lemma~5.3.1]{AndreaPhD}, we find that standard vector calculus identities give
\begin{align}\nonumber
&2\Re\Big\{(\nabla\times\bv)\cdot(\alpha\conj\bv\times\bx)\Big\}\\ \nonumber
&=
2\Re\Big\{\nabla\cdot\Big[\bv\times(\alpha\conj\bv\times\bx)\Big]
+\bv\cdot\nabla\times(\alpha\conj\bv\times\bx)\Big\}\\ \nonumber
&=
2\Re\Big\{\nabla\cdot\Big[ (\bv\cdot\bx)\alpha\conj\bv - \big(\alpha\bv\cdot\conj\bv\big)\bx\Big]
+\bv\cdot\big(\alpha\conj\bv\nabla\cdot\bx-\bx\nabla\cdot[\alpha\conj\bv]
+(\bx\cdot\nabla)(\alpha\conj\bv)-(\alpha\conj\bv\cdot\nabla)\bx\big)\Big\}\!
\nonumber\\
&=
2\Re\Big\{\nabla\cdot\Big[ (\bv\cdot\bx)\alpha\conj\bv - \big(\alpha\bv\cdot\conj\bv\big)\bx\Big]
-\bv\cdot\bx\nabla\cdot[\alpha\conj\bv]+\bv\cdot(\bx\cdot\nabla)(\alpha\conj\bv)\Big\}
+ 4 \alpha\bv\cdot\conj\bv
\nonumber\\
&= 2\Re\Big\{ \nabla\cdot\Big[ (\bv\cdot\bx)\alpha\conj\bv - \big(\alpha\bv\cdot\conj\bv\big)\bx\Big]
-\bv\cdot\bx\nabla\cdot[\alpha\conj\bv]+ \alpha \bv \cdot (\bx\cdot\nabla)\conj\bv\Big\}+ 4 \alpha\bv\cdot\conj\bv
+ 2 \big( (\bx\cdot\nabla)\alpha\big)\bv\cdot\conj\bv.\label{eq:Euan2}
\end{align}
The identity
\beq\label{eq:Euan3}
2 \Re\Big\{\alpha \bv\cdot (\bx\cdot\nabla)\conj\bv\Big\} 
= \nabla\cdot \Big[ (\alpha\bv\cdot\conj\bv)\bx\Big]  -3(\alpha\bv\cdot\conj\bv) - \big( (\bx\cdot\nabla)\alpha\big)\bv\cdot\conj\bv
\eeq
can be proved by expanding in components the divergence on its right-hand side.
Using \eqref{eq:Euan3} in \eqref{eq:Euan2} we find \eqref{eq:Euan1}, and the proof is complete.
\epf

\bpf[Proof of Lemma \ref{lem:Morawetz2}]
The identity \eqref{eq:morid1a} will follow from \eqref{eq:morid1} with $\beta=\sqrt{\epsilon_0 \mu_0} \, r$, $\epsilon=\epsilon_0$, and $\mu=\mu_0$ if we can show that
\begin{align}\nonumber
-2 \sqrt{\epsilon_0 \mu_0}\, \Re\big\{ \hx \cdot \bE \times \conj \bH\big\} =&
\frac{1}{2\mu_0} \big| \mu_0 \bH \times \hx - \sqrt{\epsilon_0 \mu_0}\, \bE\big|^2 + \frac{1}{2\epsilon_0}\big| \epsilon_0 \hx \times \bE - \sqrt{\epsilon_0 \mu_0} \,\bH\big|^2\\
&-\frac{\epsilon_0}{2}\big( |\bE|^2 + |\bE\times \hx|^2\big)
-\frac{\mu_0}{2}\big( |\bH|^2 + |\bH\times \hx|^2\big).
\label{eq:morid4}
\end{align}
Using $-2 \Re\{ \bb \cdot \conj\ba\}= |\ba-\bb|^2 - |\ba|^2 - |\bb|^2$ with $\ba= \mu_0  \bH \times \hx$ and $\bb= \sqrt{\epsilon_0 \mu_0} \bE$, we have
\begin{align}\label{eq:morid5}
-2 \sqrt{\epsilon_0 \mu_0}\, \Re\big\{ \hx \cdot \bE \times \conj \bH\big\} &=
-2 \sqrt{\epsilon_0 \mu_0}\, \Re\big\{ \bE\cdot \conj \bH\times \hx\big\} 
\nonumber\\&=
\frac{1}{\mu_0} \big| \mu_0 \bH \times \hx - \sqrt{\epsilon_0 \mu_0}\, \bE\big|^2- \mu_0 \big| \bH \times \hx\big|^2 - \epsilon_0 \big| \bE\big|^2.
\end{align}
Similarly
\begin{align}\label{eq:morid6}
-2 \sqrt{\epsilon_0 \mu_0}\, \Re\big\{ \hx \cdot \bE \times \conj \bH\big\} &=
-2 \sqrt{\epsilon_0 \mu_0}\, \Re\big\{ \conj\bH\cdot \hx\times \bE\big\} 
\nonumber\\&=
\frac{1}{\epsilon_0} \big| \epsilon_0 \hx \times \bE - \sqrt{\epsilon_0 \mu_0}\, \bH\big|^2- \epsilon_0 \big| \hx \times \bE\big|^2 - \mu_0 \big| \bH\big|^2.
\end{align}
Adding \eqref{eq:morid5} and \eqref{eq:morid6} and dividing by two, we find \eqref{eq:morid4} and the proof is complete.
\epf

To separate tangential and normal traces on boundaries we use the following identity.
\begin{lemma}\label{lem:NormalTangentMat}
Given $\bv\in\IC^3$, $\alpha\in\SPD$ and $\hn\in\IR^3$ with $|\hn|=1$, 
\begin{align}\nonumber
&2\Re\big\{(\bv\cdot\bx)(\alpha\conj\bv\cdot\hn)\big\}-(\alpha\bv\cdot\conj\bv)\bx\cdot\hn\\
&=(\alpha\bv_N\cdot\conj\bv_N-\alpha\bv_T\cdot\conj\bv_T)\bx\cdot\hn+2\Re\big\{(\bv_T\cdot\bx_T)(\alpha\conj\bv\cdot\hn)\big\},
\label{eq:NormalTangentMat}
\end{align}
where  $\bv_N:=(\bv\cdot\hn)\hn$ and $\bv_T:=\bv-\bv_N$.
\end{lemma}
\begin{proof}
By the symmetry of $\alpha$ and the decomposition $\bv=\bv_T+\bv_N$,
\begin{align*}
2&\Re\big\{(\bv\cdot\bx)(\alpha\conj\bv\cdot\hn)\big\}-(\alpha\bv\cdot\conj\bv)(\bx\cdot\hn)
\\
=&2\Re\big\{(\bv_T\cdot\bx_T+\bv_N\cdot\bx_N)(\alpha\conj\bv_T\cdot\hn+\alpha\bv_N\cdot\hn)\big\}
-\alpha(\bv_T+\bv_N)\cdot(\conj\bv_T+\conj\bv_N)(\bx\cdot\hn)\\
=&2\Re\big\{(\bv_T\cdot\bx_T)(\alpha\conj\bv\cdot\hn)  
+ \underbrace{(\bv_N\cdot\bx_N)(\alpha\conj\bv_T\cdot\hn)}_{\substack{=(\bv\cdot\hn)(\bx\cdot\hn)(\alpha\conj\bv_T\cdot\hn)\\=(\bx\cdot\hn)(\alpha\conj\bv_T\cdot\bv_N)}}
+ \underbrace{(\bv_N\cdot\bx_N)(\alpha\conj\bv_N\cdot\hn)}_{\substack{=(\bv\cdot\hn)(\bx\cdot\hn)(\alpha\conj\bv_N\cdot\hn)\\=(\bx\cdot\hn)(\alpha\conj\bv_N\cdot\bv_N)}}
  \big\}\\
&-\big(\alpha\bv_T\cdot\conj\bv_T +
\alpha\bv_N\cdot\conj\bv_N 
+2\Re\{\alpha\bv_N\cdot\conj\bv_T\}\big)(\bx\cdot\hn),
\end{align*}
and the result follows.
\end{proof}

\subsection{The identities in integrated form}

Our next result is an integrated version of the identity \eqref{eq:morid1}. 
To state this result it is convenient to define the space 
\beq\label{eq:V}
V(D,\MA):=\bigg\{ \bv\in H(\curl; D):\; \nabla\cdot[\MA\bv]\in L^2(D),\;\MA\bv\cdot\hn \in L^2 (\partial D),\; \bv_T \in L^2_T (\partial D)\bigg\}
\eeq
where $D$ is a bounded Lipschitz open set with outward-pointing unit normal vector $\hn$ and $\MA\in W^{1,\infty}(D,\SPD)$. 
We make three remarks about this space.
\bit
\item 
In Appendix \ref{app:density} we prove that $C^\infty(\conj D)^3$ is dense in $V(D,\MA)$ when $\MA\in C^1(\overline{D},\SPD)$;
this extends the density result of \cite[Theorem~1]{CoD98} to $\epsilon,\mu\ne\mymatrix{I}$.

\item At least in the case $\MA=\MI$, either the condition  $\MA\bv\cdot\hn \in L^2(\partial D)$ or the condition $\bv_T\in L^2_T(\partial D)$ can be removed from the definition of $V(D,\MA)$ by \cite[Theorem~1]{CoD98}.
\item If $\MA$ satisfies $\lambda_{\min}\preceq\MA(\bx)\preceq\lambda_{\max}$ for scalar $0<\lambda_{\min}\le \lambda_{\max}$ and almost every $\bx\in\partial D$, then the normal trace of $\bv\in V(D,\MA)$ is bounded in $L^2(\partial D)$. Indeed
$$
|\bv\cdot\hn|=|\bv_N|
\le \frac{\MA\bv_N\cdot\conj\bv_N}{\lambda_{\min}|\bv_N|}
= \frac{\MA(\bv-\bv_T)\cdot\conj\bv_N}{\lambda_{\min}|\bv_N|}
\leq\frac{|\MA\bv\cdot\hn|}{\lambda_{\min}}+\frac{|\MA\bv_T\cdot\conj\bv_N|}{\lambda_{\min}|\bv_N|}
\le \frac{|\MA\bv\cdot\hn|}{\lambda_{\min}} + \frac{\lambda_{\max}|\bv_T|}{\lambda_{\min}},
$$
and thus 
$\N{\bv\cdot\hn}_{\partial D}\le \lambda_{\min}^{-1} \N{\MA\bv\cdot\hn}_{\partial D}+\lambda_{\min}^{-1} \lambda_{\max} \N{\bv_T}_{\partial D}<\infty$.
\eit

\ble\mythmname{Integrated form of the Morawetz identity \eqref{eq:morid1}}\label{lem:IntegratedM}
Let $D$ be a bounded Lipschitz open set with outward-pointing unit normal vector $\hn$. 
If $\beta \in C^1(\conj D)$ and \emph{either}

(i) $\bE\in V(D,\epsilon)$, $\bH \in V(D,\mu)$, $\epsilon, \mu \in C^1(\overline{D}, \SPD)$, 
\emph{or}

(ii) $\bE, \bH \in H^1(D)^3$, $\epsilon, \mu \in W^{1,\infty}(D, \SPD)$,
then
\begin{align}\nonumber
&2\int_D\Re\Big\{\big(\nabla\times\bE-\ri\wn\mu\bH\big)\cdot(\epsilon\conj\bE\times\bx+\beta\conj \bH)
+\big(\nabla\times\bH+\ri\wn\epsilon\bE\big)\cdot(\mu\conj\bH\times\bx-\beta\conj\bE)\Big\}
\\ \nonumber
&\qquad +\int_D 2\Re\Big\{
(\bE\cdot\bx)\nabla\cdot[\epsilon\conj\bE]+(\bH\cdot\bx)\nabla\cdot[\mu\conj\bH]
+\nabla\beta\cdot\bE\times\conj\bH
\Big\}
\\&\qquad\qquad\nonumber
-\big(\epsilon+(\bx\cdot\nabla)\epsilon\big)\bE\cdot\conj\bE-\big(\mu+(\bx\cdot\nabla)\mu\big)\bH\cdot\conj\bH\\ \nonumber
&=\int_{\partial D}
\left( \epsilon\bE_N\cdot\conj\bE_N - \epsilon\bE_T\cdot\conj\bE_T 
+ \mu\bH_N\cdot\conj\bH_N - \mu\bH_T\cdot\conj\bH_T\right) (\bx\cdot\hn) \\ 
&\qquad\qquad+ 
2\Re\bigg\{(\bE_T\cdot\bx_T) ( \epsilon\conj\bE\cdot\hn) +(\bH_T\cdot\bx_T) (\mu\conj\bH\cdot\hn)
+ \beta \bE_T\cdot \conj\bH_T\times \hn\bigg\}.
\label{eq:intmor}
\end{align}
\ele
\bpf
Recall that the divergence theorem $\int_D \nabla \cdot \textbf{F}
= \int_{\partial D} \textbf{F} \cdot \hn$
is valid when $\textbf{F} \in C^1(\overline{D})^3$ \cite[Theorem 3.34]{MCL00}, and thus for $\bF\in H^1(D)^3$ by the density of $C^1(\overline{D})^3$ in $H^1(D)^3$ \cite[Theorem~3.29]{MCL00} and the continuity of trace operator from $H^1(D)$ to $H^{1/2}(\partial D)$ \cite[Theorem 3.37]{MCL00}.
We first assume that $\epsilon, \mu$, and $\beta$ are as in the statement of the theorem, but $\bE, \bH\in C^\infty(\conj D)^3$ 

Recall that the product of an $H^1(D)$ function and a $W^{1,\infty}(D)$ function is in $H^1(D)$,
and the usual product rule for differentiation holds for such functions. 
This result implies that $\bQ_\beta$ \eqref{eq:Qbeta} is in $H^1(D)^{3}$ and then \eqref{eq:morid1} implies that $\nabla\cdot\bQ_\beta$ is given by the integrand on the left-hand side of \eqref{eq:intmor}.
The divergence theorem then implies that \eqref{eq:intmor} holds, where the fact that $\bQ_\beta\cdot \hn$ equals the integrand on the right-hand side of \eqref{eq:intmor} follows from the identity \eqref{eq:NormalTangentMat}
and the identity $\bE\times \conj\bH\cdot \hn= \bE\cdot\conj\bH \times \hn  =  \bE_T\cdot (\conj\bH_T\times \hn)$.

For (i), when $\epsilon, \mu \in C^1(\overline{D}, \SPD)$, $C^\infty(\conj D)^3$ is dense in $V(D,\epsilon)$ and $V(D,\mu)$ by Lemma \ref{thm:Density}; 
the result then follows since \eqref{eq:intmor} is continuous in $\bE$ and $\bH$ with respect to the topologies of $V(D,\epsilon)$ and $V(D,\mu)$.

For (ii), $C^\infty(\conj D)^3$ is dense in $H^1(D)^3$ and the result then follows since \eqref{eq:intmor} is continuous in $\bE$ and $\bH$ with respect to the topology of $H^1(D)^3$.
\epf

Part (i) of Lemma \ref{thm:Density} also completes the proof of \cite[Lemma 6.1]{Ve:19}, where this density was assumed. 

\subsection{Dealing with the contribution from infinity using the Silver--M\"uller radiation condition}

The next result (Lemma \ref{lem:infinity}) uses the Silver--M\"uller radiation condition \eqref{eq:intro_SM} to deal with the term on $\partial B_R$ when 
the Morawetz-type identity is integrated over $B_R$ and $\bE, \bH$ satisfy the homogeneous Maxwell equations outside $B_R$.

\ble\mythmname{Inequality on $\partial B_R$ used to deal with the contribution from infinity in the transmission problem}\label{lem:infinity} 
Let $\epsilon_0$ and $\mu_0$ be positive constants, and let $\bE,\bH$ be a solution of the homogeneous Maxwell equations 
\begin{align}\label{eq:MaxwellHomog}
\ri\wn\epsilon_0\bE+\nabla\times\bH=\bzero, \qquad
-\ri\wn\mu_0\bH+\nabla\times\bE=\bzero
\end{align}
in $\Rea^d\setminus \overline{B_{R_0}}$, for some $R_0>0$ satisfying the Silver-M\"uller radiation condition \eqref{eq:intro_SM}.
Let $\bQ_\beta$ be defined by \eqref{eq:Qbeta}.
Then, for $R\geq R_0$,
\beq\label{eq:ML}
\int_{\partial B_R} \bQ_{R\sqrt{\epsilon_0 \mu_0}} \cdot \hx \geq 0.
\eeq
\ele
\bpf
Since $\bQ_{R\sqrt{\epsilon_0 \mu_0}}=\bQ_{r\sqrt{\epsilon_0 \mu_0}}$ on $\partial B_R$,  the plan is to integrate \eqref{eq:morid1a} over $B_{R_1}\setminus \overline{B_R}$, use the divergence theorem and let $R_1\tendi$. Note that using the divergence theorem is allowed since both $\bE$ and $\bH$ are $C^\infty$ in $\Rea^d\setminus \overline{B_{R_0}}$ by elliptic regularity.

When $\bE$ and $\bH$ satisfy \eqref{eq:MaxwellHomog}, the identity \eqref{eq:morid1a} becomes 
\beqs
0 = \nabla\cdot \bQ_{r\sqrt{\epsilon_0 \mu_0}} + \widetilde{P}_{r\sqrt{\epsilon_0 \mu_0}},
\eeqs
where 
\begin{align*}
\widetilde{P}_{r\sqrt{\epsilon_0 \mu_0}}
:=&\frac{\epsilon_0}{2}\big( |\bE|^2 - |\bE\times \hx|^2\big)
+\frac{\mu_0}{2}\big( |\bH|^2 - |\bH\times \hx|^2\big)\\& \qquad
+\frac{1}{2\mu_0} \big| \mu_0 \bH \times \hx - \sqrt{\epsilon_0 \mu_0} \bE\big|^2 + \frac{1}{2\epsilon_0}\big| \epsilon_0 \hx \times \bE - \sqrt{\epsilon_0 \mu_0} \bH\big|^2\geq 0.
\end{align*}
Integrating \eqref{eq:morid1a} over $B_{R_1}\setminus \overline{B_R}$ and using the divergence theorem, we obtain
\beqs
0 =- \int_{\partial B_R} \bQ_{r\sqrt{\epsilon_0 \mu_0}} \cdot\hx+ \int_{\partial B_{R_1}}\bQ_{r\sqrt{\epsilon_0 \mu_0}}\cdot\hx + \int_{B_{R_1}\setminus B_R} \widetilde{P}_{r\sqrt{\epsilon_0 \mu_0}}.
\eeqs
If we can establish that $\int_{\partial B_{R_1}}\bQ_{r\sqrt{\epsilon_0 \mu_0}} \cdot\hx\tendo$ as $R_1\tendi$, the result follows since then
\begin{align*}
\int_{\partial B_R} \bQ_{r\sqrt{\epz\muz}}\cdot\hx
=\lim_{R_1\to \infty}\int_{\partial B_{R_1}} \bQ_{r\sqrt{\epz\muz}}\cdot\hx+\lim_{R_1\to \infty}&\int_{B_{R_1}\setminus \overline{B_R}}\widetilde{P}_{r\sqrt{\epz\muz}}
\geq0.
\end{align*}

We therefore only need to prove that $|\bQ_{r\sqrt{\epz\muz}}\cdot\hx|=o_{R\to\infty}(R^{-2})$ uniformly in $\hx$ so that $\lim_{R\to\infty}|\int_{\partial B_R}\bQ_{r\sqrt{\epz\muz}}\cdot\hx|=0$. On $\partial B_R$, $\hn = \hx$, and therefore $\bE_N = (\bE\cdot \hx)\hx$ and $\bE_T = (\hx\times \bE)\times \hx$. 
By the definition of $\bQ_{r\sqrt{\epz\muz}}$ \eqref{eq:Qbeta} and the identity \eqref{eq:NormalTangentMat}, 
\begin{align}\label{eq:Pavia1}
\bQ_{r\sqrt{\epz\muz}}\cdot\hx
=&
 r\epz|\bE_N|^2 + r\muz|\bH_N|^2
-r\epz|\bE_T|^2 - r\muz|\bH_T|^2
+2r\sqrt{\epz\muz}\Re\Big\{\,\conj\bH\times\hx\cdot\bE\Big\}\\ \nonumber
=& r\epz|\bE_N|^2 + r\muz|\bH_N|^2
-r\Big(\epz|\bE_T|^2 +\muz|\bH_T|^2
-2\sqrt{\epz\muz}\Re\Big\{\conj\bH_T\times\hx\cdot\bE_T\Big\}\Big)\\ \nonumber
=& r\epz|\bE_N|^2 + r\muz|\bH_N|^2
- r\big|\sqrt\epz\,\bE_T-\sqrt\muz\, \bH_T\times\hx\big|^2.
\end{align}
Taking the tangential and normal components of the Silver--M\"uller radiation conditions \eqref{eq:intro_SM}, we have 
$|\sqrt\epz\,\bE_T-\sqrt\muz\, \bH_T\times\hx| = \cO_{r\tendi}(r^{-2})$, $|\bE_N|=\cO_{r\tendi}(r^{-2})$,  and $|\bH_N|=\cO_{r\tendi}(r^{-2})$ (uniformly in $\hx$). Therefore,  $|\bQ_{r\sqrt{\epz\muz}}\cdot\hx|=\cO_{R\to\infty}(R^{-3})$ uniformly in $\hx$ and the proof is complete. 
\epf

The Helmholtz analogue of Lemma \ref{lem:infinity} first appeared implicitly in \cite{MoLu:68, Mo:75}, 
and first appeared explicitly in \cite[Lemma 2.1]{CWM08}.

%%%%%%%%%%%%%%%%%%%%%%%%%%%%%%%%%%%%%%%%%%%%%%%%%%%%%%%%%%%%%%%%%%%%%%%%%%%%%%%%%%%%%%%%%%%%
\section{Proofs of Theorems \ref{thm:BoundSmooth} and \ref{thm:BoundRough} and Corollary \ref{cor:scat}}\label{sec:Proofs}

\subsection{Proof of Theorem \ref{thm:BoundSmooth}}
\label{subsec:Proofs1}

\ble\label{lem:E1}
Let $\domain\subset\IR^d$ be a bounded Lipschitz domain, 
$\MM_j \in L^\infty(\domain;\SPD)$, 
$\bv_j,\bw_j\in L^2(\domain;\IC^m)$, $j=1,2$,  $\Xi\in\IR$, and $m_1,m_2>0$.
If 
\beqs
m_1 \N{\bv_1}^2_{L^2(\domain;\MM_1)} + m_2 \N{\bv_2}^2_{L^2(\domain:\MM_2)}
\leq \Re \int_{\domain}\MM_1 \bw_1 \cdot\conj\bv_1 + \Re \int_{\domain} \MM_2 \bw_2\cdot\conj \bv_2+\Xi,
\eeqs
then
\beq\label{eq:Elem2}
m_1 \N{\bv_1}^2_{L^2(\domain;\MM_1)} + m_2 \N{\bv_2}^2_{L^2(\domain;\MM_2)}
\leq m_1^{-1} \N{\bw_1}^2_{L^2(\domain;\MM_1)} + m_2^{-1} \N{\bw_2}^2_{L^2(\domain;\MM_2)}+ 2 \Xi.
\eeq
\ele
\bpf
Since $\MM_j\in\SPD$, $j=1,2$, one can define $\MM_j^{1/2}\in\SPD$ which 
satisfies $\MM_j=\MM_j^{1/2}\MM_j^{1/2}$. 
Then, by the Cauchy--Schwarz inequality and the inequality $2ab\leq \eps^{-1} a^2 + \eps b^2 $ for all $a,b,\eps>0$,
\begin{align*}\nonumber
\Re\int_{\domain}\MM_j \bw_j\cdot\conj\bv_j
= \Re\int_{\domain} \MM_j^{1/2}\bw_j \cdot \MM_j^{1/2}\conj\bv_j
&\leq \big\| \MM_j^{1/2}\bw_j \big\|_{L^2(\domain)} 
 \big\| \MM_j^{1/2}\bv_j \big\|_{L^2(\domain)} \\
 &\leq 
 \frac{1}{2m_j}  \big\| \MM_j^{1/2}\bw_j \big\|_{L^2(\domain)}^2
+ \frac{m_j}{2} \big\| \MM_j^{1/2}\bv_j \big\|_{L^2(\domain)}^2.
\end{align*}
The result \eqref{eq:Elem2} then follows from the definition of the weighted norms \eqref{eq:weighted_norms}.
\epf

We now prove an intermediate result, bounding the solution of the first-order system with right-hand sides in $H(\div;B_R)$

\begin{lemma}
\mythmname{Bound on first-order system with $W^{1,\infty}$ coefficients and right-hand sides in $H(\div;B_R)$}
\label{lem:BoundSmooth_old}
Suppose that, in addition to the set up in \S\ref{sec:set_up},
$\epsilon,\mu\in W^{1,\infty}(\IR^3,\SPD)$ and the conditions \eqref{eq:GrowthCoeff} hold. 
Suppose $\bJ, \bK \in H(\div;B_R)$ with compact support and fix $R>0$ such that $\supp(\mu-\mu_0\MI)\cup \supp(\epsilon-\epsilon_0\MI)\cup\esssupp\bJ\cup\esssupp\bK\subset B_R$.
Then the solution $\bE,\bH\in H_{\rm loc}(\curl;\Rea^3)$ of \eqref{eq:first_order}--\eqref{eq:intro_SM} 
satisfies
\begin{align}\nonumber
&\hspace{0cm}\epsilonconstant\wnorm{\epsilon}{B_R}{\bE}^2 + 
\muconstant
\wnorm{\mu}{B_R}{\bH}^2\\ \nonumber
&\hspace{1cm}
\le
4R^2
\bigg[
\frac{1}{\epsilonconstant}
\left (
\wnorm{\epsilon}{B_R}{\bK}
+
\sqrt{\epsilon_0\mu_0}
\wnorm{\epsilon^{-1}}{B_R}{\bJ}
+
\frac{
1}{\wn \epsmin^{1/2}}
\|\nabla \cdot \bJ\|_{B_R}
\right )^2
\\
&\hspace{2.5cm}+
\frac{1}{\muconstant}
\left (
\wnorm{\mu}{B_R}{\bJ}
+
\sqrt{\epsilon_0\mu_0}
\wnorm{\mu^{-1}}{B_R}{\bK}
+
\frac{
1}{\wn \mumin^{1/2}}
\|\nabla \cdot \bK\|_{B_R}
\right )^2\bigg].
\label{eq:thm:BoundSmooth_weighted}
\end{align}
\end{lemma}

\begin{proof}
The Maxwell equations \eqref{eq:first_order} and the conditions $\bJ,\bK\in H(\dive;\IR^3)$ give 
$\nabla\cdot[\epsilon\bE],\nabla\cdot[\mu\bH]\in L^2(\IR^3)$.
We now show that $\bE, \bH \in H^1(B_R)^3$ (aiming to apply the integrated Morawetz identity \eqref{eq:intmor} with $D=B_R$ using Part (ii) of Lemma \ref{lem:IntegratedM}).
Indeed
by the regular decomposition lemma for functions in $H(\curl;B_{2R})$ (see, e.g., \cite[Lemma 2.4]{HiptmairActa}) there exist $\bz \in H^1(B_{2R})^3
$
 and $\phi \in H^1(B_{2R})$ such that $\bE= \bz + \nabla \phi$. 
The condition $\nabla\cdot[\epsilon\bE]\in L^2(\IR^3)$ implies that 
\beqs
\nabla\cdot [\epsilon \nabla\phi]
=  \nabla\cdot[\epsilon\bE]-\nabla\cdot [\epsilon \bz]
\eeqs
is in $L^2(B_{2R})$ (since $\bz\in H^1(B_{2R})^{ 3}$ and $\epsilon\in W^{1,\infty}(\Rea^3,\SPD)$). 
By interior regularity of the operator $(\nabla\cdot [\epsilon \nabla])^{-1}$ (see, e.g., \cite[Theorem 4.16]{MCL00}), and using that $\epsilon\in W^{1,\infty}(B_{2R},\SPD)= C^{0,1}(\overline{B_{2R}},\SPD)$, we have $\phi \in H^2_{\loc}(B_{2R})$ and so in particular in $H^2(B_R)$; therefore $\bE \in H^1(B_R)^{ 3}$.
Identical arguments show that $\bH \in H^1(B_R)^{ 3}$.

Applying the integrated Morawetz identity \eqref{eq:intmor} with $D=B_R$ is justified by Part (ii) of Lemma \ref{lem:IntegratedM}, and then
substituting in $\nabla\cdot[\epsilon\bE]=(\ri\wn)^{-1}\nabla\cdot\bJ$, $\nabla\cdot[\mu\bH]=(-\ri \wn)^{-1}\nabla\cdot\bK$ and \eqref{eq:first_order}, and recalling that $\epsilon = \epsilon_0\MI$ and $\mu = \mu_0\MI$ on $\partial B_R$, we obtain that, for all $\beta\in C^1(\conj{B_R})$,
\begin{align}\nonumber
&
\gepsilon \wnorm{\epsilon}{B_R}{\bE}^2 + \gmu \wnorm{\mu}{B_R}{\bH}^2
\\ \nonumber
&\leq
\int_{B_R}
\big(\epsilon+(\bx\cdot\nabla)\epsilon\big)\bE\cdot\conj\bE+\big(\mu+(\bx\cdot\nabla)\mu\big)\bH\cdot\conj\bH
\\ \nonumber
&=2\int_{B_R}\!\!\Re\bigg\{\bK\cdot(\epsilon\conj\bE\times\bx+\beta\conj \bH)
+\bJ\cdot(\mu\conj\bH\times\bx-\beta\conj\bE)
+ (\conj\bE\cdot\bx)\frac{\nabla\cdot\bJ}{\ri\wn}
-(\conj\bH\cdot\bx)\frac{\nabla\cdot\bK}{\ri\wn}
+\nabla\beta\cdot\bE\times\conj\bH\bigg\}\hspace{-1mm}
\\ \nonumber
&\quad-\int_{\partial B_R}
\left( \epsilon_0\big|\bE_N\big|^2 - \epsilon_0\big|\bE_T\big|^2 
+ \mu_0\big|\bH_N\big|^2 - \mu_0\big|\bH_T\big|^2\right) (\bx\cdot\hn) 
\\ \label{eq:EuanTemp1new}
&\quad
-\int_{\partial B_R} 2\Re\bigg\{ (\bE_T\cdot\bx_T) (\epsilon_0\conj\bE\cdot\hn) +(\bH_T\cdot\bx_T) (\mu_0\conj\bH\cdot\hn)
+ \beta \bE_T\cdot \conj\bH_T\times \hn\bigg\}.
\end{align}
The choice $\beta=R\sqrt{\epsilon_0 \mu_0}$ implies both that $\nabla\beta=\bzero$ and that the terms integrated over $\deB_R$ on the right-hand side of \eqref{eq:EuanTemp1new} equal $-\bQ_{R\sqrt{\epsilon_0 \mu_0}}$ \eqref{eq:Pavia1}, thanks to \eqref{eq:NormalTangentMat}; therefore, using the inequality \eqref{eq:ML}, \eqref{eq:EuanTemp1new} becomes
\begin{align}\nonumber
&
\gepsilon \wnorm{\epsilon}{B_R}{\bE}^2 + \gmu \wnorm{\mu}{B_R}{\bH}^2
\leq
\int_{B_R}
\big(\epsilon+(\bx\cdot\nabla)\epsilon\big)\bE\cdot\conj\bE+\big(\mu+(\bx\cdot\nabla)\mu\big)\bH\cdot\conj\bH
\\ \nonumber
&\le
2\int_{B_R}\!\Re\bigg\{\bK\cdot(\epsilon\conj\bE\times\bx+R\sqrt{\epsilon_0 \mu_0}\conj \bH)
+\bJ\cdot(\mu\conj\bH\times\bx-R\sqrt{\epsilon_0 \mu_0}\conj\bE)
+(\conj\bE\cdot\bx)\frac{\nabla\cdot\bJ}{\ri\wn}
-(\conj\bH\cdot\bx)\frac{\nabla\cdot\bK}{\ri\wn}
\bigg\}
\\
&
=
2\Re \left \{
\int_{B_R}
\epsilon \left (
\bx \times \bK
-R\sqrt{\epsilon_0 \mu_0}\epsilon^{-1}\bJ
+
\epsilon^{-1}\bx \frac{\nabla\cdot\bJ}{\ri\wn}
\right )
\cdot
\conj\bE
\right \}
\nonumber
\\
&
\hspace{2cm}+
2\Re \left \{
\int_{B_R}
\mu
\left (
\bx \times \bJ
+
R\sqrt{\epsilon_0 \mu_0}\mu^{-1} \bK
-
\mu^{-1}\bx \frac{\nabla\cdot\bK}{\ri\wn}
\right )
\cdot
\conj\bH
\right \}.\label{eq:diverge1}
\end{align}
By Lemma \ref{lem:E1} with $\bv_1 =\bE, \bv_2=\bH$, $\MM_1=\epsilon, \MM_2=\mu$, $m_1=\epsilonconstant$, $m_2=\muconstant$, and $\Xi=0$,
\beqs
\begin{aligned}
\gepsilon \wnorm{\epsilon}{B_R}{\bE}^2
+
\gepsilon \wnorm{\mu}{B_R}{\bH}^2
&\leq
\frac{4}{\gepsilon}
\wnorm{\epsilon}{B_R}
{
\bx \times \bK
-
R\sqrt{\epsilon_0 \mu_0}\epsilon^{-1}\bJ
+
\epsilon^{-1}\bx \frac{\nabla\cdot\bJ}{\ri\wn}
}^2
\\
&\qquad+
\frac{4}{\gmu}
\wnorm{\mu}{B_R}
{
\bx \times \bJ
+
R\sqrt{\epsilon_0 \mu_0}\mu^{-1} \bK
-
\mu^{-1}\bx \frac{\nabla\cdot\bK}{\ri\wn}
}^2
\end{aligned}
\eeqs
and the result \eqref{eq:thm:BoundSmooth_weighted} follows.%
\end{proof}

Theorem \ref{thm:BoundSmooth} follows from combining Lemma \ref{lem:BoundSmooth_old} with the following lemma.

\begin{lemma}\mythmname{Bound with right-hand side in $H(\div;B_R)$ implies bound with right-hand side in $L^2(B_R)$}
\label{lem:HdivtoL2}
Suppose that $\epsilon,\mu$ are such that the solution of  \eqref{eq:first_order}--\eqref{eq:intro_SM} with $\bJ,\bK \in H(\div;B_R)$ exists, is unique, and satisfies the bound \eqref{eq:thm:BoundSmooth_weighted}.
Then the solution to \eqref{eq:first_order}--\eqref{eq:intro_SM} with $\BJ, \bK\in \BL^2(B_R)$ exists, is unique, and satisfies 
\eqref{eq:thm:BoundSmooth}.
\end{lemma}

\begin{proof}
Let $\bJ,\bK\in L^2(B_R)$. Let $p,q\in H^1_0(B_R)$ be the unique solutions of the variational problems
\beq\label{eq:BVPpq}
\ri \wn (\epsilon\nabla p, \nabla v)_{L^2(B_R)} = (\bJ, \nabla v )_{L^2(B_R)} \quad\tand\quad
\ri \wn (\mu\nabla q, \nabla v)_{L^2(B_R)} = (\bK, \nabla v )_{L^2(B_R)} 
\eeq
for all $v\in H^1_0(B_R)$; i.e.,
$\ri\omega\nabla p$ is the orthogonal projection of
$\epsilon^{-1}\bJ$ in the $(\epsilon \cdot,\cdot)_{L^2(B_R)}$ inner product, so that 
\begin{align}\nonumber
\omega^2 \wnorm{\epsilon}{B_R}{\nabla p}^2
+
\wnorm{\epsilon^{-1}}{B_R}{\bJ-\ri \omega\epsilon\nabla p}^2
&=
\omega^2 \wnorm{\epsilon}{B_R}{\nabla p}^2
+
\wnorm{\epsilon}{B_R}{\epsilon^{-1} \bJ-\ri \omega \nabla p}^2\\
&=
\wnorm{\epsilon}{B_R}{\epsilon^{-1} \bJ}^2
=
\wnorm{\epsilon^{-1}}{B_R}{\bJ}^2\label{eq:orthog1}
\end{align}
and similarly
\begin{align}
\omega^2 \wnorm{\mu}{B_R}{\nabla q}^2
+
\wnorm{\mu^{-1}}{B_R}{\bK-\ri \omega\mu\nabla q}^2
=
\wnorm{\mu^{-1}}{B_R}{\bK}^2.\label{eq:orthog2}
\end{align}
Extend $p$ and $q$ to functions on $\Rea^3$ via extension by zero.
Then $\nabla p, \nabla q \in H(\curl;B_R)$ and $\nabla p,\nabla q \in H_{\rm loc}(\curl; \Rea^3\setminus\overline{B_R})$. 
By the definition of the weak derivative and integration by parts, a piecewise $H(\curl)$ function is in $H(\curl)$ if and only if  its tangential trace is continuous across the relevant interface. Since $(\nabla p)_T$ is the surface gradient of the trace of $p$ (which is zero), $\nabla p\in H_{\rm loc}(\curl, \Rea^3)$, and similarly for $\nabla q$. 

Therefore $\bE-\nabla p$, $\bH- \nabla q \in H_{\rm loc}(\curl; \Rea^3)$ 
and satisfy the Silver--M\"uller radiation condition \eqref{eq:intro_SM}. Furthermore,
\beqs
\ri \omega\epsilon(\bE- \nabla p) + \nabla\times (\bH- \nabla q) = \bJ - \ri \omega\epsilon \nabla p 
\eeqs
and
\beqs
-\ri \omega\mu(\bH- \nabla q) + \nabla\times (\bE- \nabla p) = \bK - \ri \omega\mu\nabla q.
\eeqs
By \eqref{eq:BVPpq},
$\nabla\cdot[\bJ-\ri\wn\epsilon\nabla p]=\nabla\cdot[\bK-\ri\wn\mu\nabla q]=0$; therefore \eqref{eq:thm:BoundSmooth_weighted} applies to $(\bE-\nabla p,\bH-\nabla q)$ and
\begin{align*}\nonumber
&\epsilonconstant\N{\bE-\nabla p }^2_{L^2(B_R;\epsilon)} + \muconstant \N{\bH -\nabla q}^2_{L^2(B_R;\mu)} 
\\ \nonumber
&\hspace{2cm}\le
\frac{4 R^2}{\epsilonconstant} \bigg(
\N{\bK - \ri \omega\mu \nabla q}_{L^2(B_R;\epsilon)} 
+\sqrt{\epsilon_0 \mu_0}\N{\bJ - \ri \omega\epsilon \nabla p}_{L^2(B_R;\epsilon^{-1})}
\bigg)^2\\
&\hspace{3cm}+\frac{4R^2}{\muconstant} \bigg(
\N{\bJ - \ri \omega\epsilon \nabla p}_{L^2(B_R;\mu)}
+\sqrt{\epsilon_0 \mu_0}\N{\bK - \ri \omega\mu \nabla q}_{L^2(B_R;\mu^{-1})}
\bigg)^2.
\end{align*}
Combining this with 
\begin{align*}
\epsilonconstant \N{\bE}^2_{L^2(B_R;\epsilon)} &\leq 2\epsilonconstant
 \Big(\N{\bE-\nabla p}^2_{L^2(B_R;\epsilon)} + \N{\nabla p}^2_{L^2(B_R;\epsilon)}\Big),
 \end{align*}
(and the analogous inequality for $\bH$), we obtain
\begin{align}\nonumber
\epsilonconstant\N{\bE }^2_{L^2(B_R)} + \muconstant\N{\bH}^2_{L^2(B_R)} 
&\le
\frac{8 R^2}{\epsilonconstant} \Big(
\N{\bK - \ri \omega\mu \nabla q}_{L^2(B_R;\epsilon)} 
+\sqrt{\epsilon_0 \mu_0}\N{\bJ - \ri \omega\epsilon \nabla p}_{L^2(B_R;\epsilon^{-1})}
\Big)^2\\ \nonumber
&\quad+\frac{8R^2}{\muconstant} \Big(
\N{\bJ - \ri \omega\epsilon \nabla p}_{L^2(B_R;\mu)} 
+\sqrt{\epsilon_0 \mu_0}\N{\bK - \ri \omega\mu \nabla q}_{L^2(B_R;\mu^{-1})}
\Big)^2\\
&\quad 
+ 2 \epsilonconstant \N{\nabla p}^2_{L^2(B_R;\epsilon)}
  +2 \muconstant\N{ \nabla q}^2_{L^2(B_R;\mu)}.
\label{eq:analogue1}
\end{align}
We now use the inequalities $(a+b)^2\leq 2(a^2+b^2)$ for $a,b>0$, 
\beqs
\N{\bv}^2_{L^2(B_R;\epsilon)}\leq \epsilonmax\mumax \N{\bv}^2_{L^2(B_R;\mu^{-1})} \quad\tfa \bv\in \Com^3,
\eeqs
and the analogous inequality with $\epsilon$ and $\mu$ swapped, to find
\begin{align}\nonumber
&\epsilonconstant\N{\bE }^2_{L^2(B_R)} + \muconstant\N{\bH}^2_{L^2(B_R)} \\ \nonumber
&\le
2\max\left\{ 8R^2 \left(\frac{\epsilonmax\mumax}{\epsilonconstant} + \frac{\epsilon_0 \mu_0}{\muconstant}\right), \frac{\muconstant}{\omega^2}\right\}
\bigg( \N{\bK - \ri \omega\mu \nabla q}_{L^2(B_R;\mu^{-1})}^2 + \omega^2\N{ \nabla q}^2_{L^2(B_R;\mu)}\bigg)\\
&
\quad+2\max\left\{ 8R^2 \left(\frac{\epsilonmax\mumax}{\muconstant} + \frac{\epsilon_0 \mu_0}{\epsilonconstant}\right), \frac{\epsilonconstant}{\omega^2}\right\}
\bigg( \N{\bJ - \ri \omega\epsilon \nabla p}_{L^2(B_R;\epsilon^{-1})}^2 + \omega^2\N{ \nabla p}^2_{L^2(B_R;\epsilon)}\bigg).
\label{eq:analogue2}
\end{align}
The result 
\eqref{eq:thm:BoundSmooth}
 then follows by using
\eqref{eq:orthog1} and \eqref{eq:orthog2}.

If $\bK=\bzero$, then $q=0$, and we don't need to use the inequality $(a+b)^2\leq 2(a^2+b^2)$ for $a,b>0$ in \eqref{eq:analogue1}. The $8$ in the term involving $\bJ$ in both \eqref{eq:analogue2} and \eqref{eq:thm:BoundSmooth} then reduces to $4$. Similarly if $\bJ=\bzero$, then the $8$ 
in the term involving $\bK$ in both \eqref{eq:analogue2} and \eqref{eq:thm:BoundSmooth} reduces to $4$.
\end{proof}

\subsection{Proof of Theorem \ref{thm:BoundRough}}
\label{subsec:Proofs2}

We first note that when $\epsilon$ and $\mu$ are as in the assumption \eqref{eq:RoughCoefficients}, $\epsilon_0=\epsilonmax$ and $\mu_0=\mumax$. 

\ble\mythmname{Approximation of $\eps$ and $\mu$ by $\epsilon_\delta,\mu_\delta \in C^\infty$}
\label{lem:mollifier}
Suppose $\epsilon,\mu\in L^\infty(\Rea^3,\SPD)$ are as in Theorem \ref{thm:BoundRough}; i.e., they are defined by \eqref{eq:RoughCoefficients}, with $\Pi_{\epsilon},\Pi_\mu \in L^\infty(\Rea^3,\SPD)$ monotonically non-decreasing in the radial direction in the sense of \eqref{eq:monotone}.

Then there exists $\delta_0>0$ such that for all $0<\delta\leq \delta_0$ there exists $\epsilon_\delta,\mu_\delta \in C^\infty(\Rea^3, \SPD)$
such that 
\beq\label{eq:limits2}
\epsilonmin\preceq \epsilon_\delta(\bx)\preceq \epsilonmax 
\quad\tand\quad
\mumin\preceq \mu_\delta(\bx)\preceq \mumax 
\quad \tfa \bx \in \Rea^d, 
\eeq
for any $R>0$,
\beq\label{eq:approx}
\epsilon_\delta\to\epsilon \,\tand\, \mu_\delta\to\mu\quad\tas\quad \delta \tendo \quad \text{both a.e.\ in $\IR^3$ and in }L^2(B_R),
\eeq
and $\epsilon_\delta$ and $\mu_\delta$ satisfy 
\beq\label{eq:GrowthCoeff2}
\MI + \big((\bx\cdot\nabla)\epsilon_\delta\big)\epsilon_\delta^{-1} \succeq 1
\quad\tand\quad
\MI + \big((\bx\cdot\nabla)\mu_\delta\big)\mu_\delta^{-1} \succeq 1.
\eeq
\ele

\bpf
The plan is to mollify $\mu$ and $\epsilon$; the key point, however, is that standard mollification in $\Rea^d$ is not guaranteed to preserve monotonicity in the radial direction, and thus we need to mollify along spherical coordinates.
To avoid issues with the singularities of the spherical coordinates at the origin and along the $x_3$ axis we treat separately these regions.

To see that standard mollification does not preserve radial monotonicity, consider $\epsilon=\epsilon_0/2$ in the half sphere $\{\|\bx\|<1,\; x_1>0\}$ and $\epsilon=\epsilon_0$ elsewhere.
Then the convolution $\epsilon*\psi_\delta$ between $\epsilon$ and a standard mollifier $\psi_\delta\in C^\infty(\IR^3)$ with $\supp(\psi_\delta)=\overline{B_\delta}$ and $0<\delta<1/2$ satisfies $(\epsilon*\psi_\delta)(\bzero)>\epsilon_0/2=(\epsilon*\psi_\delta)(1/2,0,0)$.

We use classical spherical coordinates $\bx=(\rho\cos\theta\sin\varphi, \rho\sin\theta\sin\varphi, \rho\cos\varphi)$, and 
we define the ``candy-shaped'' domain $\bowtie_\delta$ by 
\begin{equation}\label{eq:CandyShape}
\bowtie_\delta:=\{\rho<\delta \} \cup \{\rho<R,\;\varphi< \delta\} \cup  \{\rho<R,\;\varphi> \pi-\delta\}\subset\IR^3,
\qquad 0<\delta<\min\{1/2,R/2\},
\end{equation}
with $R>0$ as in \eqref{eq:support_scatterer}; see Figure~\ref{fig:CandyShape}.
The set $\bowtie_\delta$ has volume $|\bowtie_\delta|=\calO(\delta^2)$ as $\delta\to0$.

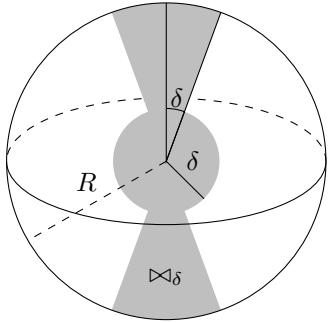
\begin{figure}[htb]\begin{center}
\begin{tikzpicture}[scale=0.7]
\draw[dashed] (3,0) arc (0:180:3 and 1.2);  % half ellipse!
\fill[lightgray] (0,0) circle(1);
\fill[lightgray] (0,0) -- (70:3) arc(70:110:3) -- cycle;
\fill[lightgray] (0,0) -- (250:3) arc(250:290:3) -- cycle;
%\fill[lightgray] (0,0) -- (1,-3)--(-1,-3);
\draw(0,0)--(.71,-.71); \draw (.5,0) node{$\delta$};
\draw(0,3)--(0,0)--({3*cos(70)},{3*sin(70)}); \draw (.2,1.2) node{$\delta$};
%\draw (.3,1.3) node{$\delta$};
\draw(0,0) circle(3);
\draw(0,0) -- (70:1) arc (70:90:1);
\draw (0,-2.2) node{$\bowtie_\delta$};
\draw[dashed](0,0)--(-3*.866,-1.5);
\draw (-1.5,-.4) node{$R$};
\draw(-3,0) arc (-180:0:3 and 1.2);  % half ellipse!
\end{tikzpicture}\end{center}
\caption{The shaded region represents the set $\bowtie_\delta\subset B_R$ defined in \eqref{eq:CandyShape}.
The parameter $0<\delta<\min\{1/2,R/2\}$ is both the radius of the inner ball and the opening of the two cones.}
\label{fig:CandyShape}
\end{figure}

% We also define
% $$
% \epsilon_\delta:=
% \begin{cases}
% \epsilonmin\MI &\iin \bowtie_\delta,\\
% \epsilon &\iin\IR^3\setminus \bowtie_\delta.
% \end{cases}
% % \hspace{20mm}\raisebox{-5mm}{
% % \begin{tikzpicture}[scale=0.3]
% % \fill[lightgray] (0,0) circle(1);
% % \fill[lightgray] (0,0) -- (70:3) arc(70:110:3) -- cycle;
% % \fill[lightgray] (0,0) -- (250:3) arc(250:290:3) -- cycle;
% % %\fill[lightgray] (0,0) -- (1,-3)--(-1,-3);
% % \draw(0,0)--(.71,-.71); \draw (.5,0) node{$\delta$};
% % \draw(0,3)--(0,0)--({3*cos(70)},{3*sin(70)}); \draw (.5,3.2) node{$\delta$};
% % %\draw (.3,1.3) node{$\delta$};
% % \draw(0,0) circle(3);
% % \draw(0,0) -- (70:1) arc (70:90:1);
% % \draw (0,-2.5) node{$\bowtie_\delta$};
% % \end{tikzpicture}}
% $$
% By eqref{eq:lambda}, $\epsilon_\delta(\bx)=\epsilonmin\MI+\Pi$, where $\Pi$ satisfies the radial monotonicity condition eqref{eq:monotone}.

We first consider the approximation of a scalar $f\in L^\infty(\IR^3)$.
To mollify $f$ along the three spherical coordinates, we define (almost everywhere) a transformed function $f^\bigstar\in L^\infty(\IR^3)$ as
\begin{align*}
&f^\bigstar:\{(\rho,\theta,\varphi),\;\rho\in\IR,\;\theta\in\IR,\;\varphi\in\IR\}\to\IR,\\
&f^\bigstar(\rho,\theta,\varphi)=
\begin{cases}
f(\rho\cos\theta\sin\varphi, \rho\sin\theta\sin\varphi, \rho\cos\varphi) & \text{if } \rho>2\delta, \;2\delta<\varphi<\pi-2\delta,\; 0<\theta<2\pi,\\
f_{\min}:=\essinf f &\text{if } \rho<2\delta \text{ or }\varphi\notin(2\delta,\pi-2\delta),\\
\text{extended $2\pi$-periodically in $\theta$}.
\end{cases}
\end{align*}
Let $0\le\psi^1\in C^\infty(\IR^3)$, be supported in $B_1$ and have integral $\int_{\IR^3}\psi^1=1$.
Let $f^\bigstar_\delta\in C^\infty(\IR^3)$ be the convolution of $f^\bigstar$ with $\psi^\delta(\bullet):=\delta^{-3}\psi^1(\bullet/\delta)$, which is supported in $B_\delta$.
Then $f^\bigstar_\delta=f_{\min}$ in $\{\rho\le\delta\}\cup\{\varphi\le\delta\}\cup\{\varphi\ge\pi-\delta\}$, is $2\pi$-periodic in $\theta$, and $\lim_{\delta\to0}f_\delta^\bigstar=f^\bigstar$ both in $L^2\loc(\IR^3)$ and almost everywhere; see, e.g., \cite[\S4.2, Theorem 1]{EvGa92}.
Moreover $f_{\min}\le f_\delta^\bigstar\le \N{f}_{L^\infty}$ by the properties of the mollification.

Finally, $f_\delta:\IR^3\to\IR$ defined by $f_\delta(\rho\cos\theta\sin\varphi, \rho\sin\theta\sin\varphi, \rho\cos\varphi):=f_\delta^\bigstar(\rho,\theta,\varphi)$ is smooth, takes values in the interval $[f_{\min},\N{f}_{L^\infty}]$, satisfies $f_\delta=f_{\min}$ on $\bowtie_\delta$, and approximates $f$ both almost everywhere and in $L^2(B_R)$:
%$\lim_{\delta\to0}f_\delta|_{B_1}=f|_{B_1}$ in $L^2(B_1)$.................
\begin{align*}
\N{f-f_\delta}_{L^2(B_R)}^2
&=\int_{\bowtie_\delta} |f-f_\delta|^2 + \int_{B_R\setminus\bowtie_\delta} |f-f_\delta|^2 \\
&=\int_{\bowtie_\delta} |f-f_{\min}|^2 
+\int_\delta^R \int_0^{2\pi}\int_\delta^{\pi-\delta} |f^\bigstar-f_\delta^\bigstar|^2\sin\varphi\di\varphi\di\theta\rho^2\di\rho
\\
&\le 4|\bowtie_\delta| \N{f}_{L^\infty}^2 + R^2 \N{f^\bigstar-f_\delta^\bigstar}^2_{L^2((\delta,R)\times(0,2\pi)\times(\delta,\pi-\delta))} %\\&
\xrightarrow{\delta\to0}0.
\end{align*}
Moreover, if $f$ is radially monotonic in the sense that $\essinf_{\bx\in\IR^3} [f((1+h)\bx)-f(\bx)]\ge0$ for all $h\ge0$, then 
$\essinf_{(\rho,\theta,\varphi)\in\IR^3} [f^\bigstar(\rho+h,\theta,\varphi)-f^\bigstar(\rho,\theta,\varphi)]\ge0$, $f_\delta^\bigstar$ is monotonic in $\rho$ by the definition of convolution, and $f_\delta$ is radially monotonic in pointwise sense:
$f_{\delta}((1+h)\bx)\ge f_{\delta}(\bx)$ for all $\bx\in\IR$, $h\ge0$.

The arguments extends to matrix-valued fields $\epsilon$ satisfying the assumptions in the assertion.
The field $\epsilon^\bigstar: \IR^3\to \SPD$ is defined similarly to $f^\bigstar$ with diagonal value $\epsilonmin\MI$ in place of $f_{\min}$.
The smooth field $\epsilon^\bigstar_\delta$ is obtained by componentwise mollification of $\epsilon^\bigstar$, and the pullback $\epsilon_\delta$ is defined as $f_\delta$.
Then $\epsilon_\delta\to\epsilon$ with the same argument for the scalar case and each component of $\epsilon_\delta$ is in $C^\infty(\IR^3)$.
The monotonicity follows:
for all $\bx\in\IR^3$, $\bv\in\IC^3$, $h>0$,
\begin{align*}
&\epsilon_\delta\big((1+h)\bx\big)\bv\cdot\conj\bv -\epsilon_\delta(\bx)\bv\cdot\conj\bv 
=\epsilon_\delta^\bigstar\big((1+h)\rho,\theta,\varphi\big)\bv\cdot\conj\bv -\epsilon_\delta^\bigstar(\rho,\theta,\varphi)\bv\cdot\conj\bv \\
&=\int_{\IR^3} \underbrace{\Big(\epsilon^\bigstar\big((1+h)\rho',\theta',\varphi'\big)\bv\cdot\conj\bv -\epsilon^\bigstar(\rho',\theta',\varphi')\bv\cdot\conj\bv\Big)}_{\ge0 \text{ a.e.}}
\psi^\delta(\rho-\rho',\theta-\theta',\varphi-\varphi')\di\varphi'\di\theta'\di\rho'
\ge0
\end{align*}
where the term in the integral is non-negative because $\epsilon((1+h)\bx)\succeq\epsilon(\bx)$ a.e., $\epsilon\succeq\epsilonmin$ a.e.\ (since $\Pi_{\epsilon} \in L^\infty(\IR^3,\SPD)$) and because of the shape of the region where $\epsilon^\bigstar=\epsilonmin\MI$.
In particular $\epsilon_\delta$ is positive definite and $\epsilonmin\preceq \epsilon_\delta\preceq\epsilonmax$.
Since $\epsilon_\delta$ is smooth and radially monotonic, $(\bx\cdot\nabla)\epsilon_\delta\succeq0$ and \eqref{eq:GrowthCoeff2} follows.
\epf

The proof of Lemma~\ref{lem:mollifier} also corrects the proof of \cite[Theorem 2.7]{GrPeSp:18}, where it was assumed that standard mollification in Cartesian coordinates preserves radial monotonicity.

We now prove Theorem \ref{thm:BoundRough}. In this proof, we use the weighted norm 
on $L^2(B_R)\times L^2(B_R)$ 
corresponding to the left-hand side of \eqref{eq:thm:BoundSmooth}, i.e.,
\beqs%\label{eq:weighted_norm2}
\Tnorm{(\bE, \bH)}_{\epsilon,\mu}^2:= 
\epsilonconstant\|\bE\|_{\weightedLtepsilon}^2+ \muconstant \N{\bH}_{\weightedLtmu}^2.
\eeqs

\bpf[Proof of Theorem \ref{thm:BoundRough}]
By \cite[Theorem 2.10]{PiWeWi01} it is sufficient to show that the bound \eqref{eq:thm:BoundSmooth} holds under the assumption that the solution of \eqref{eq:first_order} exists.	

Without loss of generality, we assume that $\bJ$ and $\bK$ are compactly supported in $B_R$. Indeed, once the bound \eqref{eq:thm:BoundSmooth} is proved for such $\bJ$ and $\bK$, since compactly-supported functions are dense in $L^2(B_R)$ 
and the bound is independent of the supports of $\bJ$ and $\bK$, the bound \eqref{eq:thm:BoundSmooth} holds for all $\bJ, \bK\in L^2(B_R)$.

By the density of $C^\infty(\overline{\domainimp})^3$ in $H(\curl;B_R)$ (see, e.g., \cite[Theorem 3.26]{MON03}), 
given $\eta>0$, $\epsilon, \mu, \bE,$ and $\bH$, 
there exists $\bE_\eta, \bH_\eta\in C^\infty(\overline{\domainimp})^3$ such that 
$\supp(\bE-\bE_\eta)\cup \supp(\bH-\bH_\eta) \Subset B_R$ and
\begin{align}
&\N{\nabla\times (\bE- \bE_\eta)}_{L^2(B_R)}^2+ 
\N{\nabla\times (\bH - \bH_\eta)}_{L^2(B_R)}^2+ 
\Big|
\Tnorm{(\bE,\bH)}_{\epsilon,\mu}^2 -\Tnorm{(\bE_\eta,\bH_\eta)}_{\epsilon,\mu}^2
\Big|
\leq \eta/4.
\label{eq:newapprox0}
\end{align}
Indeed, let $\chi \in C^\infty(B_R)$ be such that $\chi \equiv 1$ on $\supp \bJ \cup \supp \bK \cup \supp(\epsilon-\epsilon_0\MI) \cup \supp(\mu-\mu_0\MI)$ and $\chi \equiv 0$ in a neighbourhood of $\partial B_R$.
Since $\epsilon=\epsilon_0\MI$ and $\mu= \mu_0\MI$ on $\supp(1-\chi)$, by elliptic regularity, $(1-\chi)\bE$ and $(1-\chi)\bH$ are both $C^\infty$. We can then take $\bE_\eta$ to be $(1-\chi)\bE$ plus a mollified $\chi \bE$, and similarly for $\bH_\eta$; this ensures that 
$\bE-\bE_\eta$ and $\bH- \bH_\eta$ are supported near $\supp\,\chi$.

With $\epsilon_\delta$ and $\mu_\delta$ as in Lemma~\ref{lem:mollifier}, observe that
\beq\label{eq:newapprox00}
\Tnorm{(\bE_\eta,\bH_\eta)}_{\epsilon,\mu}^2 -\Tnorm{(\bE_\eta,\bH_\eta)}_{\epsilon_\delta,\mu_\delta}^2 
= \epsilonconstant \int_{B_R}(\epsilon-\epsilon_\delta) \bE_\eta\cdot\conj \bE_\eta
+ \muconstant \int_{B_R}(\mu-\mu_\delta) \bH_\eta\cdot\conj \bH_\eta,
\eeq
and thus, combining \eqref{eq:newapprox0} and \eqref{eq:newapprox00}, we obtain 
\begin{align}\nonumber
\Tnorm{(\bE,\bH)}^2_{\epsilon,\mu} &\leq
\Tnorm{(\bE_\eta,\bH_\eta)}^2_{\epsilon_\delta,\mu_\delta} + \eta/4 \\
&\quad + \Big( 
\epsilonconstant \N{\bE_\eta}^2_{L^\infty(B_R)} \N{\epsilon-\epsilon_\delta}_{L^1(B_R)}
+\muconstant \N{\bH_\eta}^2_{L^\infty(B_R)} \N{\mu-\mu_\delta}_{L^1(B_R)}
\Big).
\label{eq:newapprox1}
\end{align}
By \eqref{eq:first_order},
\beq\label{eq:PDEperturb1}
\ri\wn\epsilon_\delta\bE_\eta+\nabla\times\bH_\eta =
\bJ + 
\ri \wn (\epsilon_\delta-\epsilon)\bE_\eta + \ri\omega\epsilon(\bE_\eta-\bE)
+\nabla\times(\bH_\eta-\bH)
\eeq
and
\beq\label{eq:PDEperturb2}
\ri\wn\mu_\delta\bH_\eta-\nabla\times\bE_\eta =
\bK+
\ri \wn (\mu_\delta-\mu)\bH_\eta + \ri\omega\mu(\bH_\eta-\bH)
-\nabla\times(\bE_\eta-\bE),
\eeq
and the right-hand sides of these PDEs are supported in $B_R$. 
By Lemma \ref{lem:mollifier}, $\epsilon_\delta$ and $\mu_\delta$ satisfy the conditions of Theorem \ref{thm:BoundSmooth}, and thus, by \eqref{eq:thm:BoundSmooth},
\begin{align}\nonumber
\Tnorm{(\bE_\eta,\bH_\eta)}^2_{\epsilon_\delta,\mu_\delta} 
\leq C_1 
&\wnorm{\epsilon^{-1}_\delta}{B_R}{ \bJ + 
\ri \wn (\epsilon_\delta-\epsilon)\bE_\eta + \ri\omega\epsilon(\bE_\eta-\bE)
+\nabla\times(\bH_\eta-\bH)
}^2
\\
&
+C_2\wnorm{\mu_\delta^{-1}}{B_R}{
\bK+
\ri \wn (\mu_\delta-\mu)\bH_\eta + \ri\omega\mu(\bH_\eta-\bH)
-\nabla\times(\bE_\eta-\bE)}^2
\label{eq:newapprox2}
\end{align}
where we have abbreviated the constants on the right-hand side of \eqref{eq:thm:BoundSmooth} to $C_1$ and $C_2$ to keep the notation concise.
The crucial point is that the $C_1$ and $C_2$ corresponding to $\epsilon_\delta$ and $\mu_\delta$ can be taken to be the $C_1$ and $C_2$ corresponding to $\epsilon$ and $\mu$ by \eqref{eq:limits2} and \eqref{eq:GrowthCoeff2}; thus, in particular, $C_1$ and $C_2$ are independent of $\delta$.

We now claim that the approximation properties \eqref{eq:approx} and \eqref{eq:newapprox2} imply that 
given $\epsilon, \mu$, $\bE, \bH$ (and associated $\bJ, \bK$), $\wn>0$, and $\zeta>0$ one can choose $\eta= \eta(\zeta)>0$ and $\delta=\delta(\zeta,\eta)>0$  such that 
the difference between the right-hand side of  \eqref{eq:newapprox2} and $C_1\|\bJ\|^2_{L^2(B_R;\epsilon_\delta^{-1})} + C_2\|\bK\|^2_{L^2(B_R;\mu_\delta^{-1})}$
is $\leq \zeta/4$, so that
\beq\label{eq:nursery1}
\Tnorm{(\bE_\eta,\bH_\eta)}^2_{\epsilon_\delta,\mu_\delta} \leq C_1\|\bJ\|^2_{L^2(B_R;\epsilon_\delta^{-1})} + C_2\|\bK\|^2_{L^2(B_R;\mu_\delta^{-1})} + \zeta/4.
\eeq

Once this claim is established, without loss of generality, we can further assume that 
$\eta\leq \zeta$, and reduce $\delta$ (if necessary) so that the last term on the right-hand side of \eqref{eq:newapprox1} is $\leq \zeta/4$. 
Then combining \eqref{eq:newapprox1} and \eqref{eq:nursery1}, and using that $\eta\leq \zeta$, we obtain that 
\beq\label{eq:nursery2}
\Tnorm{(\bE,\bH)}^2_{\epsilon,\mu}\leq C_1\|\bJ\|^2_{L^2(B_R;\epsilon_\delta^{-1})} + C_2\|\bK\|^2_{L^2(B_R;\mu_\delta^{-1})} + \zeta/4 + \zeta/4 + \zeta/4.
\eeq
As $\delta\to 0$, 
\beqs
\wnorm{\epsilon_\delta^{-1}}{B_R}{\bJ}^2 = \int_{B_R}\epsilon_\delta^{-1} \bJ\cdot\conj \bJ \to \int_{B_R} \epsilon^{-1} \bJ \cdot\conj{\bJ} = \wnorm{\epsilon^{-1}}{B_R}{\bJ}^2
\eeqs
(and similarly for $\wnorm{\mu_\delta^{-1}}{B_R}{\bK}^2$) by the dominated convergence theorem (since $\epsilon_\delta\to \epsilon$ and $\mu\to \mu_\delta$ pointwise almost everywhere by Lemma \ref{lem:mollifier}).
We can therefore decrease $\delta$ again (if necessary) so that 
\beq\label{eq:nursery3}
C_1\|\bJ\|^2_{L^2(B_R;\epsilon_\delta^{-1})} + C_2\|\bK\|^2_{L^2(B_R;\mu_\delta^{-1})}
\leq C_1\|\bJ\|^2_{L^2(B_R;\epsilon^{-1})} + C_2\|\bK\|^2_{L^2(B_R;\mu^{-1})} +\zeta/4,
\eeq
and combining \eqref{eq:nursery2} and \eqref{eq:nursery3} we obtain that 
\beqs
\Tnorm{(\bE,\bH)}^2_{\epsilon,\mu} \leq C_1 \N{\bJ}^2_{L^2(B_R;\epsilon^{-1})} + C_2\N{\bK}^2_{L^2(B_R;\mu^{-1})} + \zeta.
\eeqs
Since $\zeta>0$ was arbitrary, the bound \eqref{eq:thm:BoundSmooth} follows.

We now complete the proof by establishing the claim above. First observe that
\begin{align}\nonumber
&\N{ \bJ + 
\ri \wn (\epsilon_\delta-\epsilon)\bE_\eta + \ri\omega\epsilon(\bE_\eta-\bE)
+\nabla\times(\bH_\eta-\bH)
}^2_{L^2(B_R;\epsilon_\delta^{-1})}\\ \nonumber
&\leq 
2\N{ \bJ}_{L^2(B_R;\epsilon_\delta^{-1})}\Big( 
\wn\N{ \epsilon_\delta-\epsilon}_{L^2(B_R;\epsilon_\delta^{-1})}\N{\bE_\eta}_{L^\infty(B_R)} 
+ 
\N{
\ri\omega\epsilon(\bE_\eta-\bE)
+\nabla\times(\bH_\eta-\bH)}_{L^2(B_R;\epsilon_\delta^{-1})}\Big)\\
&\,
+\N{ \bJ}^2_{L^2(B_R;\epsilon_\delta^{-1})} 
+ 2\wn^2\N{ \epsilon_\delta-\epsilon}_{L^2(B_R;\epsilon_\delta^{-1})}^2\N{\bE_\eta}_{L^\infty(B_R)}^2 
+ 2\N{
\ri\omega\epsilon(\bE_\eta-\bE)
+\nabla\times(\bH_\eta-\bH)}_{L^2(B_R;\epsilon_\delta^{-1})}^2,\label{eq:Friday0}
\end{align}
and, similarly, 
\begin{align}\nonumber
&\N{ \bK +\ri \wn (\mu_\delta-\mu)\bH_\eta + \ri\omega\mu(\bH_\eta-\bH)
-\nabla\times(\bE_\eta-\bE)}^2_{L^2(B_R;\mu_\delta^{-1})}\\ \nonumber
&\leq 
2\N{ \bK}_{L^2(B_R;\mu_\delta^{-1})}\Big( 
\wn\N{ \mu_\delta-\mu}_{L^2(B_R;\mu_\delta^{-1})}\N{\bH_\eta}_{L^\infty(B_R)} 
+ 
\N{
\ri\omega\epsilon(\bH_\eta-\bH)
-\nabla\times(\bE_\eta-\bE)}_{L^2(B_R;\mu_\delta^{-1})}\Big)\\
&\,
+\N{ \bK}^2_{L^2(B_R;\mu_\delta^{-1})} + 2\wn^2\N{\mu_\delta-\mu}_{L^2(B_R;\mu_\delta^{-1})}^2\N{\bH_\eta}_{L^\infty(B_R)}^2 
+ 2\N{
\ri\omega\epsilon(\bH_\eta-\bH)
-\nabla\times(\bE_\eta-\bE)}_{L^2(B_R;\mu_\delta^{-1})}^2.\label{eq:Friday1}
\end{align}
We now need to make the terms on the right-hand sides of \eqref{eq:Friday0} and \eqref{eq:Friday1} that are not 
$\|\bJ\|^2_{L^2(B_R;\epsilon_\delta^{-1})} $ and 
$\| \bK\|^2_{L^2(B_R;\mu_\delta^{-1})}$, respectively, small.
We first deal with the terms that involve $\bE-\bE_\eta$ and $\bH-\bH_\eta$ (and thus will be made small by choosing $\eta$ small).
By \eqref{eq:limits2}, 
\begin{align*}\nonumber
&2\N{ \bJ}_{L^2(B_R;\epsilon_\delta^{-1})}
\N{
\ri\omega\epsilon(\bE_\eta-\bE)
+\nabla\times(\bH_\eta-\bH)}_{L^2(B_R;\epsilon_\delta^{-1})}
+ 2\N{
\ri\omega\epsilon(\bE_\eta-\bE)
+\nabla\times(\bH_\eta-\bH)}_{L^2(B_R;\epsilon_\delta^{-1})}^2\\
&\leq 2\epsilonmin^{-1}\Big(\N{ \bJ}_{L^2(B_R)}
\N{
\ri\omega\epsilon(\bE_\eta-\bE)
+\nabla\times(\bH_\eta-\bH)}_{L^2(B_R)}
+ 
\N{
\ri\omega\epsilon(\bE_\eta-\bE)
+\nabla\times(\bH_\eta-\bH)}_{L^2(B_R)}^2\Big)
\end{align*}
and 
\begin{align*}
&2\N{ \bK}_{L^2(B_R;\mu_\delta^{-1})}
\N{
\ri\omega\mu(\bH_\eta-\bH)
+\nabla\times(\bE_\eta-\bE)}_{L^2(B_R;\mu_\delta^{-1})}
+ 2\N{
\ri\omega\mu(\bH_\eta-\bH)
-\nabla\times(\bE_\eta-\bE)}_{L^2(B_R;\mu_\delta^{-1})}^2\\
&\leq 2\mumin^{-1}\Big(\N{ \bK}_{L^2(B_R)}
\N{
\ri\omega\mu(\bH_\eta-\bH)
+\nabla\times(\bE_\eta-\bE)}_{L^2(B_R)}
+ \N{
\ri\omega\mu(\bH_\eta-\bH)
-\nabla\times(\bE_\eta-\bE)}_{L^2(B_R)}^2\Big).
\end{align*}
Therefore, by \eqref{eq:newapprox0} and the bounds \eqref{eq:limits} on $\epsilon$ and $\mu$, given $\zeta>0$, we can choose $\eta>0$ such that
\begin{align}\nonumber
&2\N{ \bJ}_{L^2(B_R;\epsilon^{-1}_\delta)}
\N{
\ri\omega\epsilon(\bE_\eta-\bE)
+\nabla\times(\bH_\eta-\bH)}_{L^2(B_R;\epsilon^{-1}_\delta)}
\\ \label{eq:Friday2}
&\hspace{3cm}+ 2\N{
\ri\omega\epsilon(\bE_\eta-\bE)
+\nabla\times(\bH_\eta-\bH)}_{L^2(B_R;\epsilon^{-1}_\delta)}^2\leq \frac{\zeta}{8C_1}
\end{align}
and
\begin{align}\nonumber
&2\N{ \bK}_{L^2(B_R;\mu^{-1}_\delta)}
\N{
\ri\omega\mu(\bH_\eta-\bH)
+\nabla\times(\bE_\eta-\bE)}_{L^2(B_R;\mu^{-1}_\delta)}\\ \label{eq:Friday3}
&\hspace{3cm}+ 2\N{
\ri\omega\mu(\bH_\eta-\bH)
-\nabla\times(\bE_\eta-\bE)}_{L^2(B_R;\mu^{-1}_\delta)}^2\leq \frac{\zeta}{8C_1}.
\end{align}

We now deal with the remaining terms on the right-hand sides of \eqref{eq:Friday0} and \eqref{eq:Friday1}, which will be made small by making $\delta$ small.
By \eqref{eq:approx} (and arguing similarly to above using \eqref{eq:limits2} to deal with the norms weighted by $\epsilon_\delta^{-1}$ and $\mu^{-1}_\delta$), given $\zeta$ and $\eta$, we can choose $\delta>0$ such that 
\beq\label{eq:Friday4}
2\wn^2\N{ (\epsilon_\delta-\epsilon)}_{L^2(B_R;\epsilon^{-1}_\delta)}^2\N{\bE_\eta}_{L^\infty(B_R)}^2 
+2\N{ \bJ}_{L^2(B_R\epsilon^{-1}_\delta)}
\wn\N{ (\epsilon_\delta-\epsilon)}_{L^2(B_R;\epsilon^{-1}_\delta)}\N{\bE_\eta}_{L^\infty(B_R)}\leq\frac{\zeta}{8C_1}
\eeq
and
\beq\label{eq:Friday5}
2\wn^2\N{ (\mu_\delta-\mu)}_{L^2(B_R;\mu^{-1}_\delta)}^2\N{\bH_\eta}_{L^\infty(B_R)}^2 
+2\N{ \bK}_{L^2(B_R;\mu^{-1}_\delta)}
\wn\N{ (\mu_\delta-\mu)}_{L^2(B_R;\mu^{-1}_\delta)}\N{\bH_\eta}_{L^\infty(B_R)}\leq\frac{\zeta}{8C_1}.
\eeq
Combining \eqref{eq:Friday0}--\eqref{eq:Friday5}, we obtain that 
\begin{align*}
&C_1\N{ \bJ + 
\ri \wn (\epsilon_\delta-\epsilon)\bE_\eta + \ri\omega\epsilon(\bE_\eta-\bE)
+\nabla\times(\bH_\eta-\bH)
}^2_{L^2(B_R\epsilon^{-1}_\delta)}\leq 
C_1\N{ \bJ}^2_{L^2(B_R;\epsilon^{-1}_\delta)} + \zeta/4
\end{align*}
and
\begin{align*}
C_2\N{
\bK+
\ri \wn (\mu_\delta-\mu)\bH_\eta + \ri\omega\mu(\bH_\eta-\bH)
-\nabla\times(\bE_\eta-\bE)}^2_{L^2(B_R;\mu^{-1}_\delta)}
\leq 
C_2\N{ \bK}^2_{L^2(B_R;\mu^{-1}_\delta)} + \zeta/4,
\end{align*}
and the claim (and hence also the result) is proved.
\epf

\bre\mythmname{Obtaining the bound \eqref{eq:unweighted} in unweighted norms}
\label{rem:unweighted}
The proof of the unweighted-norm bound \eqref{eq:unweighted} diverges from the proof of \eqref{eq:thm:BoundSmooth} at \eqref{eq:diverge1}, which is equivalent to 
\begin{align}\nonumber
&\int_{B_R}
\big(\epsilon+(\bx\cdot\nabla)\epsilon\big)\bE\cdot\conj\bE+\big(\mu+(\bx\cdot\nabla)\mu\big)\bH\cdot\conj\bH
\\
&\le 
 2\int_{B_R}\!\Re\bigg\{\conj\bE\cdot\bigg(\epsilon\bx\times\bK
-R\sqrt{\epsilon_0\mu_0}\bJ + \bx \frac{\nabla\cdot\bJ}{\ri\wn}\bigg)
+\conj\bH \cdot\bigg(\mu (\bx\times \bJ)+R\sqrt{\mu_0\epsilon_0}\bK 
-\bx\frac{\nabla\cdot\bK}{\ri\wn}\bigg)
\bigg\}.\nonumber
\end{align}
The conditions \eqref{eq:GrowthCoeffunweighted} on $\epsilon$ and $\mu$ and Lemma \ref{lem:E1} with 
$\bv_1=\bE$, $\bv_2=\bH$, $m_1=\epsilon_*$, $m_2=\mu_*$, and 
$\MM_j=\MI$, $j=1,2$, imply that 
\begin{align*}
&\epsilon_*\N{\bE}^2_{L^2(B_R)} + \mu_*\N{\bH}^2_{L^2(B_R)}\\
&\le 
\frac{4}{\epsilon_*} \N{
\epsilon (\bx\times \bK) -R\sqrt{\mu_0\epsilon_0}\bJ
+\bx\frac{\nabla\cdot\bJ}{\ri\wn}
}^2_{L^2(B_R)}
+ 
\frac{4}{\mu_*} \N{
\mu (\bx\times \bJ)+R\sqrt{\mu_0\epsilon_0}\bK
-\bx\frac{\nabla\cdot\bK}{\ri\wn}
}^2_{L^2(B_R)}
\\
&\le
\frac{4}{\epsilon_*} \bigg(
R\N{\epsilon}_{L^\infty(B_R)} \N{\bK}_{L^2(B_R)} 
+R\sqrt{\mu_0\epsilon_0}\N{\bJ}_{L^2(B_R)}
+\frac R\wn \N{\nabla\cdot\bJ}_{L^2(B_R)}
\bigg)^2\\
&\qquad+\frac{4}{\mu_*} \bigg(
R\N{\mu}_{L^\infty(B_R)} \N{\bJ}_{L^2(B_R)} 
+R\sqrt{\mu_0\epsilon_0}\N{\bK}_{L^2(B_R)}
+\frac R\wn \N{\nabla\cdot\bK}_{L^2(B_R)}
\bigg)^2.
\end{align*}
where 
$\|\epsilon\|_{L^\infty(B_R)} = \esssup_{\bx \in B_R}|\epsilon(\bx)|_2=\esssup_{\bx \in B_R}\sup_{\bzero\ne\bv\in\IR^3}(|\epsilon(\bx)\bv|_2/|\bv|_2)$, and similarly for $\|\mu\|_{L^\infty(B_R)}$.
Since the $2$-norm of a matrix $\MA$ is the maximum eigenvalue of $\MA^*\MA$, 
$\|\mu\|_{L^\infty(B_R)}\leq \mumax$ and $\|\epsilon\|_{L^\infty(B_R)}\leq \epsilonmax$. 

The proof of \eqref{eq:unweighted} then continues as in the weighted case. The additional factors of $\epsilonmax/\epsilonmin$ and $\mumax/\mumin$ in \eqref{eq:unweighted} compared to \eqref{eq:thm:BoundSmooth} come from using \eqref{eq:orthog1} and \eqref{eq:orthog2} to obtain the bounds 
\beqs
\N{\bJ- \ri\omega\epsilon \nabla p}^2_{L^2(B_R)} \leq \frac{\epsilon_{\max}}{\epsilon_{\min}} \N{\bJ}_{L^2(B_R)}^2
\quad\tand\quad
\N{\bK- \ri\omega\mu \nabla q}_{L^2(B_R)}^2 \leq \frac{\mu_{\max}}{\mu_{\min}} \N{\bK}_{L^2(B_R)}^2.
\eeqs
The ratios $\epsilonmax/\epsilonmin$ and $\mumax/\mumin$ do not appear in the second arguments of the maxima on the right-hand side of \eqref{eq:thm:BoundSmooth} since coercivity implies that
\beqs%\label{eq:coercive_bounds}
\omega\N{\nabla p }_{L^2(B_R)}\leq \frac{1}{\epsilon_{\min}}\N{\bJ}_{L^2(B_R)} 
\quad\tand\quad
\omega\N{\nabla q }_{L^2(B_R)}\leq \frac{1}{\mu_{\min}}\N{\bK}_{L^2(B_R)}.
\eeqs
\ere

\subsection{Proof of Corollary \ref{cor:scat}}

Let 
$\chi(\bx):=\max\{0,\min\{1,\frac{R-|\bx|}{R-\Rscat}\}\}$ (i.e., $\chi$ is piecewise-linear in the radial direction); then 
$\N{\chi}_{L^\infty(B_R)}=1$ and 
$\N{\nabla\chi}_{L^\infty(B_R)} =\frac1{R-\Rscat}$. 
Furthermore, since the tangential components of $\nabla \chi$ are continuous on the boundary of the support of $\nabla \chi$ (which is a spherical shell), $\nabla\chi \in H(\curl ;\Rea^3)$. 
Let $\widetilde{\bE}:= \chi \bE^I + \bE^S = \bE^T -(1-\chi) \bE^I$, and similarly for $\widetilde{\bH}$. 
Since $\widetilde{\bE}= \bE^S$ and $\widetilde{\bH}= \bH^S$ for $|\bx|\ge R$, $\widetilde{\bE}, \widetilde{\bH}$ 
satisfy the Silver--M\"uller condition \eqref{eq:intro_SM}. Furthermore $\widetilde{\bE}, \widetilde{\bH}\in H_{\rm loc}(\curl;\Rea^3)$ and satisfy \eqref{eq:first_order}
 with 
$\bJ:=  \nabla\chi \times \bH^I $ and 
$\bK:= \nabla\chi \times \bE^I.$ 
Applying \eqref{eq:thm:BoundSmooth} to $\widetilde{\bE}$ and $\widetilde{\bH}$, we obtain 
\begin{align*}
&\epsilonconstant\big\|\widetilde{\bE} \big\|^2_{L^2(B_R;\epsilon)} + \muconstant\big\|\widetilde{\bH}\big\|^2_{L^2(B_R;\mu)} 
\\
&\hspace{1cm}\le \| \nabla\chi \|^2_{L^\infty(B_R)}\bigg(
2\max\left\{ 8R^2 \left(\frac{\epsilonmax\mumax}{\muconstant} + \frac{\epsilon_0 \mu_0}{\epsilonconstant}\right), \frac{\epsilonconstant}{\omega^2}\right\}
\N{ \bH^I }_{L^2(B_R\setminus B_{\Rscat};\epsilon^{-1})}^2
\\
&\hspace{4cm}+2\max\left\{ 8R^2 \left(\frac{\epsilonmax\mumax}{\epsilonconstant} + \frac{\epsilon_0 \mu_0}{\muconstant}\right), \frac{\muconstant}{\omega^2}\right\}
\N{ \bE^I}_{L^2(B_R\setminus B_{\Rscat};\mu^{-1})}^2\bigg),
\end{align*}
where we have used the fact that $\epsilon$ and $\mu$ are scalar-valued on $B_R\setminus B_{\Rscat}$ to ``pull out'' $\|\nabla \chi\|_{L^\infty(B_R)}$ from the weighted norms of $\bH^I$ and $\bE^I$ on the right-hand side.

The bound \eqref{eq:BoundScat} then follows using the inequality $\|\bE^T\|_{L^2(B_R;\epsilon)}^2\leq 2 \|\widetilde{\bE}\|^2_{L^2(B_R;\epsilon)}+ 2 \|(1-\chi)\bE^I\|^2_{L^2(B_R;\epsilon)}$, its analogue with $\bE$ replaced by $\bH$, and the bounds on $\chi$ above.

\begin{remark}\mythmname{Small contrast limit}
The bound \eqref{eq:BoundScat} is sharp in its $\omega$ dependence, but does not show that the scattered fields $\bE^S$ and $\bH^S$ vanish in the small-contrast limit, i.e.,\ for $\max\{\N{\epsilon-\epz\MI}_{L^\infty(\IR^3)},$ $\N{\mu-\muz\MI}_{L^\infty(\IR^3)}\}\to0$.
One can easily bound the scattered fields $\bE^S$ and $\bH^S$ observing that they solve the Maxwell problem \eqref{eq:first_order}--\eqref{eq:intro_SM} with $\bJ=\ri\wn(\epz\MI-\epsilon)\bE^I$ and $\bK=\ri\wn(\mu-\muz\MI)\bH^I$, however the resulting bound is suboptimal in its $\wn$ dependence.
We note that \cite[Cor.~3.1]{MS17} obtained a bound for the Helmholtz scattering problem  in this latter way; its wavenumber dependence can easily be improved by adapting the proof of Corollary~\ref{cor:scat} to the Helmholtz setting.
\end{remark}

\section{Results for the interior impedance problem}\label{sec:IIP}

\subsection{Statement of the results}

\begin{theorem}\mythmname{Bound on impedance problem with certain $C^1$ coefficients}
\label{thm:ImpedanceSmooth}
Suppose that $\Omega$ is a bounded Lipschitz open set that is
star-shaped with respect to a ball centred at the origin with radius $\rho \Rimp $, where $\Rimp :=\sup\{|\bx|:\bx\in {\domainimp}\}$.
Suppose that $\epsilon,\mu \in C^1(\domainimp;\SPD)$ satisfy \eqref{eq:limits} and that $\vartheta\in L^\infty(\deOimp)$ is uniformly positive.
Let 
\beq\label{eq:Mvartheta}
M_\vartheta:=(3+\rho^{-1})\N{\max\{\vartheta^{-1}|\epsilon|,\vartheta|\mu|\}}_{L^\infty(\deOimp)},
\eeq
where $|\epsilon|$ and $|\mu|$ denote the point values of the matrix norms (induced by the Euclidean vector norm).
Then, given $\bJ,\bK\in L^2(\domainimp)$ and $\bg\in L^2(\deOimp)$, the solution 
$\bE,\bH\in H_{\imp}(\curl;\Omega)$ of \eqref{eq:ImpedanceFirst}--\eqref{eq:ImpedanceBC} exists, is unique, and satisfies
\begin{align}\nonumber
&\epsilonconstant\N{\bE}^2_{L^2(\domainimp;\epsilon)} + \muconstant\N{\bH}^2_{L^2(\domainimp;\mu)}
+2\Rimp \big\|\abs{\epsilon}^{1/2}\bE_T\|_{L^2(\deOimp)}^2
+2\Rimp \big\|\abs{\mu}^{1/2}\bH_T\|_{L^2(\deOimp)}^2\\
\nonumber
&\quad\le 
2\max \left\{8\Rimp^2
\bigg(\frac{\epsilonmax\mumax}{\muconstant}+\frac{M_\vartheta^2}{\epsilonconstant}\bigg), \frac{\epsilonconstant}{ \omega^2}\right\}
\N{\bJ}_{L^2(\domainimp;\epsilon^{-1})}^2
\\
&\qquad+
2\max \left\{
8\Rimp^2\bigg(\frac{\epsilonmax\mumax}{\epsilonconstant}
+\frac{M_\vartheta^2}{\muconstant}\bigg)
, \frac{\muconstant}{\omega^2}\right\}
\N{\bK}_{L^2(\domainimp;\mu^{-1})}^2
+4 \Rimp  M_\vartheta \big\|\vartheta^{-1/2}\bg\big\|_{L^2(\deOimp)}^2,
\label{eq:ImpedanceBound}
\end{align}
where the instances of $8$ on the right-hand side reduce to $4$ if $\bK=\bzero$.
\end{theorem}

In \eqref{eq:ImpedanceBound}, the norm on $\deOimp$ on the left-hand side can be replaced by the more natural, smaller quantity
$\N{\epsilon^{1/2}\bE_T}_{L^2(\deOimp)}^2=\int_{\deO}\epsilon\bE_T\cdot\conj\bE_T$
involving the square-root matrix of $\epsilon$ (which is $\SPD$).%

\begin{theorem}\mythmname{Bound on interior impedance problem with rough coefficients}\label{thm:ImpedanceRough}
Suppose that $\Omega$ is a bounded Lipschitz open set that is
star-shaped with respect to a ball centred at the origin with radius $\rho \Rimp $, where $\Rimp :=\sup\{|\bx|,\bx\in {\domainimp}\}$.
Let $\epsilon,\mu\in L^\infty(\Omega,\SPD)$ be such that 
\eqref{eq:RoughCoefficients} holds,
where $\Pi_{\epsilon},\Pi_\mu \in L^\infty(\domainimp,\SPD)$
are monotonically non-decreasing, in the sense of quadratic forms, in the radial direction.
Let $\vartheta\in L^\infty(\deOimp)$ be uniformly positive and define $M_\vartheta$ by \eqref{eq:Mvartheta}.
Then, given $\bJ,\bK\in L^2(\domainimp)$ and $\bg\in L^2_T(\deOimp)$, the solution 
$\bE,\bH\in H_{\imp}(\curl;\Omega)$ of \eqref{eq:ImpedanceFirst}--\eqref{eq:ImpedanceBC} exists, is unique, and satisfies
\eqref{eq:ImpedanceBound} with 
$\abs{\epsilon}$ and $\abs{\mu}$ replaced by $\epsilonmin$ and $\mumin$, respectively.
\end{theorem}

\subsection{Proof of Theorem \ref{thm:ImpedanceSmooth}}

\begin{lemma}\mythmname{Bound for $C^1$ coefficients and right-hand sides in $H(\div;\Omega)$}
\label{lem:ImpedanceFirst}
Suppose that $\Omega$ is a bounded Lipschitz open set that is
star-shaped with respect to a ball centred at the origin with radius $\rho \Rimp $, where $\Rimp :=\sup\{|\bx|,\bx\in {\domainimp}\}$.
Suppose that $\epsilon,\mu \in C^1(\domainimp;\SPD)$ satisfy \eqref{eq:limits} and $\vartheta\in L^\infty(\deOimp)$ is uniformly positive.
Let $M_\vartheta$ be defined by \eqref{eq:Mvartheta}.
Then, given $\bJ,\bK\in H(\div;\Omega)$,  the solution 
$\bE,\bH\in H_{\imp}(\curl;\Omega)$ of
\eqref{eq:ImpedanceFirst}--\eqref{eq:ImpedanceBC} 
exists, is unique, and satisfies the bound
\begin{align}\nonumber
&\epsilonconstant\N{\bE}^2_{L^2(\domainimp;\epsilon)} + \muconstant\N{\bH}^2_{L^2(\domainimp;\mu)}
+2\Rimp \big\|\abs{\epsilon}^{1/2}\bE_T\|_{L^2(\deOimp)}^2
+2\Rimp \big\|\abs{\mu}^{1/2}\bH_T\|_{L^2(\deOimp)}^2\\
&\quad\nonumber
\le 
\frac{4 \Rimp^2}{\epsilonconstant} \bigg(
\N{\bK}_{L^2(\domainimp;\epsilon)} 
+M_{\vartheta}\N{\bJ}_{L^2(\domainimp;\epsilon^{-1})}
+\frac {1}{\wn \epsilonmin^{1/2}} \N{\nabla\cdot\bJ}_{L^2(\domainimp)}
\bigg)^2\\
&
\quad+
\frac{4\Rimp^2}{\muconstant} \bigg(
\N{\bJ}_{L^2(\domainimp;\mu)} 
+M_{\vartheta}\N{\bK}_{L^2(\domainimp;\mu^{-1})}
+\frac {1}{\wn \mumin^{1/2}} \N{\nabla\cdot\bK}_{L^2(\domainimp)}
\bigg)^2
+2 \Rimp  M_\vartheta \big\|\vartheta^{-1/2}\bg\big\|_{L^2(\deOimp)}^2.
\label{eq:thm:Impedance}
\end{align}
\end{lemma}

Theorem \ref{thm:ImpedanceSmooth} follows by combining Lemma \ref{lem:ImpedanceFirst} and 
the following analogue of Lemma \ref{lem:HdivtoL2}.

\begin{lemma}\mythmname{Bound with right-hand side in $H(\div;\Omega)$ implies bound with right-hand side in $L^2(\Omega)$}
\label{lem:HdivtoL2imp}
Suppose that $\epsilon,\mu$ are such that the solution of  \eqref{eq:ImpedanceFirst}--\eqref{eq:ImpedanceBC} with $\bJ,\bK \in H(\div;\Omega)$ and $\bg \in L^2_T(\partial \Omega)$ exists, is unique, and satisfies the bound \eqref{eq:thm:Impedance}.
Then the solution to \eqref{eq:ImpedanceFirst}--\eqref{eq:ImpedanceBC} with $\BJ, \bK\in \BL^2(\Omega)$ 
and $\bg \in L^2_T(\partial \Omega)$
exists, is unique, and satisfies \eqref{eq:ImpedanceBound}.
\end{lemma}

\bpf
This is almost identical to the proof of Lemma \ref{lem:HdivtoL2}.
The main new ingredient is the fact that $(\nabla p)_T=(\nabla q)_T=\bzero$, which holds since $(\nabla p)_T$
and $(\nabla q)_T$
 are the surface gradients of the traces of $p$ and $q$, respectively, and these traces are zero, since $p,q \in H^1_0(\Omega)$. This fact shows that 
$\bE-\nabla p$ and $\bH-\nabla q$ satisfy the impedance boundary condition \eqref{eq:ImpedanceBC}
(since $\bH\times \hn =\bH_T\times \hn$), and also means that the correct norms on $\partial \Omega$ appear in the analogue of \eqref{eq:analogue1}.
\epf

The rest of this subsection is therefore dedicated to the proof of Lemma \ref{lem:ImpedanceFirst}.

\begin{lemma}\label{lem:star}
(i) If the domain $D\subset\IR^3$ is Lipschitz, then it is star-shaped with respect to $\bx_0$ if and only if $(\bx-\bx_0)\cdot\hn(\bx)\geq 0$ for all $\bx \in\partial D$ for which the outward-pointing unit normal $\hn(\bx)$ is defined.

\noindent (ii) 
$D$ is star-shaped with respect to the ball $B_{a}$ if and only if it is Lipschitz and
$\bx \cdot \hn(\bx) \geq {a}$ for all  $\bx \in \partial D$ for which $\hn(\bx)$ is defined; 
\end{lemma}

\bpf 
See \cite[Lemma 5.4.1]{AndreaPhD} 
or \cite[Lemma 3.1]{MaxwellPDE}. 
\epf

In the following two lemmas, $|\epsilon|$ and $|\mu|$ denote the point values of the matrix norms (induced by the Euclidean vector norm).

\ble
\label{lem:PreImpedanceIneq}
Let $\domainimp$ be a bounded Lipschitz open set with outward-pointing unit normal vector $\hn$. 
Let $\Rimp :=\sup\{|\bx|,\bx\in {\domainimp}\}$ and assume that ${\domainimp}$ is star-shaped with respect to a ball of radius $\rho \Rimp $, i.e.\ $\bx\cdot\hn\ge\rho \Rimp >0$ a.e.\ on $\deO$.
Let $\bE \in V({\domainimp},\epsilon)$, $\bH \in V({\domainimp}, \mu)$ and $\beta\in C(\conj{\domainimp})$.
Then 
\begin{align}
\label{eq:preimpineq}
I_{\deO}&:=\int_{\deO}
(\bx\cdot\hn)
\Big(\epsilon\bE_T\cdot\conj\bE_T - \epsilon\bE_N\cdot\conj\bE_N 
+ \mu\bH_T\cdot\conj\bH_T - \mu\bH_N\cdot\conj\bH_N\Big)
\\
&\qquad\qquad-2\Re\Big\{(\bE_T\cdot\bx_T)(\epsilon\conj\bE\cdot\hn)+(\bH_T\cdot\bx_T)(\mu\conj\bH\cdot\hn)+\beta\bE\times\conj\bH\cdot\hn\Big\}
\nonumber
\\
&\!\le \Rimp (2+\rho^{-1})
\bigg(\N{\abs{\epsilon}^{1/2}\bE_T}_{L^2(\deOimp)}^2+\N{\abs{\mu}^{1/2}\bH_T}_{L^2(\deOimp)}^2\bigg) 
-\int_{\deO} 2\Re\Big\{\beta\conj\bH\times\hn\cdot\bE\Big\}.
\nonumber
\end{align}
\ele
\bpf
We follow similar steps to \cite[Lemma 5.4.3]{AndreaPhD}, but note that this earlier result did not use $\beta$ and considered constant scalar $\epsilon,\mu$. 
First note that, for all $\bv\in\IC^3$, $\alpha\in\SPD$, and $\hn\in\IR^3$ with $|\hn|=1$,
\begin{align*}
-2\Re\{(\bv_T\cdot\bx_T)(\alpha\conj\bv\cdot\hn)\}
&=-2\Re\Big\{(\bv_T\cdot\bx_T)(\alpha\conj\bv_T\cdot\hn)+
(\bv_T\cdot\bx_T)\frac{\alpha\conj\bv_N\cdot\bv_N}{\bv_N\cdot\hn}\Big\}\\
&\le2|\bx_T| \abs{\alpha}|\bv_T|^2 +2|\bx_T| |\bv_T|\sqrt{\abs{\alpha}} \sqrt{\alpha\conj\bv_N\cdot\bv_N}\\
&\le2|\bx_T| \abs{\alpha}|\bv_T|^2 + \frac{|\bx_T|^2}{\bx\cdot\hn}\abs{\alpha}|\bv_T|^2 + (\bx\cdot\hn)(\alpha\conj\bv_N\cdot\bv_N)\\
&=\Big(2|\bx_T|+\frac{|\bx_T|^2}{\bx\cdot\hn}\Big)\abs{\alpha}|\bv_T|^2+ (\bx\cdot\hn)(\alpha\conj\bv_N\cdot\bv_N)
\end{align*}
where we used the Young inequality with weight $(\bx\cdot\hn)$ and where $\abs{\alpha}=\sup_{\bzero\ne\bw\in\IC^3}|\alpha\bw|/|\bw|$.
This allows to control the integral on $\deO$ with the tangential traces of the fields only:
\begin{align*}
I_{\deO} &\le\int_{\deO}
(\bx\cdot\hn)\Big(\epsilon\bE_T\cdot\conj\bE_T - \epsilon\bE_N\cdot\conj\bE_N 
+ \mu\bH_T\cdot\conj\bH_T - \mu\bH_N\cdot\conj\bH_N\Big)\\
&\qquad + \Big(2|\bx_T|+\frac{|\bx_T|^2}{\bx\cdot\hn}\Big)\big(\abs{\epsilon}|\bE_T|^2+\abs{\mu}|\bH_T|^2\big) 
+(\bx\cdot\hn)\big(\epsilon\bE_N\cdot\conj\bE_N+\mu\bH_N\cdot\conj\bH_N\big)\\
&\qquad-2\Re\Big\{\beta\bE\times\conj\bH\cdot\hn\Big\}
\\
&\le\int_{\deO}
\Big(\bx\cdot\hn+2|\bx_T|+\frac{|\bx_T|^2}{\bx\cdot\hn}\Big)
\big(\abs{\epsilon}|\bE_T|^2+\abs{\mu}|\bH_T|^2\big) 
-2\Re\Big\{\beta\bE\times\conj\bH\cdot\hn\Big\}.
\end{align*}
The identity 
$(\bx\cdot\hn)+|\bx_T|^2/(\bx\cdot\hn)=(|\bx_N|^2+|\bx_T|^2)/\bx_N=|\bx|^2/\bx_N\le \Rimp /\rho$
then implies the assertion.
\epf

This next lemma plays the analogous role for the interior impedance problem as Lemma \ref{lem:infinity} does for the transmission problem.
The Helmholtz analogue of this result appeared implicitly in \cite[\S8.1]{MEL95}, \cite{CUF06}, and \cite{HET07}, and explicitly as \cite[Lemma A.11]{GrPeSp:18}. 

\ble\mythmname{Inequality on $\deOimp$ used to deal with the impedance boundary condition}
\label{lem:ImpedanceIneq}
Let $\domainimp$ be a bounded Lipschitz open set with outward-pointing unit normal vector $\hn$. 
Let $\Rimp :=\sup\{|\bx|,\bx\in {\domainimp}\}$ and assume that ${\domainimp}$ is star-shaped with respect to a ball of radius $\rho \Rimp $, i.e.\ $\bx\cdot\hn\ge\rho \Rimp >0$ a.e.\ on $\deO$. 
Let $\bE \in V({\domainimp},\epsilon)$, $\bH \in V({\domainimp}, \mu)$ satisfy the impedance boundary condition 
\beq\label{eq:imp1}
\bH\times \hn-\vartheta\bE_T=\bg
\eeq
for 
$\bg\in L^2_T(\deOimp)$ and for a 
uniformly positive $\vartheta\in L^\infty(\deO)$.
If $\beta\in C^1(\overline{\domainimp})$ satisfies
\beq
\label{eq:BetaConditionMat}
\beta\ge \Rimp (3+\rho^{-1})\max\Big\{\frac{|\epsilon|}\vartheta, \vartheta{|\mu|}\Big\}
\eeq
pointwise almost everywhere on $\deO$,
then, with $I_{\deO}$ as in \eqref{eq:preimpineq},
\begin{align}
\label{eq:impineq}
&I_{\deO}
\leq\N{\Big(\frac\beta\vartheta\Big)^{1/2}\bg}^2_{L^2(\deOimp)}
-\Rimp \N{\abs{\epsilon}^{1/2}\bE_T}_{L^2(\deOimp)}^2
-\Rimp \N{\abs{\mu}^{1/2}\bH_T}_{L^2(\deOimp)}^2.
\end{align}
\ele
Note that when $\bE$ and $\bH$ are in $C^1(\overline{\domainimp})$, the left-hand side $I_{\deO}$ of \eqref{eq:preimpineq} and \eqref{eq:impineq} is $- \int_{\deO} \bQ_\beta \cdot \hn$.
\bpf
The impedance boundary condition \eqref{eq:imp1} implies that
\begin{align*}
2\Re\big\{ \conj\bH\times\hn\cdot\bE\big\} 
&= \Re\big\{\conj \bH\times\hn \cdot \bE_T + \conj \bH\times\hn \cdot \bE_T\big\}\\
&= \Re\big\{\vartheta|\bE_T|^2+\conj\bg\cdot\bE_T 
+ \vartheta^{-1}|\bH_T|^2-\vartheta^{-1}\conj\bH\times\hn\cdot\bg\big\}\\
&=\vartheta|\bE_T|^2+ \vartheta^{-1}|\bH_T|^2
+\Re\big\{\conj\bg\cdot(\bE_T - \vartheta^{-1}\bH\times\hn)\big\}\\
&=\vartheta|\bE_T|^2+ \vartheta^{-1}|\bH_T|^2-\vartheta^{-1}|\bg|^2,
\end{align*}
so \eqref{eq:preimpineq} becomes
\begin{align*}
I_{\deO}\leq &\;
\Rimp (2+\rho^{-1})\bigg(\big\|\abs{\epsilon}^{1/2}\bE_T\big\|_{L^2(\deOimp)}^2+\big\|\abs{\mu}^{1/2}\bH_T\big\|_{L^2(\deOimp)}^2\bigg) \\
&\qquad-\big\|(\beta\vartheta)^{1/2}\bE_T\big\|^2_{{L^2(\deOimp)}}
-\big\|(\beta\vartheta^{-1})^{1/2}\bH_T\big\|^2_{{L^2(\deOimp)}}
+\big\|(\beta\vartheta^{-1})^{1/2}\bg\big\|^2_{L^2(\deOimp)}.
\end{align*}
Then condition \eqref{eq:BetaConditionMat} implies that  
\beqs
I_{\deO}-\big\|(\beta\vartheta^{-1})^{1/2}\bg\big\|^2_{L^2(\deOimp)}
\leq -\Rimp \big\|\abs{\epsilon}^{1/2}\bE_T\big\|_{L^2(\deOimp)}^2-\Rimp \big\|\abs{\mu}^{1/2}\bH_T\big\|_{L^2(\deOimp)}^2,
\eeqs
and the result \eqref{eq:impineq} follows.
\epf

\begin{proof}[Proof of Lemma~\ref{lem:ImpedanceFirst}]
Using Part (i) of Lemma \ref{lem:IntegratedM}, we apply the integrated Morawetz identity \eqref{eq:intmor} with $D=\domainimp$ and use the Maxwell equations \eqref{eq:ImpedanceFirst} (which imply $\nabla\cdot[\epsilon\bE]=\nabla\cdot[\mu\bH]=0$) to arrive at 
\begin{align*}
&\int_{\Omega}\big(\epsilon+(\bx\cdot\nabla)\epsilon\big) \bE\cdot\conj\bE+\big(\mu+(\bx\cdot\nabla)\mu\big)\bH\cdot\conj\bH
\\
&=2\int_{\Omega}\!\!\!\Re\Big\{\bK\cdot(\epsilon\conj\bE\times\bx+\beta\conj \bH)
+\bJ\cdot(\mu\conj\bH\times\bx-\beta\conj\bE)+\nabla\beta\cdot\bE\times\conj\bH
+(\conj\bE\cdot\bx)\frac{\nabla\cdot\bJ}{\ri\wn}
-(\conj\bH\cdot\bx)\frac{\nabla\cdot\bK}{\ri\wn}
\Big\}
\\
&\qquad-\int_{\partial \domainimp}
\left( \epsilon\bE_N\cdot\conj\bE_N - \epsilon\bE_T\cdot\conj\bE_T
+ \mu\bH_N\cdot\conj\bH_N - \mu\bH_T\cdot\conj\bH_T\right) (\bx\cdot\hn) 
\\
&\qquad-\int_{\partial \domainimp}
2\Re\bigg\{ (\bE_T\cdot\bx_T) (\epsilon\conj\bE\cdot\hn) +(\bH_T\cdot\bx_T) (\mu\conj\bH\cdot\hn)
+ \beta \bE_T\cdot \conj\bH_T\times \hn\bigg\};
\end{align*}
(compare to \eqref{eq:EuanTemp1new}).
If
$\beta=\Rimp  M_\vartheta$ (which satisfies the condition \eqref{eq:BetaConditionMat}) then 
Lemma~\ref{lem:ImpedanceIneq} implies that the integrals over $\deOimp$ 
 are bounded above by 
$$
\Xi:= \Rimp  M_\vartheta \big\|\vartheta^{-1/2}\bg\big\|_{L^2(\deOimp)}^2
-\Rimp \big\|\abs{\epsilon}^{1/2}\bE_T\big\|_{L^2(\deOimp)}^2
-\Rimp \big\|\abs{\mu}^{1/2}\bH_T\big\|_{L^2(\deOimp)}^2.
$$
Proceeding as in the proof of Lemma \ref{lem:BoundSmooth_old} and using Lemma \ref{lem:E1} with this value of $\Xi$, one obtains the assertion \eqref{eq:thm:Impedance}.
\end{proof}

\bre\mythmname{Generalised impedance boundary conditions}
The bound \eqref{eq:impineq} holds if the impedance condition \eqref{eq:imp1} is replaced by a more general condition
$\bH\times\hn-\widehat\calT\bE_T=\bg$,
where $\widehat\calT:L^2_T({\deOimp})\to L^2_T({\deOimp})$ is a continuous, bijective, linear operator satisfying 
$\int_{\deOimp}\Re\{\widehat\calT\bv\cdot\conj\bv\}\ge c|\bv|^2_{\deOimp}$,
$\int_{\deOimp}\Re\{\widehat\calT^{-1}\bv\cdot\conj\bv\}\ge c|\bv|^2_{\deOimp}$
and
$\int_{\deOimp}\Re\{\bv\cdot\widehat\calT^{-1}\bw\}=\int_{\deOimp}\Re\{\widehat\calT^{-1}\bv\cdot\bw\}$
for all $\bv,\bw\in L^2_T({\deOimp})$ and some $c>0$.
The bound \eqref{eq:impineq} then holds for sufficiently large $\beta$, proportional to $c^{-1}$.
Under these assumptions, the bound \eqref{eq:thm:Impedance} on the solution of the (generalised) impedance problem also holds.
\ere

\subsection{Proof of Theorem \ref{thm:ImpedanceRough}}

By \cite[\S4.5]{MON03} it is sufficient to show that the bound \eqref{eq:thm:BoundSmooth} holds under the assumption that the solution of \eqref{eq:ImpedanceFirst} exists.
Let
\beqs
\Tnorm{(\bE,\bH)}^2_{\epsilon,\mu,\mathrm{imp}}:= \epsilonconstant\N{\bE}^2_{L^2(\domainimp;\epsilon)} + \muconstant\N{\bH}^2_{L^2(\domainimp;\mu)}
+2\Rimp \epsilonmin\big\|\bE_T\|_{L^2(\deOimp)}^2
+2\Rimp \mumin\big\|\bH_T\|_{L^2(\deOimp)}^2.
\eeqs
By the density of $C^\infty(\overline{\domainimp})^3$ in $H\imp(\curl;\Omega)$ (see, e.g., \cite[Theorem 3.54]{MON03}), 
given $\eta>0$, $\epsilon, \mu, \bE,$ and $\bH$, 
there exists $\bE_\eta, \bH_\eta\in C^\infty(\overline{\Omega})^3$ such that 
\begin{align*}
&\N{\nabla\times (\bE- \bE_\eta)}_{L^2(\Omega)}^2+ 
\N{\nabla\times (\bH - \bH_\eta)}_{L^2(\Omega)}^2
+\Big|
\Tnorm{(\bE,\bH)}_{\epsilon,\mu,\mathrm{imp}}^2 -\Tnorm{(\bE_\eta,\bH_\eta)}_{\epsilon,\mu,\mathrm{imp}}^2
\Big|
\leq \eta/2.
\end{align*}
Exactly as in the proof of Theorem \ref{thm:BoundRough}, $\bE_\eta$ and $\bH_\eta$ satisfy the PDEs \eqref{eq:PDEperturb1} and \eqref{eq:PDEperturb2}, and furthermore
\beqs
(\bH_\eta)\times \bn - \vartheta (\bE_\eta)_T = 
(\bH_\eta)_T\times \bn - \vartheta (\bE_\eta)_T = 
\bg-(\bH-\bH_\eta)_T\times \bn + \vartheta (\bE-\bE_\eta)_T.
\eeqs
By Lemma \ref{lem:mollifier}, $\epsilon_\delta$ and $\mu_\delta$ satisfy the conditions of Theorem \ref{thm:ImpedanceSmooth}, and thus, by the bound \eqref{eq:ImpedanceBound},
\begin{align}\nonumber
\Tnorm{(\bE_\eta,\bH_\eta)}_{\epsilon_\delta,\mu_\delta,\mathrm{imp}}^2 
\leq C_1 &\N{ \bJ + \ri \wn (\epsilon_\delta-\epsilon)\bE_\eta + \ri\omega\epsilon(\bE_\eta-\bE)
+\nabla\times(\bH_\eta-\bH)}^2_{L^2(\Omega;\epsilon^{-1})}\\
&+C_2\N{\bK+\ri \wn (\mu_\delta-\mu)\bH_\eta + \ri\omega\mu(\bH_\eta-\bH)
-\nabla\times(\bE_\eta-\bE)}^2_{L^2(\Omega;\mu^{-1})} \nonumber\\
&+C_3\big\|\vartheta^{-1/2}\big(\bg-(\bH-\bH_\eta)_T\times \bn + \vartheta (\bE-\bE_\eta)_T\big)\big\|^2_{L^2(\partial \Omega)},
\label{eq:newapprox2imp}
\end{align}
where we have abbreviated the constants on the right-hand side of \eqref{eq:thm:BoundSmooth} to $C_1,C_2$, and $C_3$ to keep the notation concise.
As in the proof of Theorem \ref{thm:BoundRough},
the crucial point is that the $C_1, C_2, C_3$ corresponding to $\epsilon_\delta$ and $\mu_\delta$ can be taken to be the $C_1, C_2, C_3$ corresponding to $\epsilon$ and $\mu$ by \eqref{eq:limits2} and \eqref{eq:GrowthCoeff2}.

We now claim that the approximation properties \eqref{eq:approx} and \eqref{eq:newapprox2imp} imply that 
given $\epsilon, \mu$, $\bE, \bH$ (and associated $\bJ, \bK, \bg$), $\wn>0$, and $\zeta>0$ one can choose $\eta= \eta(\zeta)>0$ and $\delta=\delta(\zeta,\eta)>0$  such that 
the difference between the right-hand side of  \eqref{eq:newapprox2} and $C_1\|\bJ\|^2_{L^2(\Omega;\epsilon^{-1})} + C_2\|\bK\|^2_{L^2(\Omega;\mu^{-1})} + C_3 \|\vartheta^{-1/2}\bg\|^2_{L^2(\partial \Omega)}$
is $\leq \zeta/4$.
The proof of this claim is almost exactly the same as the proof of the analogous claim in the proof of Theorem \ref{thm:BoundRough}, except now the first step of choosing $\eta$ depending on $\zeta$ also involves the terms on $\partial\Omega$ arising from $(\bH-\bH_\eta)_T$ and $(\bE-\bE_\eta)_T$ as well as the terms in $\Omega$ in \eqref{eq:Friday2} and \eqref{eq:Friday3}.
Once this claim is established, without loss of generality, we can further assume that 
$\eta\leq \zeta$ (also exactly as in the proof of Theorem \ref{thm:BoundRough}).

The proof then proceeds exactly as in the proof of Theorem \ref{thm:BoundRough}, in particular using \eqref{eq:newapprox00} with $B_R$ replaced by $\Omega$. The end result is that
\beqs
\Tnorm{(\bE,\bH)}^2_{\epsilon,\mu,\mathrm{imp}} \leq C_1 \N{\bJ}^2_{L^2(\Omega;\epsilon)} + C_2\N{\bK}^2_{L^2(\Omega;\mu)} + C_3  \|\vartheta^{-1/2}\bg\|^2_{L^2(\partial \Omega)}+ \zeta,
\eeqs
and since $\zeta>0$ was arbitrary, the result follows.

\appendix

\section{Density of \texorpdfstring{$\DOmegabar^3$}{D} in the space \texorpdfstring{$V(D,\MA)$}{V}}\label{app:density}

Throughout this appendix $D\subset \Rea^d, d \geq 2$ is a bounded Lipschitz open set. 
We are mainly interested in the case $d=3$, but some results hold for general $d\geq 2$, and thus we specify the values of $d$ explicitly in the statements.
In this appendix only, we use $\gamma$ to denote the standard Dirichlet trace operator $\gamma:H^1(D)\to H^\half(\partial D)$.

The goal of this appendix is to prove the following density result.

\begin{theorem}\label{thm:Density}
Let $\MA\in C^1(\overline{D},\SPD)$ with $d=3$.
Then $\DOmegabar^{3}:=\{ U|_D : U \in C^\infty(\Rea^d)^3\}$ is dense in the space $V(D,\MA)$ defined by \eqref{eq:V}.
\end{theorem}

Let $V_{\rm scal}(D;\MA)$ (subscript ``$\rm{scal}$" for scalar) be defined by
\beq\label{eq:Vs}
V_{\rm scal}(D;\MA):= \Big\{ v\in H^1(D) : \nabla\cdot [\MA\nabla v] \in L^2(D), \partial_{n, \MA}v\in L^2(\partial D), \gamma v \in H^1(\partial D)\Big\},
\eeq
where $\MA \in C^{0,1}(\overline{D},\SPD)=W^{1,\infty}(D,\SPD)$ and $\partial_{n, \MA}v$ is the conormal derivative defined such that $\partial_{n, \MA}w= \hn \cdot  (\MA \nabla w)$ for $w\in H^2(D)$.

The following theorem of Ne\v{c}as says that either of the conditions on $\partial D$ in \eqref{eq:Vs} can be removed.

\begin{theorem}\mythmname{\!\!\cite[\S5.1.2, \S5.2.1]{Ne:67}, \cite[Theorem 4.24]{MCL00}}\label{thm:Necas}
Let $d\geq 2$ and $u\in H^1(D)$ with $\nabla\cdot [\MA\nabla u]\in L^2(D)$, where $\MA\in W^{1,\infty}(D, \SPD)$. Then

(i) if $\partial_{n,\MA} u \in L^2(\partial D)$ then $\gamma u \in H^1(\partial D)$, and

(ii) if $\gamma u \in H^1(\partial D)$ then $\partial_{n,\MA} u \in L^2(\partial D)$.
\end{theorem}

\ble\mythmname{Density for $V(D,\MA)$ follows from density for $V_{\rm scal}(D;\MA)$}\label{lem:ScalVecDens}
If $d=3$, $\MA\in W^{1,\infty}(D,\SPD)$ and  $\DOmegabar$ is dense in $V_{\rm scal}(D;\MA)$, then $\DOmegabar^3$ is dense in $V(D;\MA)$.
\ele

\bpf
Given $\bv\in V(D;\MA)$, by the regular decomposition lemma for functions in $H(\curl; D)$ (see, e.g., \cite[Lemma 2.4]{HiptmairActa}) there exist $\bz \in H^1(D)^3$ and $\phi \in H^1(D)$ such that $\bv= \bz + \nabla \phi$. 
Now
\beqs
\nabla\cdot [\MA \nabla\phi]= -\nabla\cdot [\MA \bz] + \nabla\cdot [\MA \bv],
\eeqs
which is in $L^2(D)$ (since $\MA\in W^{1,\infty}(D,\SPD)$ and $\bv \in V(D,\MA)$). Since $(\MA\bv)\cdot \hn\in L^2(\partial D)$, $\partial_{n,\MA}\phi \in L^2(\partial D)$, and thus the Ne\v{c}as result (Theorem \ref{thm:Necas}) implies that $\phi \in V_{\rm scal}(D;\MA)$. 
Define the norms on $V(D;\MA)$ and $V_{\rm scal}(D, \MA)$ by
\beqs
\N{\bv}_{V(D,\MA)}:=\N{\bv}_{H(\curl; D)}+\N{\nabla \cdot [\MA \bv]}_{L^2(D)}+ \N{\MA\bv\cdot\hn}_{L^2 (\partial D)}+ \N{\bv_T}_{L^2_T (\partial D)}
\eeqs
and 
\beq\label{eq:normVscal}
\N{v}_{V_{\rm scal}(D,\MA)}:=\N{v}_{H^1(D)}+\N{\nabla \cdot [\MA \nabla v]}_{L^2(D)}+ \N{\partial_{n,\MA}v}_{L^2 (\partial D)}+ \N{\gamma v}_{H^1(\partial D)}.
\eeq
These definitions and the boundedness of $\gamma :H^1(D)\to H^{1/2}(\partial D)$ imply that if $\bv = \bz + \nabla\phi$ then
\beqs
\N{\bv}_{V(D,\MA)} \leq \N{\phi}_{V_{\rm scal}(D,\MA)} +
 C\big(\N{A}_{W^{1,\infty}(D)} +1\big) \N{\bz}_{H^1(D)^3}
\eeqs
for some $C>0$ (independent of $\bv, \bz, \phi$, and $\MA$). Therefore, the density of $\DOmegabar$ in $V_{\rm scal}(D;\MA)$ and of $\DOmegabar^3$ in $H^1(D)^3$ implies the density of $\DOmegabar^3$ in $V(D,\MA)$.
\epf

The result of Theorem \ref{thm:Density} therefore follows if we can prove that $\DOmegabar$ is dense in $V_{\rm scal}(D;\MA)$ (Theorem \ref{thm:scalar_density} below). 
To do this, we first need to recall some facts about Lipschitz domains and nontangential approach sets.

The definition of a Lipschitz open set (see, e.g., \cite[Definition 3.28]{MCL00}, \cite[Definition A.2]{CGLS12})  implies that there exist finite families $\{ W_i\}$, $\{D_i\}$, $\{f_i\}$, and $\{M_i\}$, such that
\bit
\item[(i)] the family $\{W_i\subset \Rea^d\}$ is a finite open cover of $\partial D$,
\item[(ii)] the family $\{D_i \subset \Rea^d\}$ is such that $D\cap W_i = D_i\cap W_i$ for all $i$,
\item[(iii)] each $f_i: \Rea^{d-1}\rightarrow \Rea$ is Lipschitz continuous with Lipschitz constant $M_i\geq 0$, and
\item[(iv)] for each $i$ there exists a rigid motion (i.e.~a composition of a rotation and translation) $r_i : \Rea^d\rightarrow \Rea^d$ such that $r_i(D_i)= \{ (x',x_d)\in \Rea^{d-1}\times \Rea : x_d < f_i(x')\}$, the hypograph of~$f_i$.
\eit
Without loss of generality we take the rigid motions $\{r_i\}$ to be a family of pure rotations. Let $\{\chi_i\}$ be a partition of unity subordinate to $\{W_i\}$, i.e.
\beqs
\chi_i \in C_{\rm comp}^\infty(W_i), \quad \chi_i\geq 0, \quad\tand\quad \sum_{i} \chi_i(\bx)=1 \quad\tfa \bx\in \partial D.
\eeqs
For $\bx \in \supp \chi_i \cap \partial D$ and $\eps>0$, we define the segment %let
\beqs
\ell_{i,\eps}(\bx) := \big\{ \bx - t \hb_i; \,\,0 <t<\eps\big\},
\eeqs
where $\hb_i$ is the unit vector that satisfies $\hb_i = r_i^{-1}(\be_d)$; observe that 
$\ell_{i,\eps}(\bx)\subset D$ for all $\bx \in W_i\cap \partial D$ and sufficiently small $\eps>0$.

\begin{defin}\mythmname{Non-tangential approach set}
Given $\bx \in \partial D$, let 
\beqs
\Theta(\bx):= \Big\{ \by \in D : |\bx-\by| < \min\big\{c, C \dist(\bx,\partial D)\big\}\Big\}
\eeqs
for $C>1$ sufficiently large (depending on the Lipschitz character of $D$) 
and $c>0$ sufficiently small such that $\Theta(\bx) \subset D$ and $\overline{\Theta(\bx)}\cap \partial D = \{\bx\}$.
\end{defin}

Observe that this definition implies that there exists $\eps^*$ (independent of $i$) such that 
\beq\label{eq:lieps}
\ell_{i,\eps}(\bx) \subset \Theta(\bx) \quad \tfa \bx \in \supp \chi_i \cap \partial D \tand 0<\eps\leq \eps^*.
\eeq

\begin{defin}\mythmname{Non-tangential maximal function}
Given the family of non-tangential approach sets $\{\Theta(\bx): \bx \in \partial D\}$ and $u \in C(D)$, 
let
\beqs
u^*(\bx) := \sup_{\by \in \Theta(\bx)} \big|u(\by)\big|, \quad \bx \in \partial D.
\eeqs
\end{defin}

\begin{defin}\mythmname{Conormal derivative via non-tangential limit}
Given the family of non-tangential approach sets $\{\Theta(\bx): \bx \in \partial D\}$, $u \in C^1(D)$, and $\MA\in C(D,\SPD)$, let
\beq\label{eq:ntnormal}
\dudnA(\bx) := \hn(\bx) \cdot \left( \lim_{\by \rightarrow \bx, \by \in \Theta(\bx)} \MA(\by) \nabla u (\by)\right)
\eeq
for all $\bx \in \partial D$ for which $\hn$ is defined.
\end{defin}

\begin{theorem}\mythmname{Regularity of the solution of the Neumann problem \cite{MiTa:99, MiTa:01}}\label{thm:MiTareg}
For $d\geq 2$, given $\MA\in C^{0,1}(\overline{D},\SPD)$, $V\in L^\infty(D)$ with $V>0$ on $\overline{D}$, and $g\in L^2(\partial D)$, there exists a unique $u \in H^1(D)\cap C^{1,s}(D)$ for all $0<s<1$ such that
\begin{align*}
&\nabla\cdot[\MA \nabla u] - V u = 0 \,\, \tin D, \quad \dudnA = g \,\, \ton \partial D, \quad\tand\\
& \hspace{2cm} \N{(\MA\nabla u)^*}_{L^2(\partial D)} \leq C\N{g}_{L^2(\partial D)}
\end{align*}
for some $C>0$ (independent of $u$ and $g$).
\end{theorem}

\bpf[References for the proof]
When $\MA\in C^{1}(\overline{D},\SPD)$, this result, without the statement that $u\in H^1(D)$, is \cite[Theorem 6.1]{MiTa:99}, with \cite[\S2]{MiTa:01} then explaining how \cite[Theorem 6.1]{MiTa:99} also holds when $\MA\in C^{0,1}(\overline{D},\SPD)$. The fact that $u\in H^1(D)$ follows from the representation of $u$ as a single-layer potential \cite[Equation 6.2]{MiTa:99}, \cite[Equation 5.3]{MiTa:01}, and the mapping property of this potential in \cite[Proposition 7.9]{MiTa:99}.
\epf

The next lemma relates the conormal derivative defined by \eqref{eq:ntnormal} to the standard conormal derivative in $H^{-1/2}(\partial D)$ defined via Green's identity (denoted by $\partial_{n,\MA}u$).

\ble\label{lem:normal}
For $d\geq 2$, if $\MA \in C^{0,1}(\overline{D},\Sym)$ and $u\in H^1(D)\cap C^1(D)$ with $\nabla\cdot (\MA\nabla u)\in L^2(D)$, $(\MA \nabla u)^*\in L^2(\partial D)$, and $\partial u /\partial n_{\MA}\in L^2(\partial D)$, then $\partial u/\partial n_{\MA} = \partial_{n,\MA} u$.
\ele

\bpf
We follow the proof of the result when $\MA =\MI$ in \cite[Lemma A.10]{CGLS12}.
By the definition of $\partial_{n,\MA}u$ (see, e.g., \cite[Lemma 4.3]{MCL00}, \cite[Pages 280--281]{CGLS12}), it is sufficient to prove that 
\beq\label{eq:normal1}
\int_{\partial D} \dudnA \,\overline{\gamma v} = \int_D \MA \gu\cdot \overline{\gv} + \overline{v}\, \nabla\cdot [\MA \gu],
\eeq
for all $v\in H^1(D)$. Since $C^\infty(D)$ is dense in $H^1(D)$, it is sufficient to prove that \eqref{eq:normal1} holds for all $v\in C^\infty(D)$. 

Let $\{1,2,\ldots,N\}$ be a finite index set for $\{\chi_i\}$; i.e.~$\{\chi_i\} = \{\chi_i : 1,\ldots,N\}$. Let $\chi_0 \in C^\infty(\overline{D})$ be defined by
\beq\label{eq:normal2}
\chi_0(\bx) := 1- \sum_{i=1}^N \chi_i(\bx), \quad \bx \in \overline{D}, \quad \text{ so that } \sum_{i=0}^N\chi_i(\bx) =1 \quad \bx\in \overline{D}.
\eeq
We now claim that it is sufficient to prove that \eqref{eq:normal1} holds for all $v_i := \chi_i v$ with $v\in C^\infty(D)$. Indeed, assuming this result, we have
\begin{align*}
\int_{\partial D} \dudnA \overline{\gamma v} = \sum_{i=0}^N \int_{\partial D} \dudnA \overline{\gamma v_i} &= \sum_{i=0}^N \int_D \MA \gu \cdot \overline{\gv}_i + \overline{v}_i \nabla \cdot [\MA \gu], \\
&= 
 \sum_{i=0}^N \int_D\chi_i \Big( \MA \gu \cdot \overline{\gv} + \overline{v} \nabla \cdot [\MA \gu]\Big) \,\, \left(\text{since}  \sum_{i=0}^N \nabla\chi_i=0 \right)\\
& =\int_D  \MA \gu \cdot \overline{\gv} + \overline{v} \nabla \cdot [\MA \gu],
\end{align*}
i.e.~\eqref{eq:normal1} holds.
For $i=0$, $v_0 = \chi_0 v\in H^1_0(D)$ and 
\beqs
\int_D  \MA \gu \cdot \overline{\gv_0} + \overline{v}_0 \nabla \cdot [\MA \gu]=0
\eeqs
by Green's identity for $u\in H^1(D)$ with $\nabla\cdot[\MA \nabla u]\in L^2(D)$ and $v_0\in H^1(D)$ \cite[Lemma 4.3]{MCL00}, \cite[Equation A.29]{CGLS12} so that \eqref{eq:normal1} holds with $v=v_0$.

For $i=1,\ldots, N$, let $D_{i,t}:= D_i - t \hb_i$ for $t>0$,
where $\hb_i:= r_i^{-1}(\be_d)$ as above;
i.e.~$D_{i,t}$ is the (rotated) hypograph $D_i$ shifted down by $t$. 
Let $G_i:= \supp \chi_i \cap D_{i,t}$. By interior regularity of the operator $(\nabla\cdot [\MA \nabla])^{-1}$ (see, e.g., \cite[Theorem 4.16]{MCL00}) $u\in H^2(G)$ and then, by Green's identity for $u\in H^2$ and $v\in H^1$ \cite[Lemma 4.1]{MCL00}, \cite[Equation A.26]{CGLS12},
\beqs
\int_{\partial D_{i,t}}\! \hn(\by) \cdot \gamma \big( \MA(\by) \gu(\by)\big)  \overline{\gamma v_i}(\by) \rd s(\by) 
= \int_{D_{i,t}}\! \Big(\MA(\by) \gu(\by) \cdot \overline{\gv_i}(\by) + \overline{v_i}(\by) \nabla\cdot [\MA(\by) \gu(\by)] \Big) \rd \by.
\eeqs
Using the change of variable $\by = \bx - t\hb_i$ for $\by \in \partial D_{i,t}$ and $\bx\in \partial D_i$, observing that $\hn(\by)= \hn(\bx)$ in this case, 
and recalling that $\MA\in C^{0,1}(\overline{D},\SPD)$ and $\gu\in C(D)$ (so we can omit the trace operator from $\MA \gu$), 
we have
\beq\label{eq:normal3}
\int_{\partial D_i} \hn(\bx) \cdot \Big( \big( \MA \gu\big) \, \overline{\gamma v_i}\Big)(\bx- t\hb_i) \,\rd s(\bx) 
= \int_{D_i} \Big(\MA \gu \cdot \overline{\gv_i} + \overline{v_i} \nabla\cdot (\MA \gu) \Big)(\bx - t\hb_i) \rd \bx.
\eeq
With $\phi_t(\bx):= \phi(\bx- t \hb_i)$, observe that $\phi_t(\bx)- \phi(\bx)\tendo$ if $\phi$ is continuous at $\bx$, and thus (by density of continuous functions in $L^2$) $\phi_t- \phi\tendo$ in $L^2$. Therefore, as $t\tendo$, the right-hand side of \eqref{eq:normal3} tends to 
\beqs
\int_{D_i} \MA \gu \cdot \overline{\gv_i} + \overline{v_i} \nabla\cdot [\MA \gu] = \int_D \MA \gu \cdot \overline{\gv_i} + \overline{v_i} \nabla\cdot [\MA \gu],
\eeqs
where for the last equality we have used the facts that $\supp v_i \subset \supp \chi_i \subset W_i$ and $D\cap W_i = D_i\cap W_i$.

For the left-hand side of \eqref{eq:normal3}, observe that 
\beqs
\Big| \hn(\bx) \cdot \Big(\big( \MA \gu\big) \, \overline{\gamma v_i}\Big)(\bx- t\hb_i) \Big| \leq C (\MA \gu)^*( \bx - t \hb_i) 
\eeqs
for some $C>0$ (dependent on $v_i$). Therefore, taking the limit as $t \tendo$ in the left-hand side of \eqref{eq:normal3} using the dominated convergence theorem (noting that $L^2(\supp \chi_i \cap \partial D_i) \subset L^1(\supp \chi_i \cap \partial D_i)$) along with the inclusion \eqref{eq:lieps} and the definition of $\partial u/\partial n_{\MA}$ \eqref{eq:ntnormal}, we find that the left-hand side of \eqref{eq:normal3} tends to $\int_{\partial D_i} \partial u /\partial n_{\MA} \,\overline{\gamma v_i}$, and the proof is complete.
\epf

Combining Theorem \ref{thm:MiTareg} and Lemma \ref{lem:normal} gives the following corollary.

\begin{cor}\label{cor:reg}
For $d\geq 2$, given $\MA\in C^{0,1}(\overline{D},\SPD)$, $V\in L^\infty(D)$ with $V>0$ on $\overline{D}$, and $g\in L^2(\partial D)$,
let $u\in H^1(D)$ be 
 the solution of the Neumann problem
\beqs
\nabla\cdot[\MA \nabla u] - V u = 0 \,\, \tin D, \quad \partial_{n,\MA}u= g \,\, \ton \partial D.
\eeqs
Then
\beq\label{eq:reg1}
u \in C^{1,s}(D) \,\text{ for all } 0<s<1, \quad (\MA \nabla u)^* \in L^2(\partial D), \quad \tand\quad \partial_{n,\MA} u = \dudnA. 
\eeq
\end{cor}

We are now in a position to prove that $\DOmegabar$ is dense in $V_{\rm scal}(D;\MA)$, from which Theorem~\ref{thm:Density} follows.
The idea of the proof is to decompose a general element $w\in V_{\rm scal}(D;\MA)$ as the sum of a term $v$ that is the restriction of the solution of a non-homogeneous PDE in $\IR^d$ (which enjoys interior elliptic regularity) and a term $u$ that is the solution of a homogeneous PDE in $D$ (which enjoys the regularity provided by Corollary~\ref{cor:reg}).

\begin{theorem}\label{thm:scalar_density}
For $d\geq 2$, if $\MA \in C^1(\overline{D},\SPD)$ then $\DOmegabar$ is dense in $V_{\rm scal}(D;\MA)$.
\end{theorem}

\bpf
By Part (i) of Theorem \ref{thm:Necas}, we can omit the $H^1(\partial D)$ term from the definition of the norm on $V_{\rm scal}(D;\MA)$ \eqref{eq:normVscal}, and we do for the rest of this proof.

Given $w \in V_{\rm scal}(D;\MA)$, choose $v\in \cD^{*}(\Rea^d)$ such that $\nabla\cdot[\MA \gv] - v = \nabla\cdot [\MA \nabla w]- w \in L^2(\Rea^d)$. Then, by interior elliptic regularity, $v \in H^2_{\rm loc}(\Rea^d)$ so $v\in H^2(D)$ and $\partial_{n,\MA}v = \hn \cdot \gamma (\MA\gv) \in L^2(\partial D)$.
Let $u:= w-v$; then $u\in H^1(D)$ with $\nabla\cdot[\MA \nabla u] - u=0$ and $\partial_{n,\MA}u\in L^2(\partial D)$. By Corollary~\ref{cor:reg}, $u$ then satisfies the regularity conditions in \eqref{eq:reg1}.

Since $\DOmegabar$ is dense in $H^2(D)$ and there exists a $C>0$ such that, for any $\varphi \in H^2(D)$, 
\beq\label{eq:H2}
\N{\varphi}_{V_{\rm scal}(D;\MA)} \leq C\big( \N{A}_{W^{1,\infty}(D)}+1 \big)\N{\varphi}_{H^2(D)}
\eeq
(where $\|\cdot\|_{V_{\rm scal}(D;\MA)}$ is defined by \eqref{eq:normVscal} without the $H^1(\partial D)$ term), 
 $v$ can be approximated by smooth functions in $V_{\rm scal}(D;\MA)$. Thus
it is sufficient to prove that, given $\eps>0$, there exists a $\phi \in \DOmegabar$ such that $\|u-\phi\|_{V_{\rm scal}(D;\MA)}<\eps$.
Furthermore, by interior regularity, $u\in H^2_{\rm loc}(D)$ and therefore, given $\eps>0$, there exists $\phi_0\in\DOmegabar$ such that $\|\chi_0 u - \phi_0\|_{H^2(D)}< \eps$ where $\chi_0$ is defined by \eqref{eq:normal2}. 
Therefore, using the partition of unity $\{\chi_i\}_{i=0}^N$ in \eqref{eq:normal2}, it is sufficient to prove that, given $\eps>0$, there exists $\phi_i \in \DOmegabar$ such that $\|\chi_i u -\phi_i\|_{V_{\rm scal}(D;\MA)} < \eps$.

Let $G_i:= D_i \cap \supp\chi_i$ and define $u_t$ 
\beq\label{eq:shift}
u_t(\bx):= u(\bx - t\hb_i) \quad \tfor 0<t\leq \eps^*,
\eeq
where $\eps^*$ is as in \eqref{eq:lieps}. By interior regularity $u_t \in H^2(G_i)$. Using again the fact that $\DOmegabar$ is dense in $H^2$, and the fact that the $H^2$ norm controls the $V_{\rm scal}(D;\MA)$ norm (by \eqref{eq:H2}), by the triangle inequality, to obtain the result it is sufficient to prove that 
\beqs
\N{u-u_t}_{V_{\rm scal}(D;\MA)} \tendo \quad \tas t\tendo.
\eeqs
With $(\gu)_t$ and $(\nabla\cdot [\MA\gu])_t$ defined in an analogous way to \eqref{eq:shift}, we have $(\gu)_t = \gu_t$ and $(\nabla\cdot [\MA\gu])_t= \nabla \cdot [\MA_t \gu_t]$. Since translation is continuous in $L^2$, 
\beq\label{eq:translation}
\N{u-u_t}_{L^2(G_i)}, \,\, \N{\gu - \gu_t}_{L^2(G_i)}, \,\, \tand \,\, \N{\nabla \cdot [\MA \gu]-\nabla \cdot [\MA_t \gu_t]}_{L^2(G_i)} \tendo \,\,\tas t\tendo.
\eeq
Now, for some $C>0$,
\begin{align}\nonumber
&\N{\nabla\cdot[\MA \gu] - \nabla\cdot[\MA \gu_t]}_{L^2(G_i)} \\
&\hspace{2cm}\leq \N{\nabla\cdot[\MA \gu] - \nabla\cdot[\MA_t \gu_t]}_{L^2(G_i)} + \N{\nabla\cdot[(\MA-\MA_t) \gu_t] }_{L^2(G_i)},\nonumber\\
&\hspace{2cm}\leq \N{\nabla\cdot[\MA \gu] - \nabla\cdot[\MA_t \gu_t]}_{L^2(G_i)} +  C \N{\MA-\MA_t}_{W^{1,\infty}(G_i)} \N{u_t}_{H^2(G_i)},
\label{eq:C1needed}
\end{align}
which tends to zero as $t\tendo$ using the last limit in \eqref{eq:translation} and the fact that $\MA\in C^1(\overline{G_i},\SPD)$.

It only remains to prove that $\|\partial_{n,\MA} u - \partial_{n,\MA} u_t\|_{L^2(\partial D_i)}\tendo$ as $t\tendo$. Since $u_t \in H^2(G_i)$, for $\bx \in \partial D_i$,
\begin{align*}
\partial_{n,\MA} u_t(\bx) &= \hn(\bx) \cdot \gamma \big( \MA \gu_t\big)(\bx) = \hn(\bx) \cdot  \big( \MA(\bx) \gu(\bx-t \hb_i)\big),
\end{align*}
where we have dropped the trace operator since $\MA\in C^1(\overline{D},\SPD)$ and $\gu\in C(\overline{D})$ by \eqref{eq:reg1}. Therefore,
\begin{align*}
\partial_{n,\MA} & u_t(\bx)= \hn(\bx) \cdot \Big[ \MA(\bx- t \hb_i) \nabla u (\bx- t \hb_i) + \big( \MA(\bx)- \MA(\bx - t \hb_i)\big) \gu(\bx - t \hb_i)\Big],\\
&= \hn(\bx) \cdot \Big[ \MA(\bx- t \hb_i) \nabla u (\bx- t \hb_i) + \big( \MA(\bx)- \MA(\bx - t \hb_i)\big) 
\MA^{-1}(\bx - t \hb_i)\MA(\bx - t \hb_i)\gu(\bx - t \hb_i)\Big],\\
& \xrightarrow{t\to0} \dudnA(\bx)\, +\, 0 \quad 
\text{by \eqref{eq:ntnormal} and the memberships $\MA\in C^1(\overline{D},\SPD)$ and $(\MA \gu)^*\in L^2(\partial D)$},\\
& = \partial_{n,A}u(\bx) \qquad \text{ by the last property in \eqref{eq:reg1}},
\end{align*}
and the proof is complete.
\epf

We make two remarks.
(i) The only step in the argument leading to Theorem~\ref{thm:Density} 
that fails if we assume only $\sfA\in W^{1,\infty}(D,\SPD)$ (as opposed to $\sfA\in C^1(\conj D,\SPD)$) is the use of $\N{\MA-\MA_t}_{W^{1,\infty}}\xrightarrow{t\to0}0$ in \eqref{eq:C1needed}.
(ii)
\cite[Lemma A.1]{MOS12} tried to prove density of $\DOmegabar$ in $V_{\rm scal}(D,\MI)$, i.e.\ a special case of Theorem~\ref{thm:scalar_density}, using the same method of proof as in Theorem \ref{thm:Density}. However, this proof was wrong; in particular, the condition $(\gu)^*\in L^2(\partial D)$ from \eqref{eq:reg1} (coming when $\MA=\MI$ originally from \cite[Theorem 2]{JeKe:81}) was not used.

\section*{Acknowledgements}

We thank Giovanni S.\ Alberti (University of Genoa), Ralf Hiptmair (ETH Z\"urich), Steven Johnson (MIT), 
Dirk Pauly (Universit\"at Duisburg-Essen), Luca Rondi (Universit\`a di Pavia)  and Michael Taylor (University of North Carolina at Chapel Hill) 
for useful discussions.

AM acknowledges support from GNCS--INDAM, from PRIN project ``NA\_FROM-PDEs'' and from MIUR through the ``Dipartimenti di Eccellenza'' Programme (2018--2022)--Dept.\ of Mathematics, University of Pavia.
EAS acknowledges support from EPSRC grant EP/R005591/1.

\paragraph{Competing-interests statement.} The authors declare no competing interests.

\footnotesize{
\bibliographystyle{siam}
\bibliography{references_andrea2020} 

\begin{thebibliography}{10}

\bibitem{AC16}
{\sc G.~S. Alberti and Y.~Capdeboscq}, {\em Lectures on elliptic methods for
  hybrid inverse problems}, vol.~25 of Cours Sp\'{e}cialis\'{e}s [Specialized
  Courses], Soci\'{e}t\'{e} Math\'{e}matique de France, Paris, 2018.

\bibitem{RoGe:06}
{\sc A.~Alonso~Rodriguez and L.~Gerardo-Giorda}, {\em {New nonoverlapping
  domain decomposition methods for the harmonic Maxwell system}}, SIAM Journal
  on Scientific Computing, 28 (2006), pp.~102--122.

\bibitem{AnBl:15}
{\sc L.~Andersson and P.~Blue}, {\em {Uniform energy bound and asymptotics for
  the Maxwell field on a slowly rotating Kerr black hole exterior}}, Journal of
  Hyperbolic Differential Equations, 12 (2015), pp.~689--743.

\bibitem{Ba:19}
{\sc T.~Baden-Riess}, {\em Phd thesis: Existence, uniqueness \& explicit bounds
  for scattering by rough surfaces}, arXiv preprint arXiv:1904.04011,  (2019).

\bibitem{BCT12}
{\sc J.~M. Ball, Y.~Capdeboscq, and B.~Tsering-Xiao}, {\em On uniqueness for
  time harmonic anisotropic {M}axwell's equations with piecewise regular
  coefficients}, Math. Models Methods Appl. Sci., 22 (2012), pp.~1250036, 11.

\bibitem{BaChGo:17}
{\sc H.~Barucq, T.~Chaumont-Frelet, and C.~Gout}, {\em Stability analysis of
  heterogeneous {H}elmholtz problems and finite element solution based on
  propagation media approximation}, Math. Comp., 86 (2017), pp.~2129--2157.

\bibitem{BeDe:97}
{\sc J.-D. Benamou and B.~Despr{\'e}s}, {\em {A domain decomposition method for
  the Helmholtz equation and related optimal control problems}},
  J.~Comp.~Phys., 136 (1997), pp.~68--82.

\bibitem{Bl:73}
{\sc C.~O. Bloom}, {\em {E}stimates for {S}olutions of {R}educed {H}yperbolic
  {E}quations of the {S}econd {O}rder with a {L}arge {P}arameter}, {J}ournal of
  {M}athematical {A}nalysis and {A}pplications, 44 (1973), pp.~310--332.

\bibitem{BlKa:77}
{\sc C.~O. Bloom and N.~D. Kazarinoff}, {\em {A priori bounds for solutions of
  the Dirichlet problem for $[\Delta+\lambda^2 n (x)] u= f (x, \lambda)$ on an
  exterior domain}}, Journal of Differential Equations, 24 (1977),
  pp.~437--465.

\bibitem{Bl:08}
{\sc P.~Blue}, {\em {Decay of the Maxwell field on the Schwarzschild
  manifold}}, Journal of Hyperbolic Differential Equations, 5 (2008),
  pp.~807--856.

\bibitem{BoDoGrSpTo:19}
{\sc M.~Bonazzoli, V.~Dolean, I.~G. Graham, E.~A. Spence, and P.-H. Tournier},
  {\em {Domain decomposition preconditioning for the high-frequency
  time-harmonic Maxwell equations with absorption}}, Mathematics of
  Computation, 88 (2019), pp.~2559--2604.

\bibitem{BoDoKyPe:20}
{\sc N.~Bootland, V.~Dolean, A.~Kyriakis, and J.~Pestana}, {\em {Analysis of
  parallel Schwarz algorithms for time-harmonic problems using block Toeplitz
  matrices}}, arXiv preprint arXiv:2006.08801,  (2020).

\bibitem{Cap12}
{\sc Y.~Capdeboscq}, {\em On the scattered field generated by a ball
  inhomogeneity of constant index}, Asymptot. Anal., 77 (2012), pp.~197--246.

\bibitem{CaPoVo:99}
{\sc F.~Cardoso, G.~Popov, and G.~Vodev}, {\em Distribution of resonances and
  local energy decay in the transmission problem {II}}, Mathematical Research
  Letters, 6 (1999), pp.~377--396.

\bibitem{CGLS12}
{\sc S.~N. Chandler-Wilde, I.~G. Graham, S.~Langdon, and E.~A. Spence}, {\em
  Numerical-asymptotic boundary integral methods in high-frequency acoustic
  scattering}, Acta Numer., 21 (2012), pp.~89--305.

\bibitem{CWM08}
{\sc S.~N. Chandler-Wilde and P.~Monk}, {\em Wave-number-explicit bounds in
  time-harmonic scattering}, SIAM J. Math. Anal., 39 (2008), pp.~1428--1455.

\bibitem{ChGrLaTa:21}
{\sc T.~Chaumont-Frelet, M.~Grote, S.~Lanteri, and J.~Tang}, {\em {A
  controllability method for Maxwell's equations}}, SIAM Journal on Scientific
  Computing, 44 (2022), pp.~A3700--A3727.

\bibitem{CoD98}
{\sc M.~Costabel and M.~Dauge}, {\em Un r\'esultat de densit\'e pour les
  \'equations de {M}axwell r\'egularis\'ees dans un domaine lipschitzien}, C.
  R. Acad. Sci. Paris S\'er. I Math., 327 (1998), pp.~849--854.

\bibitem{CUF06}
{\sc P.~Cummings and X.~Feng}, {\em Sharp regularity coefficient estimates for
  complex-valued acoustic and elastic {H}elmholtz equations}, Math. Models
  Methods Appl. Sci., 16 (2006), pp.~139--160.

\bibitem{De:91}
{\sc B.~Despr\'{e}s}, {\em M\'{e}thodes de D\'{e}composition de Domaine pour
  les Probl\'{e}mes de Propagation d'Ondes en R\'{e}gime Harmonique}, PhD
  thesis, Universit\'{e} Dauphine, Paris IX, 1991.

\bibitem{DoGaGe:09}
{\sc V.~Dolean, M.~J. Gander, and L.~Gerardo-Giorda}, {\em {Optimized Schwarz
  methods for Maxwell's equations}}, SIAM Journal on Scientific Computing, 31
  (2009), pp.~2193--2213.

\bibitem{DoLaPe:08}
{\sc V.~Dolean, S.~Lanteri, and R.~Perrussel}, {\em {Optimized Schwarz
  algorithms for solving time-harmonic Maxwell's equations discretized by a
  discontinuous Galerkin method}}, IEEE Transactions on Magnetics, 44 (2008),
  pp.~954--957.

\bibitem{DyZw:19}
{\sc S.~Dyatlov and M.~Zworski}, {\em Mathematical theory of scattering
  resonances}, vol.~200, American Mathematical Society, 2019.

\bibitem{EvGa92}
{\sc L.~C. Evans and R.~F. Gariepy}, {\em Measure theory and fine properties of
  functions}, Studies in Advanced Mathematics, CRC Press, Boca Raton, FL, 1992.

\bibitem{FeWu:14}
{\sc X.~Feng and H.~Wu}, {\em {An absolutely stable discontinuous Galerkin
  method for the indefinite time-harmonic Maxwell equations with large wave
  number}}, SIAM Journal on Numerical Analysis, 52 (2014), pp.~2356--2380.

\bibitem{GaMe:12}
{\sc G.~Gatica and S.~Meddahi}, {\em {Finite element analysis of a time
  harmonic Maxwell problem with an impedance boundary condition}}, IMA Journal
  of Numerical Analysis, 32 (2012), pp.~534--552.

\bibitem{GoGrSp:21}
{\sc S.~Gong, I.~G. Graham, and E.~A. Spence}, {\em {Domain decomposition
  preconditioners for high-order discretizations of the heterogeneous Helmholtz
  equation}}, IMA Journal of Numerical Analysis, 41 (2021), pp.~2139--2185.

\bibitem{GrPeSp:18}
{\sc I.~G. Graham, O.~R. Pembery, and E.~A. Spence}, {\em The {H}elmholtz
  equation in heterogeneous media: a priori bounds, well-posedness, and
  resonances}, J. Differential Equations, 266 (2019), pp.~2869--2923.

\bibitem{HaLe:11}
{\sc H.~Haddar and A.~Lechleiter}, {\em Electromagnetic wave scattering from
  rough penetrable layers}, SIAM Journal on Mathematical Analysis, 43 (2011),
  pp.~2418--2443.

\bibitem{HET07}
{\sc U.~Hetmaniuk}, {\em Stability estimates for a class of {H}elmholtz
  problems}, Commun. Math. Sci., 5 (2007), pp.~665--678.

\bibitem{HiptmairActa}
{\sc R.~Hiptmair}, {\em Finite elements in computational electromagnetism},
  Acta Numer., 11 (2002), pp.~237--339.

\bibitem{MaxwellPDE}
{\sc R.~Hiptmair, A.~Moiola, and I.~Perugia}, {\em Stability results for the
  time-harmonic {M}axwell equations with impedance boundary conditions}, Math.
  Models Methods Appl. Sci., 21 (2011), pp.~2263--2287.

\bibitem{HiMoPe:13}
\leavevmode\vrule height 2pt depth -1.6pt width 23pt, {\em {Error analysis of
  Trefftz-discontinuous Galerkin methods for the time-harmonic Maxwell
  equations}}, Mathematics of Computation, 82 (2013), pp.~247--268.

\bibitem{HiScScSc:18}
{\sc R.~Hiptmair, L.~Scarabosio, C.~Schillings, and C.~Schwab}, {\em Large
  deformation shape uncertainty quantification in acoustic scattering},
  Advances in Computational Mathematics, 44 (2018), pp.~1475--1518.

\bibitem{HiScSp:22}
{\sc R.~Hiptmair, C.~Schwab, and E.~A. Spence}, {\em {Frequency-Explicit Shape
  Uncertainty Quantification for Acoustic Scattering}}, in preparation,
  (2023).

\bibitem{JeKe:81}
{\sc D.~S. Jerison and C.~E. Kenig}, {\em The {N}eumann problem on {L}ipschitz
  domains}, Bull. Amer. Math. Soc. (N.S.), 4 (1981), pp.~203--207.

\bibitem{JeKe:95}
\leavevmode\vrule height 2pt depth -1.6pt width 23pt, {\em The inhomogeneous
  {D}irichlet problem in {L}ipschitz domains}, J. Funct. Anal., 130 (1995),
  pp.~161--219.

\bibitem{Ka:90}
{\sc B.~Kapitonov}, {\em {On exponential decay as $t\to\infty$ of solutions of
  an exterior boundary value problem for the Maxwell system}}, Mathematics of
  the USSR-Sbornik, 66 (1990), p.~475.

\bibitem{Ka:94}
\leavevmode\vrule height 2pt depth -1.6pt width 23pt, {\em {Stabilization and
  exact boundary controllability for Maxwell's equations}}, SIAM Journal on
  Control and Optimization, 32 (1994), pp.~408--420.

\bibitem{LeNg:13}
{\sc A.~Lechleiter and D.-L. Nguyen}, {\em On uniqueness in electromagnetic
  scattering from biperiodic structures}, ESAIM: Mathematical Modelling and
  Numerical Analysis, 47 (2013), pp.~1167--1184.

\bibitem{LRX16}
{\sc H.~Liu, L.~Rondi, and J.~Xiao}, {\em Mosco convergence for {$H(\rm curl)$}
  spaces, higher integrability for {M}axwell's equations, and stability in
  direct and inverse {EM} scattering problems}, J. Eur. Math. Soc. (JEMS), 21
  (2019), pp.~2945--2993.

\bibitem{LuChQi:17}
{\sc P.~Lu, H.~Chen, and W.~Qiu}, {\em {An absolutely stable $hp$-HDG method
  for the time-harmonic Maxwell equations with high wave number}}, Mathematics
  of Computation, 86 (2017), pp.~1553--1577.

\bibitem{Ma:20}
{\sc S.~Ma}, {\em {Uniform energy bound and Morawetz estimate for extreme
  components of spin fields in the exterior of a slowly rotating Kerr black
  hole I: Maxwell field}}, Annales Henri Poincare, 21 (2020), pp.~815--863.

\bibitem{MCL00}
{\sc W.~McLean}, {\em Strongly elliptic systems and boundary integral
  equations}, Cambridge University Press, Cambridge, 2000.

\bibitem{MEL95}
{\sc J.~M. Melenk}, {\em On Generalized Finite Element Methods}, PhD thesis,
  Univ.\ of Maryland, 1995.

\bibitem{MeSa:22}
{\sc J.~M. Melenk and S.~A. Sauter}, {\em {Wavenumber-explicit $hp$-FEM
  analysis for Maxwell's equations with impedance boundary conditions}}, arXiv
  preprint arXiv:2201.02602,  (2022).

\bibitem{MeSj:78}
{\sc R.~B. Melrose and J.~Sj{\"o}strand}, {\em {Singularities of boundary value
  problems. I}}, Communications on Pure and Applied Mathematics, 31 (1978),
  pp.~593--617.

\bibitem{MeSj:82}
\leavevmode\vrule height 2pt depth -1.6pt width 23pt, {\em Singularities of
  boundary value problems. {II}}, Communications on Pure and Applied
  Mathematics, 35 (1982), pp.~129--168.

\bibitem{Mi:95}
{\sc M.~Mitrea}, {\em The method of layer potentials in electromagnetic
  scattering theory on nonsmooth domains}, Duke Mathematical Journal, 77
  (1995), pp.~111--133.

\bibitem{MiTa:99}
{\sc M.~Mitrea and M.~Taylor}, {\em {Boundary layer methods for Lipschitz
  domains in Riemannian manifolds}}, Journal of Functional Analysis, 163
  (1999), pp.~181--251.

\bibitem{MiTa:01}
\leavevmode\vrule height 2pt depth -1.6pt width 23pt, {\em {Potential theory on
  Lipschitz domains in Riemannian manifolds: $L^p$ Hardy, and H\"older space
  results}}, Communications in Analysis and Geometry, 9 (2001), pp.~369--421.

\bibitem{AndreaPhD}
{\sc A.~Moiola}, {\em Trefftz-discontinuous {G}alerkin methods for
  time-harmonic wave problems}, PhD thesis, Seminar for applied mathematics,
  ETH Z\"urich, 2011.
\newblock \newline Available at
  \url{http://e-collection.library.ethz.ch/view/eth:4515}.

\bibitem{MOS12}
{\sc A.~Moiola and E.~A. Spence}, {\em Is the {H}elmholtz equation really
  sign-indefinite?}, SIAM Rev., 56 (2014), pp.~274--312.

\bibitem{MS17}
\leavevmode\vrule height 2pt depth -1.6pt width 23pt, {\em Acoustic
  transmission problems: wavenumber-explicit bounds and resonance-free
  regions}, Math. Models Methods Appl. Sci., 29 (2019), pp.~317--354.

\bibitem{MON03}
{\sc P.~Monk}, {\em Finite element methods for {M}axwell's equations},
  Numerical Mathematics and Scientific Computation, Oxford University Press,
  2003.

\bibitem{Mo:74}
{\sc C.~S. Morawetz}, {\em {A decay theorem for Maxwell's equations}}, Russian
  Mathematical Surveys, 29 (1974), p.~242.

\bibitem{Mo:75}
\leavevmode\vrule height 2pt depth -1.6pt width 23pt, {\em {Decay for solutions
  of the exterior problem for the wave equation}}, Communications on Pure and
  Applied Mathematics, 28 (1975), pp.~229--264.

\bibitem{MoLu:68}
{\sc C.~S. Morawetz and D.~Ludwig}, {\em An inequality for the reduced wave
  operator and the justification of geometrical optics}, Comm. Pure Appl.
  Math., 21 (1968), pp.~187--203.

\bibitem{Ne:67}
{\sc J.~Ne{\v{c}}as}, {\em {Les m{\'e}thodes directes en th{\'e}orie des
  {\'e}quations elliptiques}}, Masson, 1967.

\bibitem{NgTr:19}
{\sc H.-M. Nguyen and L.~X. Tran}, {\em {Approximate cloaking for
  time-dependent Maxwell equations via transformation optics}}, SIAM Journal on
  Mathematical Analysis, 51 (2019), pp.~4142--4171.

\bibitem{NgWa:12}
{\sc T.~Nguyen and J.-N. Wang}, {\em {Quantitative uniqueness estimate for the
  Maxwell system with Lipschitz anisotropic media}}, Proceedings of the
  American Mathematical Society, 140 (2012), pp.~595--605.

\bibitem{NiPi:05}
{\sc S.~Nicaise and C.~Pignotti}, {\em {Internal stabilization of Maxwell's
  equations in heterogeneous media}}, Abstract and Applied Analysis,  (2005),
  pp.~791--811.

\bibitem{NiTo:20}
{\sc S.~Nicaise and J.~Tomezyk}, {\em {Convergence analysis of a $hp$-finite
  element approximation of the time-harmonic Maxwell equations with impedance
  boundary conditions in domains with an analytic boundary}}, Numerical Methods
  for Partial Differential Equations, 36 (2020), pp.~1868--1903.

\bibitem{PeSp:18}
{\sc O.~R. Pembery and E.~A. Spence}, {\em The {H}elmholtz equation in random
  media: well-posedness and a priori bounds}, SIAM/ASA J. Uncertain. Quantif.,
  8 (2020), pp.~58--87.

\bibitem{PeVe:99}
{\sc B.~Perthame and L.~Vega}, {\em Morrey--{C}ampanato estimates for
  {H}elmholtz equations}, Journal of Functional Analysis, 164 (1999),
  pp.~340--355.

\bibitem{PiWeWi01}
{\sc R.~Picard, N.~Weck, and K.-J. Witsch}, {\em Time-harmonic maxwell
  equations in the exterior of perfectly conducting, irregular obstacles},
  Analysis, 21 (2001), pp.~231--264.

\bibitem{PoVo:99}
{\sc G.~Popov and G.~Vodev}, {\em Resonances near the real axis for transparent
  obstacles}, Communications in Mathematical Physics, 207 (1999), pp.~411--438.

\bibitem{Ra:71}
{\sc J.~V. Ralston}, {\em Trapped rays in spherically symmetric media and poles
  of the scattering matrix}, Communications on Pure and Applied Mathematics, 24
  (1971), pp.~571--582.

\bibitem{SpChGrSm:11}
{\sc E.~A. Spence, S.~N. Chandler-Wilde, I.~G. Graham, and V.~P. Smyshlyaev},
  {\em A new frequency-uniform coercive boundary integral equation for acoustic
  scattering}, Communications on Pure and Applied Mathematics, 64 (2011),
  pp.~1384--1415.

\bibitem{Th:06}
{\sc M.~Thomas}, {\em Analysis of rough surface scattering problems}, PhD
  thesis, University of Reading, 2006.

\bibitem{Va:75}
{\sc B.~R. Vainberg}, {\em {On the short wave asymptotic behaviour of solutions
  of stationary problems and the asymptotic behaviour as $t\rightarrow \infty$
  of solutions of non-stationary problems}}, Russian Mathematical Surveys, 30
  (1975), pp.~1--58.

\bibitem{Ve:19}
{\sc B.~Verf{\"u}rth}, {\em {Heterogeneous Multiscale Method for the Maxwell
  equations with high contrast}}, ESAIM: Mathematical Modelling and Numerical
  Analysis, 53 (2019), pp.~35--61.

\bibitem{Ya:88}
{\sc K.~Yamamoto}, {\em {Singularities of solutions to the boundary value
  problems for elastic and Maxwell's equations}}, Japanese journal of
  mathematics. New series, 14 (1988), pp.~119--163.

\end{thebibliography}
\addcontentsline{toc}{section}{References}
}

\end{document}